\documentclass[12pt,oneside,reqno]{amsart}
\usepackage{soul}
\usepackage{amsmath}
\usepackage{amssymb}
\usepackage{color}
\usepackage[dvipsnames,x11names,svgnames]{xcolor}
\usepackage[colorlinks=true,linkcolor=NavyBlue,citecolor=DarkGreen, urlcolor=blue,pagebackref]{hyperref}
\usepackage{amsthm}
\usepackage{amscd}
\usepackage{geometry}
\usepackage{mathrsfs}
\usepackage{dsfont}
\usepackage{mathtools}
\usepackage{verbatim}
\usepackage{realhats}
\usepackage{tikz-cd}
\usepackage[sort,numbers]{natbib}

\newcommand{\kommentar}[1]{}
\newcommand{\mf}{\mathfrak}

\newcommand{\mc}{\mathcal}

\newcommand{\sumstar}{\sideset{}{^*}\sum}

\newcommand{\cM}{\mc{M}}
\newcommand{\cD}{\mc{D}}

\newcommand{\cZ}{\mc{Z}}
\newcommand{\cL}{\mc{L}}
\newcommand{\cR}{\mc{R}}
\newcommand{\cA}{\mc{A}}
\newcommand{\cB}{\mc{B}}
\newcommand{\scrM}{\mathscr{M}}

\newcommand{\scrI}{\mathscr{I}}
\newcommand{\fa}{\mf{a}}
\newcommand{\fb}{\mf{b}}
\newcommand{\fp}{\mf{p}}
\newcommand{\fP}{\mf{P}}

\newcommand{\1}{\mathds{1}}

\DeclareMathOperator{\tr}{Tr}
\DeclareMathOperator{\res}{res}

\renewcommand{\pmod}[1]{\hspace{0.05em}(\mathrm{mod}\,#1)}
\newcommand{\R}{\mathbb{R}}
\newcommand{\Z}{\mathbb{Z}}

\newcommand{\N}{\mathbb{N}}
\newcommand{\C}{\mathbb{C}}

\newcommand{\F}{\mathbb{F}}
\newcommand{\Fq}{\mathbb{F}_q}
\newcommand{\FqT}{\Fq[T]}
\newcommand{\BA}{\mathbb{A}}
\newcommand{\BB}{\mathbb{B}}

\newcommand{\ov}{\overline}

\DeclareMathOperator{\sgn}{sgn}
\DeclareMathOperator{\Kl}{Kl}

\newcommand{\geqs}{\geqslant}
\newcommand{\leqs}{\leqslant}

\newtheorem{theorem}{Theorem}[section]
\newtheorem{lemma}[theorem]{Lemma}
\newtheorem{corollary}[theorem]{Corollary}

\newtheorem{proposition}[theorem]{Proposition}

\theoremstyle{definition}

\theoremstyle{remark}
\newtheorem{remark}[theorem]{Remark}

\newcommand{\mcom}[1]{{\color{red}{Matilde: #1}} }
\newcommand{\acom}[1]{{\color{blue}{Alexandra: #1}} }

\usepackage{marginnote}

\reversemarginpar

\newcommand{\copsum}{\,\sideset{}{^*}\sum}

\begin{document}
 
\title{The shifted convolution problem in function fields}

\author{Alexandra Florea}
\address{Department of Mathematics, University of California Irvine, 340 Rowland Hall, Irvine CA 92697, USA}
\email{\href{mailto:floreaa@uci.edu}{floreaa@uci.edu}}
\thanks{AF was partially supported by NSF grant DMS-2101769 and NSF CAREER grant DMS-2339274.}
\author{Matilde Lal\'in}
\address{D\'epartement de math\'ematiques et de statistique, Universit\'e de Montr\'eal, CP 6128, succ. Centre-ville, Montreal, QC H3C 3J7, Canada}\email{\href{mailto:matilde.lalin@umontreal.ca}{matilde.lalin@umontreal.ca}}
\thanks{ML was partially supported by the Natural Sciences and Engineering Research Council of Canada, RGPIN-2022-03651, and the Fonds de recherche du Qu\'ebec - Nature et technologies, Projet de recherche en \'equipe 300951}

\author{Amita Malik}
\address{Department of Mathematics, The Pennsylvania State University, McAllister Building, University Park, PA 16802, USA}
\email{\href{mailto:amita.malik@psu.edu}{amita.malik@psu.edu}}
\thanks{AM was partially supported by the Simons Foundation Grant TSM-00002309.}
\author{Anurag Sahay}
\address{Department of Mathematics, Purdue University, 150 N. University Street, West Lafayette, IN 47907, USA}
\email{\href{mailto:anuragsahay@purdue.edu}{anuragsahay@purdue.edu}}
\thanks{AS was partially supported by the AMS-Simons Travel Grant and Trevor Wooley's startup funding at Purdue University.}

\begin{abstract} We study the shifted convolution problem for the divisor function in function fields in the large degree limit, that is, the average value of $d(f) d(f+h)$ where $f$ runs over monic polynomials in $\FqT$ of a given degree, and $h$ is a given monic polynomial. We prove an asymptotic formula in the range $\deg(h) < (2-\epsilon)\deg(f)$. We also consider mixed correlations and self-correlations of $r_\chi = 1 \star \chi$, the convolution of $1$ with a Dirichlet character mod $\ell$, where $\ell$ is a monic irreducible polynomial, proving asymptotic formulae in various ranges. This includes the case of quadratic characters, which yields results about correlations of norm-counting functions of quadratic extensions of $\FqT$. A novel feature of our work is a Voronoi summation formula (equivalently, a functional equation for the Estermann function) in $\FqT$ which was not previously available.

\end{abstract}
\maketitle

\section{Introduction}

The shifted convolution of divisor functions
\begin{equation}\label{SCPinZ} \frac{1}{X}\sum_{n\leqs X} d(n) d(n+h), \end{equation}
with $h \in \N$ is a classical problem in the analytic theory of the integers, initially considered by Ingham \cite{ingham1927div} who proved an asymptotic formula for \eqref{SCPinZ} with log-savings. Soon thereafter, Estermann \cite{estermann} established a power saving asymptotic for \eqref{SCPinZ} by discovering and exploiting a connection to Kloosterman sums. Bounds on Kloosterman sums have since played a crucial role in this and related problems, as we shall see below. 

There has been sustained interest in an asymptotic formula for questions like \eqref{SCPinZ} with uniformity in the shift $h$ due to applications to moment problems, a connection due to Atkinson \cite{atkinson}. In this case, \eqref{SCPinZ} is connected to understanding the fourth moment of the zeta function,
\[ \int_T^{2T} |\zeta(\tfrac12+it)|^4 dt. \]
Indeed, Heath-Brown's breakthrough work \cite{heathbrownfourth} establishing a power-saving asymptotic for the fourth moment uses as an input the asymptotic (obtained via an application of Weil's bound for Kloosterman sums)
\[ \eqref{SCPinZ} = P_2(\log X;h) + O(X^{-1/6+\epsilon}) \text{ uniformly for } h \leqs X^{5/6},\]
where $P_2(t;h)$ is a polynomial of degree $2$ whose coefficients depend on $h$. 

Later, Deshouillers--Iwaniec \cite{deshouillersiwaniec-main, deshouillersiwaniec-div} introduced powerful methods from the spectral theory of automorphic forms to utilize cancellation between the Kloosterman sum to different moduli, thereby improving on the point-wise use of Weil's bound.

There are many subsequent works around this circle of ideas; the reader may consult \cite{motohashidiv,topad3,topadk} for a more comprehensive account of the extensive literature. We content ourselves with highlighting the work of Meurman \cite{meurman} who proved an asymptotic formula for \eqref{SCPinZ} in the widest known range $h \ll X^{2 - \epsilon}$. 

In this article, we investigate the analogue of \eqref{SCPinZ} for the function field $\FqT$, i.e.,
\begin{equation} \label{SCPinFF} \frac{1}{q^n} \sum_{f \in \cM_n} d(f) d(f+h), \end{equation}
where here and throughout $q$ is a fixed prime power, $h \in \cM$ is a monic polynomial in $\FqT$, $d(f)$ counts the monic divisors of $f$, and the sum $f \in \cM_n$ runs over the set of monic polynomials of degree $n$. 

The literature on \eqref{SCPinFF} is much sparser than \eqref{SCPinZ} but there are still some antecedents to our work. One is interested in evaluating \eqref{SCPinFF} when $q^n \to \infty$. Unlike in $\Z$, there are two natural regimes one can consider.  
 For example, we may fix $n$ and let $q \to \infty$. Andrade--Bary-Soroker--Rudnick \cite{ABSR} considered this problem (among others) and demonstrated that, provided $\deg(h) < n$, 
\begin{equation} \label{eq:ABSR} \bigg(\frac{1}{q^n} \sum_{f \in \cM_n} d(f)d(f+h)\bigg)\Bigg/ \bigg(\frac{1}{q^n} \sum_{f \in \cM_n} d(f) \bigg)^2 \longrightarrow 1 \text{ as } q \longrightarrow \infty. \end{equation}
For a fixed $q$ (or indeed, for $\Z$), the analogous quantity would instead depend on the prime factorization of $h$, demonstrating some local dependence between the multiplicative structures of $f$ and $f+h$. Thus, \eqref{eq:ABSR} may be interpreted as saying that the additive arithmetic structure disappears in the large field limit.

The complementary regime of $q$ fixed and $n \to \infty$ is more reminiscent of the integer setting and is the subject of our investigations. In this case, several independent authors (Conrey--Florea in unpublished work, Gorodetsky in his PhD thesis \cite[Lemma 3.3]{gorothesis}, Woo in her undergraduate thesis \cite{woothesis}, and Yiasemides in \cite{yiasemides}) established an exact formula for \eqref{SCPinFF} provided that $\deg(h) < n$. Ultimately, this relies on the fact that (the function field analogue of) the Estermann function is a polynomial. In particular, its functional equation is not used.

Our goal is to extend the range of $\deg(h)$ in which an asymptotic formula for \eqref{SCPinFF} is possible (eschewing, as will be necessary, the exact formula). Our main result is the following.
\begin{theorem}
    \label{dd}
    Suppose that $h$ is a monic polynomial with $\deg(h)=m$ and let $d =\max\{m,n\}$. Then
\begin{multline*}
        \sum_{f \in \mathcal{M}_n} d(f) d(f+h)  = q^n \sum_{\substack{g|h \\ \deg(g) \leqs [n/2]}} \frac{1}{|g|} \Big[ 4 \deg(g)^2 \Big(1-\frac{1}{q} \Big)\\
       -2\deg(g) \Big(2+\frac{2}{q}+ \Big( 1-\frac{1}{q}\Big)(n+d)\Big)+\Big[  2n+2 +(d-1) \Big( n \Big( 1-\frac{1}{q}\Big)+1+\frac{1}{q}\Big) \Big]\Big] + E,
    \end{multline*}
  where 
  $E=0$ if $m \leqs n+1$ or $m=n+2$ and $n$ is odd; otherwise
  \[ E \ll q^{m/2+\epsilon m}.\] 
\end{theorem}
Note that this gives an asymptotic formula provided that $\deg(h) <  (2-\epsilon)n$; thus, we recover exactly the same range as the aforementioned work of Meurman \cite{meurman}.

\begin{remark}Notice that when $m\leqs [n/2]$, we have that $d=n$ and the above expression simplifies as follows:
\begin{align*}
        \sum_{f \in \mathcal{M}_n} d(f) d(f+h) %& = q^n \sum_{\substack{g|h }} \frac{1}{|g|} \Big[ 4 \deg(g)^2 \Big(1-\frac{1}{q} \Big)-4\deg(g) \Big(\Big( 1-\frac{1}{q}\Big)n+\Big( 1+\frac{1}{q}\Big) \Big)\\
       %&+\Big[n^2\Big( 1-\frac{1}{q}\Big)+2n\Big( 1+\frac{1}{q}\Big)+\Big( 1-\frac{1}{q}\Big)\Big]\Big] + E\\
       &= q^n \sum_{\substack{g|h}} \frac{1}{|g|} \Big[\Big(n-2\deg(g)+1\Big)^2-\frac{1}{q}\Big(n-2\deg(g)-1\Big)^2\Big].
    \end{align*}
\end{remark}
We now turn to generalizations. In $\Z$, progress on problems like \eqref{SCPinZ} goes hand in hand with similar progress about the sums-of-two-squares function
\[ r(n) = \#\{(a,b) \in \Z^2 : a^2+b^2 = n \} = 4 \sum_{d \mid n} \chi_{-4}(d). \]
where $\chi_{-4}$ is the non-trivial character modulo $4$. 
An analogue of sums-of-two-squares for $\FqT$ was introduced in \cite{FFanalogLandau} (see also \cite{BaBSF,Bary-Soroker-Fehm}).  
For this, defining
\[ r(f) = \#\{ (a,b) \in \Fq[T]: a^2 + Tb^2 = f\}/\{\pm 1\}, \]
then $r$ can similarly be written as a convolution of $1$ with the quadratic character mod $T$ given by the Legendre symbol $\chi(f) = \Big(  \frac{f}{T} \Big)$, see \eqref{quadratic}. Regarding this function, we can prove the following (which follows as a corollary of Theorem~\ref{th1}). 

\begin{corollary}
    
\label{thm:sumsof2squares}    Suppose $q$ is odd, $h$ is a monic polynomial with $\deg(h)=m$ and let $d =\max\{m,n\}$. Let $\chi$ be the quadratic character (i.e., the Legendre symbol) modulo $T$. 
    Then
    \begin{align*}        \sum_{f \in \mathcal{M}_n} r(f) r(f+h) = & \frac{ q^{n}  L(1,\chi)^2}{1+q^{-1}}  \sum_{\substack{g|h \\ \deg(g) \leqs [d/2], T \nmid g}} \frac{ 1}{|g|} \\  &   -  \frac{q^{n-1} L(1, \chi)^2}{1+q^{-1}} \sum_{\substack{g| h \\ \deg(g) \leqs [(d-1)/2], T \nmid h/g}} \frac{ 1}{|g|}+E,\end{align*}
where
\[E \ll \begin{cases}
 q^{3n/4+\epsilon n} & \mbox{ if } m \leqs n, \\  
    q^{3d/4+2\deg(\ell)+\epsilon d} & \mbox{ otherwise. } 
    \end{cases}
\]
     \end{corollary} 
Finally, note that the character in the previous result is quadratic. However, our methods work equally well for general characters. We thus prove also the following theorems, from which the above corollary immediately follows. Let $r_{\chi}(f)$ be as in equation \eqref{def:rchi} and let $G(\chi)$ be as in equation \eqref{def:Gauss_sum_over_FqT}. We have the following.
\begin{theorem}
\label{th1}
    Suppose $h$ is a monic polynomial with $\deg(h)=m$ and let $d =\max\{m,n\}$. Let $\chi_1,\chi_2$ be two primitive characters modulo $\ell$, for some monic irreducible polynomial $\ell$ with $\deg(\ell)>0$. 
    Then
    \begin{multline*}
        \sum_{f \in \mathcal{M}_n} r_{\chi_1}(f) r_{\chi_2}(f+h) =  \frac{ q^n 
 L(1,\chi_1)L(1,\chi_2)}{L(2,\chi_1\chi_2)}  \sum_{\substack{g|h \\ \deg(g) \leqs [d/2]}} \frac{ \chi_1 \chi_2(g)}{|g|}
     \\
    +  \frac{q^n \ov{\chi_2}(-1) G(\chi_1) G(\chi_2)G(\ov{\chi_1 \chi_2})L(1,\ov \chi_1)L(1,\ov{\chi_2})}{|\ell^2|L(2,\ov{\chi_1 \chi_2})} \sum_{\substack{g| h \\ \deg(g) \leqs [(d-1)/2]}} \frac{ \chi_1 \chi_2(h/g)}{|g|}+E,
    \end{multline*}
    where 
\[E \ll \begin{cases}
 q^{3n/4+\epsilon n} & \mbox{ if } m \leqs n \text{ and } \deg(\ell) = 1, \\  
    q^{3d/4+2\deg(\ell)+\epsilon d} & \mbox{ otherwise. } 
    \end{cases}
\]
    
\end{theorem}

We note that the above is truly an asymptotic formula when $m<\frac{4n}{3}(1-\epsilon) - \frac{8 \deg(\ell)}{3}$. 

\begin{remark} \label{rem:norm-counting}When $\chi_1 = \chi_2 = \chi$ is a quadratic character and $q \equiv 1 \pmod 4$, $r_\chi$ is a norm-counting function for a quadratic extension of $\Fq[T]$ (see \S\ref{subsec:norm-counting} for the precise definition). This is analogous to the work of Muller \cite{muller} in the number field setting. 
\end{remark}

The following theorem provides an asymptotic formula in the case $\ell_1 \neq \ell_2$.
\begin{theorem}
    \label{th3}
    Suppose that $h$ is a monic polynomial with $\deg(h)=m$ and let $d =\max\{m,n\}$. Let $\chi_1, \chi_2$ be primitive characters modulo $\ell_1$ and $\ell_2$ respectively, where $\ell_1 \neq \ell_2$. Further assume that $\ell_1, \ell_2 \neq 1$. Then
    \begin{align*}
       &  \sum_{f \in \mathcal{M}_n} r_{\chi_1}(f) r_{\chi_2}(f+h) =   \frac{ q^n L(1,\chi_1)L(1,\chi_2)}{L(2,\chi_1\chi_2)} 
    \sum_{\substack{g|h \\ \deg(g) \leqs [d/2]}} \frac{\chi_1 \chi_2(g)}{|g|} \\
    &+ \frac{q^n \chi_2(\ell_1) L(1,\ov{\chi_1}) L(1,\chi_2)  }{|\ell_1|L(2,\ov{\chi}_1\chi_2)} \sum_{\substack{g|h \\ \deg(g) \leqs [d/2]-\deg(\ell_1)}} \frac{\chi_2(g) \chi_1(h/g)}{|g_4|} \\
    &+ \frac{q^n \chi_1(\ell_2) \ov{\chi_2}(-1) L(1,\chi_1) L(1,\ov{\chi_2}) }{|\ell_2|L(2,\chi_1 \ov{\chi_2})} \sum_{\substack{g|h \\ \deg(g) \leqs [(d-1)/2]}} \frac{\chi_1(g) \chi_2(h/g)}{|g|}  \\
    &+\frac{q^n \ov{\chi_1}(\ell_2) \ov{\chi_2}(-\ell_1) L(1,\ov \chi_1)L(1,\ov{\chi_2})}{|\ell_1 \ell_2 | L(2,\ov{\chi_1 \chi_2}) }     \sum_{\substack{g|h \\ \deg(g) \leqs [(d-1)/2] - \deg(\ell_1)}} \frac{ \chi_1 \chi_2(h/g)}{|g|} +E,
    \end{align*}
    where 
    \[
        E \ll 
           \begin{cases}q^{3n/4+\epsilon n} & \mbox{ if } m \leqs n, n \equiv 0 \pmod 2, \deg(\ell)=1,\\ 
        q^{3 d/4+ \deg(\ell_1)/2+ 3 \deg(\ell_2)/2+\epsilon d} & \mbox{ otherwise.}
    \end{cases}
    \]
\end{theorem}
We finally have the following two companion theorems.
\begin{theorem}
    \label{th4}
   Suppose that $h$ is a monic polynomial with $\deg(h)=m$ and let $d =\max\{m,n\}$. Let $\chi$ be a primitive character modulo $\ell$, where $\ell \neq 1$. Then 
   \begin{multline*}
       \sum_{f \in \mathcal{M}_n} d(f) r_{\chi}(f+h) = \frac{q^n L(1,\chi)}{L(2,\chi)} \sum_{\substack{g|h \\ \deg(g) \leqs [n/2]}} \frac{\chi(g)}{|g|} \Big[ n+1-2\deg(g)- \frac{2L'(2,\chi)}{(\log q) L(2,\chi)}\Big]\\
       +\frac{q^n \ov{\chi}(-1) L(1,\ov{\chi})}{|\ell|L(2,\ov{\chi})}  \sum_{\substack{g|h \\ \deg(g) \leqs [n/2]-\deg(\ell)}} \frac{\chi(h/g) }{|g|} \Big[n+1-2\deg(g \ell_2) - \frac{2L'(2,\ov{\chi})}{(\log q) L(2,\ov{\chi})} \Big]+ E, 
   \end{multline*}
   where 
   \[
   E \ll 
   \begin{cases}
       q^{n/2+\epsilon n}  & \mbox{ if }  m \leqs n, \deg(\ell)=1 \text{ or } m =n+1, n \text{ odd}, \deg(\ell)=1, \\
      q^{3d/4+3\deg(\ell)/2+\epsilon d} & \mbox{ otherwise.}
   \end{cases}
   \]
\end{theorem}

\begin{theorem}
    \label{th5}
  Suppose that $h$ is a monic polynomial with $\deg(h)=m$ and let $d =\max\{m,n\}$. Let $\chi$ is a primitive character modulo $\ell$, where $\ell \neq 1$. Then 
    \begin{multline*}
      \sum_{f \in \mathcal{M}_n} r_{\chi}(f) d(f+h) =   \frac{q^n L(1,\chi)}{L(2,\chi)}  \sum_{\substack{g|h \\ \deg(g) \leqs d/2}} \frac{\chi(g)}{|g|} \Big( d+1 - 2\deg(g)\Big)\\
    +\frac{q^n L(1,\ov \chi)}{|\ell_1|L(2,\ov \chi)} \sum_{\substack{g | h \\ \deg(g) \leqs d/2-\deg(\ell_1)}} \frac{ \chi(h/g)}{|g|} \Big[d+1-2\deg(g\ell_1) - \frac{2L'(2,\ov \chi)}{\log q L(2,\ov \chi)} \Big] +E,
    \end{multline*}
  where
    \[
   E \ll 
   \begin{cases}
       q^{n/2+\epsilon n} &\mbox{ if } m \leqs n, \deg(\ell)=1, \\ 
       q^{3d/4+\deg(\ell)/2+\epsilon d} & \mbox{otherwise.}
   \end{cases}
   \]
\end{theorem}

Let us now discuss the proof of Theorem~\ref{dd}. We follow the strategy of Estermann: we open the divisor function $d(f+h)$ with a divisor switching trick to restrict to divisors $g \mid f+h$ satisfying $\deg(g) \leqs d/2$. Thus, after swapping the order of summation, we need to understand
\begin{equation} \label{goal1} \sum_{\substack{f \in \cM_n \\ f \equiv -h \pmod g}} d(f) \text{ for } \deg(g) \leqs d/2. \end{equation}
One may now use additive characters to detect the congruence. As usual, the zero frequency gives a dominant contribution and the goal is to bound the nonzero frequencies 
\begin{equation} \label{goal2} \sum_{f \in \cM_n} d(f)e\bigg(\frac{af}{g}\bigg) \text{ for }\deg(a) < \deg(g) \leqs d/2 \text{ and } (a,g) = 1. \end{equation}
Here $e(\cdot)$ is the additive character defined in \S\ref{subsec:basics}. It is now natural to look for a Voronoi-type summation formula for \eqref{goal2}, which was hitherto unavailable. The main new contribution in this work is the formulation and derivation of such formulae (see \S\S\ref{sec:func}-\ref{sec:voronoi}). These results depend on a functional equation (Proposition~\ref{prop:Z_func}) for (the function field analogue of) the Hurwitz zeta function. This formula will transform \eqref{goal2} into
\[ \sum_{f \in \cM_{\mu}} d(f) e\bigg(\frac{\ov{a} f}{g}\bigg), \]
with $\mu \approx d - n$ and $\ov{a}$ is the inverse of $a$ modulo $g$. This is a beneficial move if $d < 2n$ which is why the error term in Theorem~\ref{dd} only dominates the main term in this case. The sum over $a$ yields a Kloosterman sum over $\FqT$ (see \eqref{def:S_over_FqT}) and ultimately our theorem relies on the Linnik--Selberg conjecture for these sums, which was proved by Cogdell and Piatetski-Shapiro \cite{cps}.

The ranges in Theorem~\ref{th3} are smaller than in Theorem~\ref{dd} for two reasons. First, the conductor calculation implicit in the dual length being $\mu \approx d-n$ is no longer valid, since $\ell$ will also contribute to the conductor. Second, and more importantly, we do not have access to Linnik--Selberg type cancellation when the Kloosterman sums are twisted by multiplicative characters. We must instead rely on a point-wise Weil bound \cite[Lemma A.13]{bagshaw}. 

\subsection*{Acknowledgments} This project was initiated during the Analytic Number Theory @ UCI'22 (ANTeater'22) summer school (August 1-4, 2022). ML, AM, and AS are grateful to the University of California, Irvine for its hospitality. This project was supported by a SQuaRE from the American Institute of Mathematics. The authors are grateful to AIM for the congenial work environment. The authors would also like to thank Gilyoung Cheong for many useful discussions at an earlier stage of the project. 

AF would like to thank Brian Conrey who provided the impetus for studying these questions, and for many helpful discussions and calculations related to this at a preliminary stage.

AS benefited from several conversations that took place at the AIM workshop on ``Delta symbols and the subconvexity problem" (October 16-20, 2023). AS would also like to thank Keshav Aggarwal, Ofir Gorodetsky, Yash V.~Singh, and Trevor Wooley for several useful discussions on related topics. 

\section{Preliminaries and Notation}
We adopt the standard convention that $\epsilon$ is a sufficiently small quantity that may vary from line to line. We shall use asymptotic notation, such as $\ll,\gg, O(\cdot),$ and $ o(\cdot)$ with their usual meanings. The implicit constants will usually depend on $q$ (the size of the constant field) or $\epsilon$ but will be uniform in all other parameters. 

\subsection{Basics of function fields} \label{subsec:basics} Let $p$ be the characteristic of the function field so that $q = p^\eta$ for some $\eta\geqs 1$. The set $\cM = \cup_{n\geqs 0} \cM_n$ denotes the set of monic polynomials in $\FqT$, with $\cM_n$ being the ones of fixed degree $n$.

The only valuation of $\Fq(T)$ that we shall use will be the one at the infinite place, so we reserve the notation $|\cdot|$ for the associated norm. Thus,
\begin{equation} \label{def:infinitenorm} |f| = q^{\deg f} = |\FqT/(f)|, \end{equation}
where the lattermost quantity is the cardinality of the residue ring. The completion of $\Fq(T)$ with respect to this norm is the ring of Laurent series in $1/T$,
\begin{equation} \label{def:inftycompl} \Fq(T)_\infty = \bigg\{ \alpha = \sum_{j=-\infty}^N \alpha_j T^{j} : \alpha_j \in \Fq \text{ and } N \in \Z \text{ with } \alpha_N \neq 0 \bigg\}.\end{equation}
We now fix an additive character on this ring. Denote by $e_p:\F_p \to \C^\times$ and $e_q:\Fq \to \C^\times$ the canonical additive characters,
\[ e_p(x) = \exp\Big(\frac{2\pi i x}{p}\Big), \qquad e_q(y) = e_p(\tr(y)), \]
where in the former $x \in \F_p$ and one interprets $x/p$ as an element of $\R/\Z$ and in the latter $y \in \Fq$ and  $\tr:\Fq \to \F_p$ is the (absolute) trace. The standard additive character on \eqref{def:inftycompl} is then defined by
\[ e(\alpha) = e_q(\res(\alpha)) = e_q(\alpha_{-1}), \]
where $\alpha \in \Fq(T)_\infty$ and $\res(\alpha)$ is the residue of $\alpha$ (i.e., the coefficient $\alpha_{-1}$).

For a primitive Dirichlet character $\chi$ modulo $\ell$, we write $r_\chi = 1 \star \chi$ for the Dirichlet convolution
\begin{equation} \label{def:rchi} r_\chi(f) = \sum_{g \mid f} \chi(g). \end{equation}
Here, and throughout, sums of the form $g \mid f$ run only over monic divisors $g \in \cM$. 

For $P$ a monic irreducible polynomial, we also define the Legendre symbol by
\begin{equation} 
\label{quadratic}
    \Big(  \frac{f}{P} \Big) = f^{ \frac{|P|-1}{2}} \pmod P,
    \end{equation}
for $f \in \mathbb{F}_q[t]$. Note that $(\frac{f}{P}) \in \{ 0, \pm1\}$.

For $k \geqs 1$, $d_k(\cdot)$ is the generalized divisor function, and for $c \in \FqT$, $g \in \cM$, and $\Re(s) > 1$ we define
\[ \zeta(s;c,g) = \sum_{\substack{f \in \cM \\ f \equiv c \pmod g}} \frac{1}{|f|^s}, \]
and
\[ D_k(s;c/g) = \sum_{f \in \cM} \frac{d_k(f)e(cf/g)}{|f|^s}. \]
These are the respective analogues of the Hurwitz zeta function and the (degree $k$) Estermann functions in $\FqT$. 

Throughout, we will let $u = q^{-s}$. Thus, $|f|^{-s} = u^{\deg f} $. With this convention, one may view the above as functions of $u$. When we do so, we will use script letters $\cZ$ and $\cD$ to emphasize the change-of-variables. That is, for $|u|<1/q$,

\begin{equation}
\label{zu}
\cZ(u;c,g) = \sum_{\substack{f \in \cM \\ f \equiv c \pmod g}} u^{\deg f}, \qquad\qquad \cD_k(u;c/g) = \sum_{f \in \cM} d_k(f)e(cf/g) u^{\deg(f)}.
\end{equation} 

\begin{remark} \label{rem:well-foundedness} Observe that if $c \equiv c' \pmod g$, then \[\cZ(u;c,g) = \cZ(u;c',g) \text{ and } \cD_1(u;c/g) = \cD_1(u;c'/g)\] A similar caveat applies to all other generating functions of this type such as $\cD_2$ or $\cR_2$ (defined below). 

We have
\[ \cR_1\Big(u;\chi,\frac{c}{g}\Big) := \sum_{f \in \cM} \chi(f) e\Big(\frac{cf}{g}\Big) u^{\deg f}, \]
\[ \cZ_\chi(u;c,g) := \sum_{\substack{f \in \cM\\f \equiv c \pmod g}} \chi(f) u^{\deg f}, \]
\begin{equation} 
\label{r2}\cR_2\Big(u;\chi, \frac{c}{g}\Big) :=  \sum_{f \in \cM} r_\chi(f)  e\Big(\frac{cf}{g}\Big) u^{\deg f},
\end{equation}
which are the character twists of $\cD_1$, $\cZ$ and $\cD_2$, respectively.

\end{remark}

We also have the usual (Riemann) zeta function,
\begin{equation}\label{def: Rzeta} \zeta(s) = \sum_{f\in \cM} \frac{1}{|f|^s}, \qquad\qquad \cZ(u) = \sum_{f \in \cM} u^{\deg f} = \frac{1}{1-qu}, \end{equation}
the functional equation of which can be written as
\begin{equation} \label{eqn:RZ_func} \cZ(u) = \frac{1-u}{u(qu-1)} \cZ\Big(\frac{1}{qu}\Big). \end{equation}

For a Dirichlet character modulo $\ell$, we also define the $L$--function as 
\[\mathcal{L}(u,\chi) = \sum_{f \in \mathcal{M}} \chi(f) u^{\deg(f)}= \prod_P (1-\chi(P)u^{\deg(P)})^{-1},\] where the product above is over monic, irreducible polynomials in $\mathbb{F}_q[t]$. When $\chi$ is a non-principal character modulo $\ell$, then $\mathcal{L}(u,\chi)$ is a polynomial of degree at most $\deg(\ell)-1$.

We will frequently use the Lindel\"{o}f bound for the $L$--function as follows. If $\chi$ is a non-principal character modulo $\ell$ and for $|u| \leqs q^{-1/2}$ and any $\epsilon>0$, we have that
\begin{equation}
\label{lindelof}
|\mathcal{L}(u,\chi)| \ll |\ell|^{\epsilon}.
    \end{equation}
(For a reference on this, see Lemma $2.5$ in \cite{DFL}.) We shall also use the fact that for $|u| \leqs1/q$, we have
\begin{equation}
\label{inverse}
 |\mathcal{L}(u,\chi)|^{-1} \ll |\ell|^{\epsilon}.   
\end{equation}
For a proof of this, see Lemma $2.6$ in \cite{DFL}. 

Throughout the paper we will also frequently use Perron's formula in function fields. Namely, if the power series $\mathcal{A}(u) = \sum_{f \in \mathcal{M}} a(f) u^{\deg(f)}$ is absolutely convergent in $|u| \leqs r<1$, then 
\begin{equation}
\label{perron}
\sum_{f \in \mathcal{M}_n} a(f) = \frac{1}{2 \pi i} \oint_{|u|=r} \frac{\mathcal{A}(u)}{u^{n+1}} \, du, \, \, \, \, \, \sum_{f \in \mathcal{M}_{\leqs n}} a(f) = \frac{1}{2 \pi i} \oint_{|u|=r} \frac{\mathcal{A}(u)}{u^{n+1}(1-u)} \, du.
\end{equation}

\subsection{Norm-counting functions for quadratic extensions} \label{subsec:norm-counting}  Let $\BA$ be a Dedekind domain whose field of fractions $k$ is a global field. Further, let $K$ be a finite extension of $k$ and let $\BB$ be the integral closure of $\BA$ in $K$. For an ideal $\fb$ in $\BB$ (resp. an ideal $\fa$ in $\BA$), one has the notion of absolute norm $|\cdot|$ given by %\marked
\[ |\fb| = |\BB/\fb| \qquad \text{ (resp. } |\fa| = |\BA/\fa|\text{).} \]
On the other hand, there is the notion of relative norm $N_{\BB/\BA}(\fb)$ which is an ideal in $\BA$ defined on primes $\fP$ by 
\[ N_{\BB/\BA}(\fP) = \fp^f \]
where $\fp = \fP \cap \BA$ is the prime under $\fP$ and $f$ is the residual (or inertial) degree of $\fP$. On other ideals of $\BB$, $N_{\BB/\BA}$ is defined by complete multiplicativity. See \cite[Chapter I, \S5]{serre}. 

\begin{remark} \label{rem:normcommutes} It is not difficult to check that the following diagram commutes, where $I$ is the set of ideals.
\[
  \begin{tikzcd}
    I_{\BB} \arrow{r}{N_{\BB/\BA}} \arrow[swap]{dr}{|\cdot|} & I_{\BA} \arrow{d}{|\cdot|} \\
     & \Z
  \end{tikzcd}
\]
\end{remark}

We now take $\BA = \FqT$ so that $k = \Fq(T)$. All ideals $\fa = (f)$ are principal and the absolute norm agrees with the valuation of $f$ as given in \eqref{def:infinitenorm}. We will also drop the subscript from $N_{\BB/\BA}$. The Dedekind zeta function of $\BB$ is given by
\[ \zeta_{\BB}(s) = \sum_{\fb \subset \BB} |\fb|^{-s}, \]
where the sum runs over non-zero integral ideals of $\BB$. This differs from the zeta function of $K$ as a curve only in that it is missing the Euler factors corresponding to the infinite places of $K$. By Remark~\ref{rem:normcommutes}, we have that
\[ \zeta_{\BB}(s) = \sum_{\fa \subset \BA} \frac{r_K(\fa)}{|\fa|^{s}} = \sum_{f \in \cM} \frac{r_K(f)}{|f|^s}, \]
where the first sum runs over integral ideals of $\BA$ and $r_K$ is the norm-counting function given by
\[ r_K(\fa) = \#\{ \fb \subset \BB : N(\fb) = \fa \}, \qquad\qquad r_K(f) = r_K((f)). \]
Some basic properties of $r_K$ are proved in \cite[\S5.2]{B-SGKS}. 

We now assume that $q \equiv 1 \pmod 4$ and assume that $K$ is a quadratic extension of $\Fq(T)$ whose field of constants has not changed. Then, $\BB = \Fq[T,\sqrt{\ell(T)}]$ for some squarefree $\ell \in \cM_{\geqs 1}$. Artin proved in his thesis \cite{artin}, \cite[Chapter 3]{roquette} that, in this case,
\[ \zeta_\BB(s) = \zeta(s)L(s,\chi) \]
where $\chi$ is the primitive quadratic character given by the Kronecker symbol $\chi(f) = \Big(  \frac{f}{\ell} \Big) = \Big(  \frac{\ell}{f} \Big)$, see \eqref{quadratic}. From this, it follows immediately that
\begin{equation} \label{eqn:r_K is r_chi} r_K(f) = r_\chi(f). \end{equation}
See also Remark~\ref{rem:norm-counting}.
\subsection{Character sums over \texorpdfstring{$\FqT$}{FqT}} 

For $f,h \in \Fq[T]$ and $g \in \cM$, we denote the Kloosterman sum over $\FqT$ by
\begin{equation} \label{def:S_over_FqT} S(f,h;g) = \copsum_{a \pmod g} e\Big(\frac{af + \ov{a}h}{g}\Big), \end{equation}
where the $*$ indicates that the sum runs over coprime residue classes modulo $g$, and $\ov{a}$ is the multiplicative inverse of $a$ modulo $g$.

For $\ell \in \cM$ and a character $\chi$ modulo $\ell$, we denote the Gauss sum by
\begin{equation} \label{def:Gauss_sum_over_FqT} G(\chi,f) = \sum_{a \pmod \ell} \chi(a)e\Big(\frac{af}{\ell}\Big), \qquad G(\chi) = G(\chi,1). \end{equation}
Note that if $\chi$ is primitive then $G(\chi,f) = \ov{\chi}(f) G(\chi)$. 

Throughout the paper, we will use the following point-wise bound for the Kloosterman sum. Lemma $1.13$ in \cite{bagshaw} gives that 
\begin{equation}
    |S(f,h;g)| \leqs2^{\omega(g)} q^{\frac{\deg(g)}{2}+ \frac{ \deg(f,h,g)}{2}},
    \label{weil}
\end{equation}
where $\omega(g)$ denotes the number of monic, irreducible factors of $g$. 

We will also use the Linnik--Selberg bound for averages of Kloosterman sums over function fields (see Remark 3 following Conjecture $1.4$ in \cite{Sardari-Zargar} and \cite{cps}). Namely,
\begin{equation} 
\label{linnik-selberg}
    \Big| \sum_{g \in \mathcal{M}_n} S(f,h;g) \Big| \ll q^{n(1+\epsilon)} |fh|^{\epsilon}.
\end{equation}

\subsection{Character sums over \texorpdfstring{$\Fq$}{Fq}} 

The orthogonality of additive characters states that
\[ \1_{x = 0} = \frac{1}{q} \sum_{\lambda \in \Fq} e_q(\lambda x) \]
for $x \in \Fq$. Separating the term $\lambda = 0$ and rearranging, we have that
\begin{equation} \label{eqn:ortho_over_Fqunits} \sum_{\lambda \in \Fq^\times} e_q(\lambda x) = - 1 + q\1_{x = 0}, \end{equation}
which will be useful in \S\ref{sec:func}.

The $r$-dimensional hyper-Kloosterman sum over $\Fq$ will be denoted by $\Kl_r$. For $\alpha_1,\cdots,\alpha_r \in \Fq$, this is defined by 
\begin{equation} \label{def:HKl_over_Fq} \Kl_r(\alpha_1,\cdots,\alpha_r) = \sum_{\substack{\gamma_1,\cdots,\gamma_r \in \Fq^\times\\\gamma_1\gamma_2 \cdots\gamma_n = 1}} e_q(\alpha_1 \gamma_1 + \alpha_2 \gamma_2 + \cdots +\alpha_n \gamma_n). \end{equation}
The important cases for us are 
\begin{equation} \label{def:Kl_over_Fq} \Kl_1(\lambda) = e_q(\lambda), \qquad \Kl_2(\alpha,\beta) = \sum_{\gamma\in \Fq^\times}e_q(\alpha\gamma + \beta\gamma^{-1}).\end{equation}
Further, it is not hard to see that
\begin{equation} \label{eqn:Kl0a_orKla0} \Kl_2(\alpha,0) = \Kl_2(0,\alpha) = -1 \text{ when } \alpha \neq 0 \end{equation}
and
\begin{equation} \label{eqn:Kl00} \Kl_2(0,0) = q-1. \end{equation}

For a multiplicative character of $\Fq$, $\chi$, we will also need the $\chi$-twisted Kloosterman sum $\Kl_{\chi}$ defined by 
\begin{equation} \label{def:chi_Kl_over_Fq} \Kl_{\chi}(\alpha,\beta) = \sum_{\gamma\in\Fq^\times} \chi(\gamma) e_q(\alpha \gamma + \beta \gamma^{-1}). \end{equation}
\begin{remark}
If $\chi$ is a Dirichlet character over $\FqT$, then the restriction $\chi\big|_{\F_q^\times}$ is a bonafide multiplicative character of $\Fq$ and hence has an associated twisted Kloosterman sum. We will slightly abuse notation and write $\Kl_\chi$ also for the latter. 
\end{remark}

An important special case of \eqref{def:chi_Kl_over_Fq} is
\begin{equation} \label{def:tau_f_over_Fq} \Kl_\chi(\alpha,0) = \sum_{\gamma \in \Fq^\times} \chi(\gamma) e_q(\alpha\gamma) =: \tau(\chi,\alpha), \end{equation}
which is the Gauss sum of (the restriction of) $\chi$ over $\Fq$. Similarly,
\begin{equation} \label{eqn:Klchi0b} \Kl_\chi(0,\beta) = \sum_{\gamma \in \Fq^\times} \chi(\gamma)e_q(\beta\gamma^{-1}) = \tau(\ov{\chi},\beta),\end{equation} 
by substituting $\gamma$ by $\gamma^{-1}$ and noting that $\ov{\chi}(\gamma) = \chi(\gamma^{-1})$. Finally,
\begin{equation} \label{eqn:Klchi00} \Kl_\chi(0,0) = \sum_{\gamma \in \Fq^\times} \chi(\gamma) = (q-1)\1(\chi \text{ is even}),\end{equation}
where we recall that $\chi$ is an even Dirichlet character if it is trivial on the units $\Fq^\times$ and is odd otherwise. 

We also set,
\begin{equation} \label{def:tau_over_Fq} \tau(\chi) = \tau(\chi,1) = \sum_{\gamma \in \Fq^\times} \chi(\gamma) e_q(\gamma). \end{equation}

\section{Functional equations} \label{sec:func}
Recall the definition \eqref{zu} of $\mathcal{Z}(u;c,g)$ and $\mathcal{D}_k(u;c/g)$, which holds for $|u|<1/q$. 

\subsection{Degree one} Our first goal is to derive functional equations for the degree one functions $\cZ$ and $\cD_1$. We have the following representation for these functions, which implicitly gives us a meromorphic continuation to the whole complex plane.

\begin{lemma} \label{lem: DZ simple} Let $g \in \cM$ and $c \in \Fq[T]$ with $\deg(c) < \deg(g)$. Then, for all $u \in \C$, we have that
\begin{equation} \label{eqn: Z simple} \cZ(u;c,g) = \frac{u^{\deg(g)}}{1 - qu}+ \1_{c\in \mathcal{M}} u^{\deg(c)}. \end{equation}
If, further, $c \neq 0$, then we have that
\begin{equation} \label{eqn: D simple} \mathcal{D}_{1}\Big(u; \frac{c}{g} \Big) = (qu)^{\deg(g) - \deg(c) - 1}\bigg( e_q\big(\sgn(c)\big) - \frac{1}{1 - qu} \bigg) + \frac{1}{1 - qu} \end{equation}
where $\1_{c\in \cM}$ is an indicator function detecting whether $c$ is monic and $\sgn(c)$ is the leading coefficient of $c$ as a polynomial. 
\end{lemma} 

\begin{proof}

To prove the first equality, we note that writing $f = g f_1 + c$, we find that since $g$ is monic and $\deg(c)<\deg(g)$, then the condition that $f \in \cM$ is equivalent to $f_1 \in \cM$ \emph{unless} $c$ is monic and $f_1 = 0$. In the former case, $\deg(f) = \deg(gf_1 + c) = \deg(g) + \deg(f_1)$, whence $|f| = |g||f_1|$. In the latter case, $|f|=|c|$. Thus, we find that 
\[ \zeta(s;c,g) = \sum_{\substack{f \in \cM \\ f\equiv c \pmod g}} |f|^{-s} = |g|^{-s} \sum_{f_1 \in \cM} |f_1|^{-s} + \1_{c\in \cM} |c|^{-s}. \]
The claim now follows from \eqref{def: Rzeta} by rewriting both sides as functions of $u = q^{-s}$. Note that this argument also applies for $g = 1$ (in which case $c$ is necessarily $0$). 

To prove the second equality, we first note (by Taylor expansion) that the right hand side is a polynomial in $u$, given by
\[ \sum_{n = 0}^{\deg(g)-\deg(c) - 1} a_n u^n, \]
with $a_n = q^n$ for $0 \leqs n \leqs \deg(g) - \deg(c) - 2$ and $a_n = q^n e_q\big(\sgn(c)\big)$ for $n = \deg(g) - \deg(c) - 1$. On the other hand, arranging the left hand side by the degree of $f$, we find that
\[ \cD_1(u;c/g) = \sum_{f \in \cM} e\Big(\frac{cf}{g}\Big) u^{\deg f} = \sum_{n=0}^\infty u^n \sum_{f \in \cM_n} e\Big(\frac{cf}{g}\Big).\]
The equality then follows from the fact \cite[Theorem 3.7]{hayes} that under the hypotheses of the lemma,
\[ \frac{1}{q^n}\sum_{f\in\cM_n} e\Big(\frac{cf}{g}\Big) = \begin{cases} \;\; 1 & \text{if }n < \deg(g) - \deg(c) -1, \\ \;\; e_q\big(\sgn(c)\big) & \text{if }n = \deg(g) - \deg(c) - 1, \\ \;\; 0 & \text{if }n > \deg(g) - \deg(c) - 1.   \end{cases} \]
Note that in this case, we need not worry about the case $g = 1$, since that is incompatible with the requirements that $c \neq 0$ and $\deg(c) < \deg(g)$. 
\end{proof}

It is worth remarking that the above lemma implies that
\[ \cZ(u;0,g) = u^{\deg(g)} \cZ(u), \]
with $\cZ(u)$ given in \eqref{def: Rzeta}. On the other hand, it follows directly from the definition that
\[ \cD_1(u;0) = \cZ(u).\]

We are now ready to prove the functional equation of the Hurwitz zeta function. 
 
\begin{proposition}[Functional equation of $\cZ$] \label{prop:Z_func} Let $g \in \cM$ and $c \in \Fq[T]$. Then, for all $u \in \C$, we have that
\begin{equation} \label{eqn:Z_func} \cZ(u;c,g)= u^{\deg(g)-1} \sum_{\lambda \in \F_q^\times} \bigg[ \frac{e_q(-\lambda)}{q}-\frac{1}{q(1-qu)} \bigg] \mathcal{D}_1\Big( \frac{1}{qu}, \frac{\lambda c}{g}\Big) \end{equation}
\end{proposition}

\begin{proof}
First, since both sides depend only on the congruence class of $c$ modulo $g$ by Remark~\ref{rem:well-foundedness}, we can assume without loss of generality that $\deg(c) < \deg(g)$. 

Let us quickly dispose of the case $c = 0$ by noting that in this case $\cZ(u;0,g)$ on the left hand side becomes $u^{\deg(g)} \cZ(u)$ while every copy of $\cD_1$ on the right hand side of \eqref{eqn:Z_func} is equal to $\cZ(1/qu)$. Thus, we may execute the sum over $\lambda$ by using \eqref{eqn:ortho_over_Fqunits}. On doing so, \eqref{eqn:Z_func} simply becomes the functional equation of the Riemann zeta function, \eqref{eqn:RZ_func}, multiplied by $u^{\deg(g)}$.

When $c \neq 0$, we can replace $u$ by $1/(qu)$ in \eqref{eqn: D simple} and simplify, to find that
\[ \cD_1\Big(\frac{1}{qu}; \frac{\lambda c}{g}\Big) = u^{\deg(c) + 1 - \deg (g)}\bigg( e_q\big(\sgn(\lambda c)\big) - \frac{u}{u - 1} \bigg) + \frac{u}{u - 1}. \]
Since $\sgn(\lambda c) = \lambda \sgn(c)$, we see that the right hand side of \eqref{eqn:Z_func} may be written as 
\begin{equation} \label{eqn: ABCD} \sum_{\lambda \in \Fq^\times} \Big[A e_q(\lambda (\sgn(c) - 1)) + B e_q(\lambda\sgn(c)) + C e_q(-\lambda) + D\Big], \end{equation}
where 
\[ A : = \frac{u^{\deg(c)}}{q}, \qquad B:= -\frac{u^{\deg(c)}}{q(1-qu)},\] 
\[ C:= \frac{u^{\deg(g)} - u^{\deg(c) + 1}}{q(u-1)}, \qquad D:= \frac{u^{\deg(c) + 1}-u^{\deg(g)}}{q(1-qu)(u-1)}. \]
By orthogonality (cf. \eqref{eqn:ortho_over_Fqunits}), \eqref{eqn: ABCD} is equal to
\[ - A + qA \1_{\sgn(c) - 1 = 0} - B - C + (q-1)D. \]
Now, by a routine calculation, we find that
\[ -A - B - C + (q-1)D = \frac{u^{\deg(g)}}{1-qu}, \]
while clearly,
\[ qA\1_{\sgn(c) - 1 = 0} = \1_{c \in \cM} u^{\deg(c)}. \]
Note that, by \eqref{eqn: Z simple}, the sum of the above two expressions is precisely the left hand side of \eqref{eqn:Z_func}, completing the proof.

\end{proof}

\begin{remark} \label{rem:eta_sum_for_Z} The expression in square parentheses in \eqref{eqn:Z_func} may be written as
\[ \sum_{\eta \in \{0,1\}} (-1)^{\eta} \frac{e_q\big((\eta-1)\lambda\big)}{q(1-qu)^{\eta}} = \sum_{\eta \in \{0,1\}} (-1)^{\eta} \frac{\Kl_1\big((\eta-1)\lambda\big)}{q(1-qu)^{\eta}}.\]
This will be useful later.
\end{remark}

Proposition~\ref{prop:Z_func} can be reversed (for example, by basic linear algebra over $\C(u)$) to obtain the functional equation of $\cD_1$. It is slightly more convenient to proceed directly, with a similar argument.

\begin{proposition}[Functional equation of $\cD_1$]  \label{prop:D1_func} Let $g \in \cM$ and $c \in \Fq[T]$. Then, for all $u \in \C$, we have that
 \begin{equation} \label{eqn:D1_func}
\mathcal{D}_{1}\Big(u; \frac{c}{g} \Big) = (qu)^{\deg(g)-1} \sum_{\lambda \in \mathbb{F}_q^\times} \bigg[e_q(\lambda^{-1})-\frac{1}{1-qu}\bigg]  \mathcal{Z}\Big(\frac{1}{qu};\lambda c,g\Big).\end{equation}
\end{proposition}

\begin{proof}

We proceed similarly to the proof of Proposition~\ref{prop:Z_func}. As before, we may restrict without losing generality to $\deg(c) < \deg(g)$ with $c \neq 0$ (in the case $c = 0$, \eqref{eqn:D1_func} devolves to \eqref{eqn:RZ_func} by applying \eqref{eqn:ortho_over_Fqunits}) and thus, by \eqref{eqn: Z simple}, 
\[ \cZ\Big(\frac{1}{qu};\lambda c, g\Big) = (qu)^{-\deg(g)}\frac{u}{u - 1}+ \1_{\lambda c\in \mathcal{M}} (qu)^{-\deg(c)}, \]
whence the right hand side of \eqref{eqn:D1_func} may be written as 
\begin{equation} \label{eqn: ABCD'} \sum_{\lambda \in \Fq^\times} \Big[A' e_q(\lambda^{-1})\1_{\lambda c \in \cM} + B'\1_{\lambda c \in \cM} + C' e_q(\lambda^{-1}) + D'\Big],\end{equation}
with
\[ A' : = (qu)^{\deg(g) - \deg(c) - 1}, \qquad B':= -\frac{(qu)^{\deg(g) - \deg(c)-1}}{1-qu},\] 
\[ C':= \frac{1}{q(u-1)}, \qquad D':= -\frac{1}{q(1-qu)(u-1)}. \]
Note that $\lambda \to \lambda^{-1}$ is an involution on $\Fq^\times$, so that \eqref{eqn: ABCD'} is equal to
\begin{equation*} \begin{split} \sum_{\lambda \in \Fq^\times} \Big[A' e_q(\lambda)\1_{\lambda^{-1} c \in \cM} + B'\1_{\lambda^{-1} c \in \cM} + C' e_q(\lambda) + D'\Big], \\ = A'e_q\big(\sgn(c)\big) + B' -C' + (q-1)D',\end{split} \end{equation*}
where we have used \eqref{eqn:ortho_over_Fqunits} together with the fact that $\lambda^{-1} c \in \cM$ if and only if $\lambda = \sgn(c)$. 
The result now follows from observing that
\[ A'e_q\big(\sgn(c)\big) + B' = (qu)^{\deg(g) - \deg(c) - 1}\bigg( e_q\big(\sgn(c)\big) - \frac{1}{1 - qu} \bigg) \]
is the first term in \eqref{eqn: D simple} while
\[ - C' + (q-1) D' = \frac{1}{1-qu}\]
is the second.

\end{proof}

\begin{remark} \label{rem:eta_sum_for_D1} Similar to Remark~\ref{rem:eta_sum_for_Z}, the expression in square parentheses in \eqref{eqn:D1_func} may be written as
\[ \sum_{\eta \in \{0,1\}} (-1)^{\eta} \frac{e_q\big((1-\eta)\lambda^{-1}\big)}{(1-qu)^{\eta}} = \sum_{\eta \in \{0,1\}} (-1)^{\eta} \frac{\Kl_1\big((1-\eta)\lambda^{-1}\big)}{(1-qu)^{\eta}}.\]
\end{remark}

\subsection{Degree two} We will now use Propositions~\ref{prop:Z_func}~and~\ref{prop:D1_func} in tandem to prove the functional equation for $\cD_2$. For this, we need a lemma that expresses $\cD_2$ as a combination of $\cZ$ and $\cD_1$. This lemma also implicitly provides meromorphic continuation for $\cD_2$ to the whole complex plane. 

\begin{lemma} \label{lem:D2_simple} Let $g \in \cM$ and let $c \in \Fq[T]$. Then, for all $u \in \C$, we have that
\begin{equation} \label{eqn:D2_simple} \cD_2\Big(u;\frac{c}{g}\Big) = \sum_{a \pmod g} \cZ(u;a,g)\cD_1\Big(u;\frac{ac}{g}\Big). \end{equation}

If we further assume that $(c,g) = 1$ and let $b\in \Fq[T]$ (not necessarily coprime to $g$), then for all $u \in C$, we have that
\begin{equation} \label{eqn:D2_simple_inverse} \cD_2\Big(u;\frac{b\ov{c}}{g}\Big) = \sum_{a \pmod g} \cZ(u;ac,g)\cD_1\Big(u;\frac{ab}{g}\Big), \end{equation}
where $\ov{c}$ is any inverse of $c$ modulo $g$.
\end{lemma}

\begin{proof} As in the proof of Lemma~\ref{lem: DZ simple}, we may assume $|u| < 1/q$ for absolute convergence. On doing so, we find that
\begin{equation} \label{eqn:D2_unfolded} \cD_2\Big(u;\frac{c}{g}\Big) = \sum_{f_1,f_2 \in \cM} e\Big(\frac{cf_1f_2}{g}\Big) u^{\deg(f_1) + \deg(f_2)}. \end{equation} 
We now collect the terms with $f_1$ in a particular congruence class $a$ modulo $g$ together. For all such terms, $e(cf_1f_2/g) = e(acf_2/g)$; from this it follows that
\[ \cD_2\Big(u;\frac{c}{g}\Big) = \sum_{a \pmod g} \sum_{\substack{f_1 \in \cM\\f_1 \equiv a \pmod g}} u^{\deg(f_1)} \sum_{f_2 \in \cM} e\Big(\frac{acf_2}{g}\Big) u^{\deg(f_2)}. \]
We can close the sum over $f_2$ as a copy of $\cD_1(u;ac/g)$, while the sum over $f_1$ will then yield a copy of $\cZ(u;a,g)$; from this we obtain
\[ \cD_2\Big(u;\frac{c}{g}\Big) = \sum_{a \pmod g} \cZ(u;a,g)\cD_1\Big(u;\frac{ac}{g}\Big), \]
as desired to prove \eqref{eqn:D2_simple}.

To prove \eqref{eqn:D2_simple_inverse}, we first replace $c$ with $b\ov{c}$ and $a$ with $a_1$ in \eqref{eqn:D2_simple} to get
\[ \cD_2\Big(u;\frac{b\ov{c}}{g}\Big) = \sum_{a_1 \pmod g} \cZ(u;a_1,g)\cD_1\Big(u;\frac{a_1b\ov{c}}{g}\Big). \]
Now, we make the substitution $a_1= ac$. This gives
\[ \cD_2\Big(u;\frac{b\ov{c}}{g}\Big) = \sum_{a \pmod g} \cZ(u;ac,g)\cD_1\Big(u;\frac{acb\ov{c}}{g}\Big) = \sum_{a \pmod g} \cZ(u;ac,g)\cD_1\Big(u;\frac{ab}{g}\Big), \]
as desired. Here, we have used $c\ov{c} \equiv 1 \pmod g$, together with Remark~\ref{rem:well-foundedness}. 
\end{proof}

We are now in a position to prove the functional equation for the Estermann function, $\cD_2$. 

\begin{proposition} \label{prop:D2_func}  Let $g \in \cM_{\geqs 1}$ and $c\in \Fq[T]$ be such that $(c,g) = 1$. Then, for all $u \in \C$, we have that
\begin{equation} \label{eqn:D2_func} \cD_2\Big(u;\frac{c}{g}\Big) = (qu^2)^{\deg(g) -1} \sum_{\lambda \in \Fq^\times} \cA_\lambda(u) \cD_2\Big(\frac{1}{qu};\frac{\lambda\ov{c}}{g}\Big),\end{equation}
where $\ov{c}$ is any inverse of $c$ modulo $g$ and
\begin{equation}\label{def:A} \begin{split} \cA_\lambda(u) :&= \sum_{\eta_1,\eta_2 \in \{0,1\}}\bigg[ (-1)^{\eta_1+\eta_2}\frac{\Kl_2\big(\eta_1-1,(1-\eta_2)\lambda\big)}{q (1-qu)^{\eta_1 + \eta_2}}\bigg] \\ & =   \frac{\Kl_2(-1,\lambda)}{q} + \frac{1+q - 2qu}{q(1-qu)^2}.\end{split} \end{equation}
Here, $\Kl_2$ is as in \eqref{def:Kl_over_Fq}.
\end{proposition}

\begin{proof}
Before proving \eqref{eqn:D2_func}, let us establish the second equality in \eqref{def:A}. By \eqref{eqn:Kl0a_orKla0} and \eqref{eqn:Kl00}, we have
\begin{equation*} \begin{split} \cA_\lambda(u) & = \frac{\Kl_2(-1,\lambda)}{q} - \frac{\Kl_2(-1,0)}{q(1-qu)} - \frac{\Kl_2(0,\lambda)}{q(1-qu)} + \frac{\Kl_2(0,0)}{q(1-qu)^2} \\ & = \frac{\Kl_2(-1,\lambda)}{q} + \frac{1}{q(1-qu)} + \frac{1}{q(1-qu)} + \frac{q-1}{q(1-qu)^2}. \end{split} \end{equation*} 
The claim follows on combining the latter three terms above.

The proof strategy for \eqref{eqn:D2_func} is hopefully clear: we use Lemma~\ref{lem:D2_simple} to write $\cD_2$ as a combination of $\cD_1$ and $\cZ$ and apply the functional equations (cf. Propositions~\ref{prop:D1_func}~and~\ref{prop:Z_func}) of $\cD_1$ and $\cZ$. To put this into effect, we substitute \eqref{eqn:Z_func} and \eqref{eqn:D1_func} into the right hand side of \eqref{eqn:D2_simple} -- with $c$ replaced by $a$ or $ac$, as necessary -- to find that
\begin{multline} \label{eqn:D2_func_step1} \cD_2\Big(u;\frac{c}{g}\Big) = (qu^2)^{\deg(g)-1}\sum_{a \pmod g}\Bigg(\sum_{\lambda_1 \in \Fq^\times}\cD_1\Big(\frac{1}{qu};\frac{\lambda_1 a}{g}\Big) \sum_{\eta_1 \in \{0,1\}} (-1)^{\eta_1} \frac{e_q\big((\eta_1 - 1)\lambda_1\big)}{q(1-qu)^{\eta_1}} \\ \times \sum_{\lambda_2 \in \Fq^\times} \cZ\Big(\frac{1}{qu};\lambda_2ac,g\Big) \sum_{\eta_2 \in \{0,1\}} (-1)^{\eta_2} \frac{e_q\big((1-\eta_2)\lambda_2^{-1}\big)}{(1-qu)^{\eta_2}}  \Bigg),\end{multline}
where we have implicitly used Remark~\ref{rem:eta_sum_for_Z} and Remark~\ref{rem:eta_sum_for_D1}. 

We now focus only on the sum over $a$. We have
\[ \sum_{a\pmod g} \cZ\Big(\frac{1}{qu};\lambda_2 ac,g\Big) \cD_1\Big(\frac{1}{qu};\frac{\lambda_1 a}{g}\Big) = \cD_2\Big(\frac{1}{qu};\frac{\lambda_1\lambda_2^{-1}\ov{c}}{g}\Big),\]
where we have used \eqref{eqn:D2_simple_inverse} after making the appropriate substitutions. Substituting this back into \eqref{eqn:D2_func_step1} and rearranging, we see that 
%\begin{multline*} \cD_2\Big(u;\frac{c}{g}\Big) = (qu^2)^{\deg(g)-1}\sum_{\lambda_1,\lambda_2 \in \Fq^\times}\Bigg( \bigg[\sum_{\eta_1,\eta_2 \in \{0,1\}} (-1)^{\eta_1+\eta_2} \frac{e_q\big((\eta_1 - 1)\lambda_1+(1-\eta_2)\lambda_2^{-1}\big)}{q(1-qu)^{\eta_1+\eta_2}}\bigg] \\ \times \cD_2\Big(\frac{1}{qu};\frac{\lambda_1\lambda_2^{-1}\ov{c}}{g}\Big) \Bigg).\end{multline*}
\begin{multline*} \cD_2\Big(u;\frac{c}{g}\Big) = (qu^2)^{\deg(g)-1}\sum_{\lambda_1,\lambda_2 \in \Fq^\times} \Bigg(\sum_{\eta_1,\eta_2 \in \{0,1\}} (-1)^{\eta_1+\eta_2} \\ \times \frac{e_q\big((\eta_1 - 1)\lambda_1+(1-\eta_2)\lambda_2^{-1}\big)}{q(1-qu)^{\eta_1+\eta_2}}  \cD_2\Big(\frac{1}{qu};\frac{\lambda_1\lambda_2^{-1}\ov{c}}{g}\Big)\Bigg).\end{multline*}
We now collect the terms which have the same sign $\lambda = \lambda_1\lambda_2^{-1}$ together. Thus, the sum over $(\lambda_1,\lambda_2)$ may be replaced by a sum over $(\lambda,\lambda_1)$. On doing so, and rearranging with the sum over $\lambda$ outside, and the sum over $\lambda_1$ all the way inside, we get
\begin{multline*} \cD_2\Big(u;\frac{c}{g}\Big) = (qu^2)^{\deg(g)-1}\sum_{\lambda\in\Fq^\times} \cD_2\Big(\frac{1}{qu};\frac{\lambda\ov{c}}{g}\Big) \bigg[\sum_{\eta_1,\eta_2 \in \{0,1\}} \frac{(-1)^{\eta_1+\eta_2}}{q(1-qu)^{\eta_1+\eta_2}} \\\times \sum_{\lambda_1\in\Fq^\times} e_q\big((\eta_1 - 1)\lambda_1+(1-\eta_2)\lambda\lambda_1^{-1}\big)\bigg]. \end{multline*}
We can now close the sum over $\lambda_1$ as $\Kl_2(\eta_1 - 1,(1-\eta_2)\lambda)$, whence \eqref{eqn:D2_func} follows by noting that the quantity in the square brackets in the right hand side is precisely $\cA_\lambda(u)$.
\end{proof}

\begin{remark} The identity
\[ \cD_2\Big(u;\frac{c}{g}\Big) = \sum_{a,b \pmod g} e\Big(\frac{abc}{g}\Big) \cZ(u;a,g)\cZ(u;b,g) \]
holds simply by fibering the right hand side of \eqref{eqn:D2_unfolded} in terms of congruences classes of both $f_1$ and $f_2$ modulo $g$. With this as a starting point, an alternate proof of Proposition~\ref{prop:D2_func} may be obtained by applying the functional equation of $\cZ$ twice. In the classical setting, this is essentially Estermann's original strategy \cite{OGestermann} (see also \cite[Lemma 4]{conreylevinson}). This route is marginally tedious compared to our approach (especially for character-twisted analogues and the degree $3$ analogue \`a la Fazzari \cite{fazzari}; see \S\ref{sec:chifunc} and the remark below), so we leave the details to the interested reader. 
\end{remark}

\begin{remark}
An analogue of Propositions~\ref{prop:D1_func}~and~\ref{prop:D2_func} for the $GL(3)$ Estermann function, $\cD_3(u;c/g)$, has been obtained by the authors, the details of which will appear elsewhere. 
\end{remark}

\subsection{Character twists} \label{sec:chifunc}

In this section, $\chi$ is a fixed primitive character modulo $\ell$. We are making neither the assumption that $\ell$ is prime nor the assumption that $\ell = 1$ here, though one of these assumptions will become necessary in later sections. Our goal is to prove a functional equation for $\cR_2(u;\chi,c/g)$, given in \eqref{r2}. This is more complicated when the greatest common divisor $(\ell,g) \notin \{1,\ell\}$, so we restrict ourselves to the cases $(\ell,g) = 1$ and $(\ell,g) = \ell$. The reader interested in the general case may try adapting \cite[Theorem~2.3]{topaindiv}.

\begin{lemma} \label{lem:R1_func}

Let $g,\ell \in \cM$ and $c \in \Fq[T]$ such that $(\ell,g) = 1$. Further, let $\chi$ be a primitive Dirichlet character modulo $\ell$. Then, for all $u \in \C$, we have that

\begin{multline} \label{eqn:R1_func} \cR_1\Big(u;\chi,\frac{c}{g}\Big) = \chi(g)G(\chi)(qu)^{\deg(g)-1} u^{\deg(\ell)} \\\times\sum_{\lambda\in\Fq^\times} \chi(-\lambda)\bigg[e_q(\lambda^{-1}) - \frac{1}{1-qu}\bigg] \cZ_{\ov\chi}\Big(\frac{1}{qu};\lambda c\ell, g\Big), \end{multline}
where $G(\chi)$ is the Gauss sum of $\chi$ defined in \eqref{def:Gauss_sum_over_FqT}.
\end{lemma}

\begin{proof}

From the definition of $\cR_1$, it follows that
\[ \cR_1\Big(u;\chi,\frac{c}{g}\Big) = \sum_{a_1 \pmod \ell} \sum_{a_2\pmod g} \chi(a_1) e\Big(\frac{a_2c}{g}\Big) \sum_{\substack{f \equiv a_1 \pmod \ell \\ f \equiv a_2 \pmod g}} u^{\deg f}, \]
by fibering the sum over $f \in \cM$ according to the congruence class of $f$ modulo $g$ (say $a_1$) and $\ell$ (say $a_2$) respectively. 

Now, let $\ov{g}$ represent any fixed multiplicative inverse of $g$ modulo $\ell$ (and the same for $\ov{\ell}$ with roles reversed). These are well-defined as $(\ell,g) = 1$. By the Chinese Remainder Theorem, if we set
\begin{equation} \label{eqn:aCRT} a  : = a_1 g\ov{g} + a_2 \ell \ov{\ell}, \end{equation}
where $a$ depends on $a_1$ and $a_2$, then the congruence conditions in the right-most sum above may be replaced by $f \equiv a \pmod{g\ell}$. Doing so and recalling the definition of $\cZ$ \eqref{zu} yields the identity
\[ \cR_1\Big(u;\chi,\frac{c}{g}\Big) = \sum_{a_1 \pmod \ell} \sum_{a_2\pmod g} \chi(a_1) e\Big(\frac{a_2c}{g}\Big) \cZ(u;a,g\ell). \]
Inserting the functional equation of $\cZ$ above (cf. \eqref{eqn:Z_func}), we find the following expression for $\cR_1(u;\chi,c/g)$:
\[ \sum_{a_1 \pmod \ell} \sum_{a_2\pmod g} \chi(a_1) e\bigg(\frac{a_2c}{g}\bigg) \sum_{\lambda \in \Fq^\times} \bigg[ \frac{e_q(\lambda^{-1})}{q} - \frac{1}{q(1-qu)}\bigg]u^{\deg(g\ell)-1} \cD_1\Big(\frac{1}{qu},\frac{\lambda a}{g\ell}\Big). \]

We now open the definition of $\cD_1$ as a sum over $f$ and push the sums over $a_1$ and $a_2$ inside. Recalling the definition of $a$ from \eqref{eqn:aCRT} and focusing only on the terms depending on $a_1$ and $a_2$, we find 
\[ \sum_{a_1 \pmod{\ell}} \sum_{a_2\pmod{g}} \chi(a_1) e\bigg(\frac{a_2c}{g} + \frac{\lambda a_2\ov{\ell} f}{g} + \frac{\lambda a_1 \ov{g} f}{\ell}\bigg). \]
We can perform the summation over $a_2$ using orthogonality of characters. This yields the congruence condition $f \equiv - c\lambda^{-1} \ell \pmod g$ and contributes a factor of $|g| = q^{\deg(g)}$. On the other hand, closing the sum over $a_1$ yields a Gauss sum
\[ G(\chi,\lambda \ov{g} f) = \chi(\lambda^{-1})\chi(g)\ov{\chi}(f) G(\chi), \]
where we have used the primitivity of $\chi$. From here, we get that
\begin{multline*} \cR_1\Big(u;\chi,\frac{c}{g}\Big) =  \chi(g)G(\chi)(qu)^{\deg(g)-1} u^{\deg(\ell)} \sum_{\lambda\in\Fq^{\times}} \chi(\lambda^{-1})\bigg[e_q(-\lambda) - \frac{1}{1-qu}\bigg] \\ \times \sum_{\substack{f \in \cM \\ f \equiv -\lambda^{-1}c\ell \pmod{g}}} \ov{\chi}(f) (qu)^{-\deg f}. \end{multline*}
Here it is worth remarking that we have consolidated the powers of $q$ arising from the congruence sums with the extra factor inside the $\lambda$-sum. Making the change of variables $\lambda \mapsto -\lambda^{-1}$ and closing the sum over $f$, we obtain 
\begin{multline*} \cR_1\Big(u;\chi,\frac{c}{g}\Big) = \chi(g)G(\chi)(qu)^{\deg(g)-1} u^{\deg(\ell)} \\\times\sum_{\lambda\in\Fq^\times} \chi(-\lambda)\bigg[e_q(\lambda^{-1}) - \frac{1}{1-qu}\bigg] \cZ_{\ov\chi}\Big(\frac{1}{qu};\lambda c\ell, g\Big), \end{multline*}
as desired. 
\end{proof}
\begin{remark} \label{rem:eta_sum_for_R1}
It is worth noting that the term in square brackets in \eqref{eqn:R1_func} is the same as in Remark~\ref{rem:eta_sum_for_D1} and so may be written as
\[ \sum_{\eta\in\{0,1\}} (-1)^{\eta} \frac{e_q\big((1-\eta)\lambda^{-1}\big)}{(1-qu)^{\eta}}.\] 
\end{remark}

The case $g = 1$ of the above lemma is particularly interesting, since it gives an alternate proof of the functional equation for a primitive Dirichlet $L$-function, $\cL(u,\chi)$, which is similar to Hurwitz's original proof of the functional equation in the integer setting (see the discussion in Davenport \cite[pp. 65, 71-72]{davenport}). 

The stage is now set for proving a functional equation for $\cR_2(u;\chi,c/g)$ when $(\ell,g) = 1$.
\begin{proposition} \label{prop:R2_func_coprime} Let $g \in \cM_{\geqs 1}$, $\ell \in \cM$, and $c\in \Fq[T]$ be such that $(\ell c,g) = 1$. Further, let $\chi$ be a primitive Dirichlet character modulo $\ell$. Then, for all $u \in \C$, we have that
\begin{equation} \label{eqn:R2_func_coprime} \cR_2\Big(u;\chi,\frac{c}{g}\Big) =\chi(g)G(\chi)(qu^2)^{\deg(g)-1} u^{\deg(\ell)} \sum_{\lambda \in \Fq^\times} \cB_{\lambda}(u;\ov{\chi})\cR_2\Big(\frac{1}{qu};\ov{\chi},\frac{\lambda\ov{c}}{g}\Big),\end{equation}
where 
\begin{equation} \begin{split} \label{def:B} \cB_\lambda(u;\chi) &:= \sum_{\eta_1,\eta_2 \in \{0,1\}} (-1)^{\eta_1+\eta_2}\frac{\Kl_{\chi}\big(\eta_1-1,(1-\eta_2)\lambda\big)}{q(1-qu)^{\eta_1+\eta_2}} \\ &= \frac{\Kl_{\chi}(-1,\lambda)}{q} - \bigg[\frac{\tau(\chi,-1) + \tau(\ov{\chi},\lambda)}{q(1 - qu)}\bigg] + \frac{q-1}{q(1-qu)^2} \1(\chi \textup{ is even}).\end{split} \end{equation}
Here, $G(\chi)$, $\Kl_\chi$, $\tau(\chi,\cdot)$ are as in \eqref{def:Gauss_sum_over_FqT}, \eqref{def:chi_Kl_over_Fq}, \eqref{def:tau_f_over_Fq}.
\end{proposition}

\begin{proof}
For $(\ell c,g) = 1$ and any $b \in \FqT$, the identities
\begin{equation} \label{eqn:R2_simple1} \cR_2\Big(u;\chi,\frac{c}{g}\Big) = \sum_{a \pmod{g}} \cZ(u;a,g) \cR_1\Big(u;\chi,\frac{ac}{g}\Big) \end{equation}
and
\begin{equation} \label{eqn:R2_simple2} \cR_2\Big(u;\chi, \frac{b\ov{c}}{g}\Big) = \sum_{a \pmod{g}} \cZ_\chi(u;ac,g) \cD_1\Big(u,\frac{ab}{g}\Big) \end{equation}
can be shown to hold entirely analogously to Lemma~\ref{lem:D2_simple}. Substituting the functional equations for $\cZ$ and $\cR_1$ (i.e., \eqref{eqn:Z_func} and \eqref{eqn:R1_func}) into the right hand side of \eqref{eqn:R2_simple1}, we get
\begin{multline*} \cR_2\Big(u;\chi, \frac{c}{g}\Big) = \chi(g)G(\chi)(qu^2)^{\deg(g)-1} u^{\deg(\ell)} \\ \times \sum_{a\pmod{g}}\Bigg(\sum_{\lambda_2 \in \Fq^\times} \bigg[\sum_{\eta_2 \in \{0,1\}} (-1)^{\eta_2} \frac{e_q\big((\eta_2-1)\lambda_2\big)}{q(1-qu)^{\eta_2}}\bigg] \cD_1\bigg(\frac{1}{qu};\frac{\lambda_2 a}{g}\bigg) \\\times \sum_{\lambda_1 \in \Fq^\times}\chi(-\lambda_1)\bigg[\sum_{\eta_1\in\{0,1\}} (-1)^{\eta_1} \frac{e_q\big((1-\eta_1)\lambda_1^{-1}\big)}{(1-qu)^{\eta_1}}\bigg]  \cZ_{\ov{\chi}}\bigg(\frac{1}{qu};\lambda_1ac,g\bigg) \Bigg).\end{multline*}
Here, we have incorporated Remark~\ref{rem:eta_sum_for_Z} and Remark~\ref{rem:eta_sum_for_R1}. Using \eqref{eqn:R2_simple2}, we can execute the sum over $a \pmod g$ to get 
\begin{multline*} \cR_2\Big(u;\chi, \frac{c}{g}\Big) = \chi(g)G(\chi)(qu^2)^{\deg(g)-1} u^{\deg(\ell)} \sum_{\lambda_1,\lambda_2 \in \Fq^\times} \chi(-\lambda_1) \\ \times \bigg[\sum_{\eta_1,\eta_2 \in \{0,1\}} (-1)^{\eta_1+\eta_2} \frac{e_q\big((1-\eta_1)\lambda_1^{-1}+(\eta_2-1)\lambda_2\big)}{q(1-qu)^{\eta_1+\eta_2}}\bigg] \cR_2\Big(\frac{1}{qu};\ov{\chi},\frac{\lambda_1^{-1}\lambda_2\ov{c}}{g}\Big) .\end{multline*}

Letting $\lambda = \lambda_1^{-1} \lambda_2$, writing the sums over $(\lambda_1,\lambda_2)$ as a sum over $(\lambda,\lambda_1)$, and pushing the sum over $\lambda_1$ inside we find that the right hand side above transforms into
\begin{multline*} \chi(g)G(\chi)(qu^2)^{\deg(g)-1} u^{\deg(\ell)} \sum_{\lambda \in \Fq^\times} \cR_2\Big(\frac{1}{qu};\ov{\chi},\frac{\lambda\ov{c}}{g}\Big)  \bigg[\sum_{\eta_1,\eta_2 \in \{0,1\}} \frac{(-1)^{\eta_1+\eta_2}}{q(1-qu)^{\eta_1+\eta_2}}\bigg] \\ \times \sum_{\lambda_1 \in \Fq^\times} \chi(-\lambda_1) e_q\big((1-\eta_1)\lambda_1^{-1}+(\eta_2-1)\lambda\lambda_1\big).\end{multline*}

To finish the proof of \eqref{eqn:R2_func_coprime}, we need to show that the sum over $\lambda_1$ above is equal to 
\[ \Kl_{\ov{\chi}}\big(\eta_1-1,(1-\eta_2)\lambda\big) = \sum_{\gamma \in \Fq^\times} \ov{\chi}(\gamma) e_q\big((\eta_1-1)\gamma + (1-\eta_2)\lambda \gamma^{-1}\big),\]
which can be seen by substituting $\lambda_1 = -\gamma^{-1}$ and by noting that $\chi(\gamma^{-1}) = \ov{\chi}(\gamma)$; in this way, we find that the sum over $\eta_1,\eta_2$ in the above expression gives $\cB(u;\ov{\chi})$ as desired.

It remains to show the second equality in \eqref{def:B}. Thus,
\begin{equation*} \begin{split} 
\cB_\lambda(u;\chi) 
& = \frac{\Kl_\chi(-1,\lambda)}{q} - \frac{\Kl_\chi(-1,0)}{q(1-qu)} - \frac{\Kl_\chi(0,\lambda)}{q(1-qu)} + \frac{\Kl_{\chi}(0,0)}{q(1-qu)^2} \\
& =  \frac{\Kl_\chi(-1,\lambda)}{q} - \frac{\tau(\chi,-1)}{q(1-qu)} - \frac{\tau(\ov{\chi},\lambda)}{q(1-qu)} + \frac{q-1}{q(1-qu)^2}\1(\chi \text{ is even}), 
\end{split} \end{equation*}
by \eqref{def:tau_f_over_Fq}, \eqref{eqn:Klchi0b}, and \eqref{eqn:Klchi00}.
\end{proof}

\begin{remark}
Notice that if $\chi$ is even, then $\tau(\chi,\lambda) = -1$ for $\lambda \neq 0$. Thus, we find that whenever $\chi$ is even, $\cB_\lambda(u;\chi) = \cA_\lambda(u)$ for every $\lambda \in \Fq^\times$.
\end{remark}

\begin{remark}
The above proposition remains valid when $\ell = 1$. In this case, we see that there is only one primitive character modulo $\ell$ (namely, the constant function $f \mapsto 1$). Thus, $\cR_2(u;1,c/g) = \cD_2(u;c/g)$ and in view of the previous remark, Proposition~\ref{prop:R2_func_coprime} simply becomes Proposition~\ref{prop:D2_func}.
\end{remark}

We finish off this section by establishing a version of Proposition~\ref{prop:R2_func_coprime} in the alternate situation where $\ell$ divides $g$. 

\begin{proposition} \label{prop:R2_func_not_coprime} Let $g \in \cM_{\geqs 1}$, $\ell \in \cM$, and $c\in \Fq[T]$ be such that $\ell \mid g$ and $(c,g) = 1$. Further, let $\chi$ be a primitive Dirichlet character modulo $\ell$. Then, for all $u \in \C$, we have that
\[ \cR_2\Big(u;\chi,\frac{c}{g}\Big) = \ov{\chi}(-c) (qu^2)^{\deg(g) - 1} \sum_{\lambda\in \Fq^\times} \cB_\lambda(u;\chi) \cR_2\Big(\frac{1}{qu};\chi,\frac{\lambda\ov{c}}{g}\Big),\]
where $\cB_\lambda(u;\chi)$ is the same as in \eqref{def:B}. 
\end{proposition}

\begin{proof}

By opening $r_\chi(f) = \sum_{f_1 f_2 = f} \chi(f_2)$ and fibering the sum according to the congruence class of $f_2$ modulo $g$, one finds that
\begin{equation} 
\label{eqn:R2_simple3}
  \cR_2\Big(u;\chi,\frac{c}{g}\Big) = \sum_{a \pmod{g}} \chi(a)  \cZ(u;a,g)\cD_1\Big(u;\frac{ac}{g}\Big).  
\end{equation}  
Similarly, for $(c,g) = 1$,
\begin{equation}
\label{eqn:R2_simple4}
 \chi(\ov{c}) \cR_2\Big(u;\chi,\frac{b\ov{c}}{g}\Big) = \sum_{a \pmod{g}} \chi(a) \cZ(u;ac,g) \cD_1\Big(u;\frac{ab}{g}\Big).  
\end{equation}
Here, $\ov{c}$ is the multiplicative inverse of $c$ modulo $g$ (and hence, also modulo $\ell$). These are analogous to \eqref{eqn:R2_simple1} and \eqref{eqn:R2_simple2}, but depend greatly on the fact that $\ell \mid g$.

Applying the functional equations for $\cD_1$ and $\cZ$ (cf. Propositions~\ref{prop:Z_func}~and~\ref{prop:D1_func}), the right hand side of \eqref{eqn:R2_simple3} becomes
%\begin{multline*} (qu^2)^{\deg(g) - 1} \sum_{a\pmod{g}} \Bigg( \chi(a) \sum_{\lambda_1,\lambda_2 \in \Fq^\times}\bigg[\sum_{\eta_1,\eta_2 \in \{0,1\}} (-1)^{\eta_1+\eta_2} \frac{e_q\big((1-\eta_1)\lambda_1^{-1}+(\eta_2-1)\lambda_2\big)}{q(1-qu)^{\eta_1+\eta_2}} \bigg]\\ \times \cZ\Big(\frac{1}{qu};\lambda_1 ac, g\Big)\cD_1\Big(\frac{1}{qu};\frac{\lambda_2 a}{g}\Big) \Bigg). \end{multline*}
\begin{multline*} (qu^2)^{\deg(g) - 1} \sum_{a\pmod{g}} \Bigg( \chi(a) \sum_{\lambda_1,\lambda_2 \in \Fq^\times}\sum_{\eta_1,\eta_2 \in \{0,1\}} (-1)^{\eta_1+\eta_2} \frac{e_q\big((1-\eta_1)\lambda_1^{-1}+(\eta_2-1)\lambda_2\big)}{q(1-qu)^{\eta_1+\eta_2}} \\ \times \cZ\Big(\frac{1}{qu};\lambda_1 ac, g\Big)\cD_1\Big(\frac{1}{qu};\frac{\lambda_2 a}{g}\Big) \Bigg). \end{multline*}
Executing the sum over $a \bmod g$ using \eqref{eqn:R2_simple4}, we get
%\begin{multline*} \cR_2\Big(u;\chi,\frac{c}{g}\Big) = (qu^2)^{\deg(g) - 1} \sum_{\lambda_1,\lambda_2 \in \Fq^\times}\bigg[\sum_{\eta_1,\eta_2 \in \{0,1\}} (-1)^{\eta_1+\eta_2} \frac{e_q\big((1-\eta_1)\lambda_1^{-1}+(\eta_2-1)\lambda_2\big)}{q(1-qu)^{\eta_1+\eta_2}} \bigg]\\ \times \chi(\lambda_1^{-1} \ov{c}) \cR_2\Big(\frac{1}{qu};\chi,\frac{\lambda_1^{-1} \lambda_2\ov{c}}{g}\Big). \end{multline*}
\begin{multline*} \cR_2\Big(u;\chi,\frac{c}{g}\Big) = (qu^2)^{\deg(g) - 1} \sum_{\lambda_1,\lambda_2 \in \Fq^\times}\sum_{\eta_1,\eta_2 \in \{0,1\}} (-1)^{\eta_1+\eta_2} \\ \times \frac{e_q\big((1-\eta_1)\lambda_1^{-1}+(\eta_2-1)\lambda_2\big)}{q(1-qu)^{\eta_1+\eta_2}} \times \chi(\lambda_1^{-1} \ov{c}) \cR_2\Big(\frac{1}{qu};\chi,\frac{\lambda_1^{-1} \lambda_2\ov{c}}{g}\Big). \end{multline*}
Now, we follow the usual strategy of letting $\lambda = \lambda_1^{-1}\lambda_2$ and replacing sums over $(\lambda_1,\lambda_2)$ by $(\lambda,\lambda_1)$. In particular, we see that the sum over $\lambda_1$ is given by 
\[ \sum_{\lambda_1 \in \Fq^\times} \chi(\lambda_1^{-1}) e_q\big((1-\eta_1)\lambda_1^{-1}+(\eta_2-1)\lambda\lambda_1\big) = \chi(-1) \Kl_\chi\big(\eta_1-1,(1-\eta_2)\lambda), \]
by making the substitution $\lambda_1 = -\gamma^{-1}$ and recalling \eqref{def:chi_Kl_over_Fq}. Thus, we find
\begin{equation*} \cR_2\Big(u;\chi,\frac{c}{g}\Big) = \ov{\chi}(-c) (qu^2)^{\deg(g) - 1} \sum_{\lambda\in Fq^\times} \cB_\lambda(u;\chi) \cR_2\Big(\frac{1}{qu};\chi,\frac{\lambda\ov{c}}{g}\Big), \end{equation*}
as desired.
\end{proof}

\begin{remark} As with Proposition~\ref{prop:R2_func_coprime}, one can set $\ell = 1$ in Proposition~\ref{prop:R2_func_not_coprime} in which case we get back Proposition~\ref{prop:D2_func}.
\end{remark}

\section{Voronoi summation formulae} \label{sec:voronoi}

In this section, we shall derive the requisite Voronoi summation formulae from the functional equations we found in \S\ref{sec:func}. The reader may wish to compare the formulae to Theorems~4.13 and 4.14 in \cite[Chapter 4]{iwanieckowalski}. 

\begin{proposition} \label{prop:Rvoronoi_coprime}
Let $g \in \cM_{\leqs n/2}$,  $\ell \in \cM$, and $c \in \Fq[T]$ be such that $(\ell c,g) = 1$ and let $r_\chi$ be as in \eqref{def:rchi} with $\chi$ a primitive Dirichlet character modulo $\ell$. Then,
\begin{multline} \label{eqn:Rvoronoi_coprime} \sum_{f \in \cM_n} r_\chi(f)e\Big(\frac{cf}{g}\Big) = \scrM_n\\ + \frac{q^n}{|g\ell|}\chi(g)G(\chi) \sum_{\lambda \in \Fq^\times} \sum_{\eta_1,\eta_2 \in \{0,1\}} \sum_{0\leqs k \leqs \mu} b_{\mu-k}(\ov \chi) \sum_{f \in \cM_k} r_{\ov{\chi}}(f)e\Big(\frac{\lambda\ov{c}f}{g}\Big) \end{multline}
where $\mu = 2\deg(g) +\deg(\ell) - n - 2$,
\begin{equation} \label{def:Rvoronoi_main_term} \scrM_n = - \res_{u = 1/q} \bigg(\frac{\cR_2(u;\chi,c/g)}{u^{n+1}}\bigg),\end{equation}
$b_k = b_k(\chi;\lambda,\eta_1,\eta_2)$ is given by
\begin{equation}
    b_k(\chi;\lambda,\eta_1,\eta_2) = (-1)^{k-\eta_1-\eta_2}\Kl_\chi\big(\eta_1-1,(1-\eta_2)\lambda\big) \times \Big[\1_{k \geqs \eta_1 + \eta_2} \binom{-\eta_1 - \eta_2}{k - \eta_1 - \eta_2}\Big].\label{bk}
\end{equation}
%and $\scrD_{k} = \scrD_k(\chi,c,g;\lambda)$ is given by\[ \scrD_{k} = \frac{q^n}{|g\ell|}\chi(g)G(\chi) \sum_{f \in \cM_k} r_{\ov{\chi}}(f)e\Big(\frac{\lambda\ov{c}f}{g}\Big).\]
Notice that when $\mu< 0$, the second term in \eqref{eqn:Rvoronoi_coprime} is simply 0. 

\end{proposition} 

\begin{proof} 
Let 
\[ \scrI(\sigma,n) = \frac{1}{2\pi i}\oint_{|u|=q^{-\sigma}} \cR_2\Big(u;\chi,\frac{c}{g}\Big) \frac{du}{u^{n+1}}, \] 
where the contour $|u| = q^{-\sigma}$ is being traversed counter-clockwise. Note that the integrand has poles only at $u = 1/q$ (coming from $\cR_2$) and $u = 0$, whence two applications of Cauchy's residue theorem give that
\[ \scrI(-1,n) = - \scrM_n + \scrI(2,n). \]
On the other hand, Perron's formula \eqref{perron} tells us that
\[ \scrI(2,n) = \sum_{f \in \cM_n} r_\chi(f)e\Big(\frac{cf}{g}\Big). \]
Thus, it suffices to show that 
\[ \scrI(-1,n) = \frac{q^n}{|g\ell|}\chi(g)G(\chi) \sum_{\lambda \in \Fq^\times} \sum_{\eta_1,\eta_2 \in \{0,1\}} \sum_{0\leqs k \leqs \mu} b_{\mu-k}(\ov \chi) \sum_{f \in \cM_k} r_{\ov{\chi}}(f)e\Big(\frac{\lambda\ov{c}f}{g}\Big). \]

To do this, we apply Proposition~\ref{prop:R2_func_coprime} to the integrand of $\scrI(-1,n)$. On doing so, and pulling the sum over $\lambda$ out, we obtain that
\[ \scrI(-1,n) = \frac{q^{\deg(g) - 1}\chi(g)G(\chi)}{2\pi i}\sum_{\lambda \in \Fq^\times} \oint_{|u|=q} \cB_\lambda(u;\ov{\chi})\cR_2\Big(\frac{1}{qu};\ov{\chi},\frac{\lambda\ov{c}}{g}\Big)\frac{du}{u^{n-2\deg(g)-\deg(\ell) + 3}}. \]
We now substitute $u = 1/(qv)$ and recall the definition of $\mu$, finding that
\begin{multline*} \scrI(-1,n) = q^{n - \deg(g) - \deg(\ell) + 1}\chi(g)G(\chi) \\ \times \frac{1}{2\pi i} \sum_{\lambda \in \Fq^\times} \oint_{|v|=q^{-2}} \cB_\lambda\Big(\frac{1}{qv};\ov{\chi}\Big)\cR_2\Big(v;\ov{\chi},\frac{\lambda\ov{c}}{g}\Big)\frac{dv}{v^{\mu+1}}, \end{multline*}
where we have consolidated the factors of $q$ coming from the change of variables together with the existing factor $q^{\deg(g) - 1}$. Now, by substituting $u = 1/(qv)$ in \eqref{def:B} and simplifying, we obtain
\begin{equation*} \begin{split} \cB_\lambda\Big(\frac{1}{qv};\ov{\chi}\Big) & = \sum_{\eta_1,\eta_2 \in \{0,1\}} \frac{\Kl_{\overline{\chi}}\big(\eta_1-1,(1-\eta_2)\lambda\big)}{q} \times \frac{v^{\eta_1 + \eta_2}}{(1-v)^{\eta_1 + \eta_2}} \\ & = \frac{1}{q}\sum_{\eta_1,\eta_2 \in \{0,1\}} \sum_{k=0}^\infty b_k (\ov \chi) v^k. \end{split} \end{equation*} 
The last step follows from the binomial theorem. Substituting this above, we get 
\[ \scrI(-1,n) = \frac{q^n}{|g\ell|}\chi(g)G(\chi) \sum_{\lambda\in\Fq^\times}\sum_{\eta_1,\eta_2\in\{0,1\}}\sum_{k=0}^\infty \frac{b_k(\ov \chi) }{2\pi i}\oint_{|v|=q^{-2}} \cR_2\Big(v;\ov{\chi},\frac{\lambda \ov{c}}{g}\Big) \frac{dv}{v^{\mu - k + 1}}. \]
The inner integral vanishes if $k > \mu$, whence the sum over $k$ can be truncated to $k \leqs \mu$. Then, using Perron's formula to evaluate the integral, we get 
\[ \scrI(-1,n) = \sum_{\lambda\in\Fq^\times}\sum_{\eta_1,\eta_2\in\{0,1\}}\sum_{0\leqs k \leqs \mu} b_k(\ov \chi) \bigg[\frac{q^n}{|g\ell|}\chi(g)G(\chi)\sum_{f\in \cM_{\mu-k}} r_{\ov{\chi}}(f)e\Big(\frac{\lambda\ov{c}f}{g}\Big)\bigg], \]
completing the proof.
\end{proof}

We will need an explicit evaluation of the term $\scrM_n$ in the above proposition. 
\begin{lemma}
\label{residue}
Under the same assumptions as Proposition~\ref{prop:Rvoronoi_coprime} and the additional assumption that $\chi$ is nonprincipal (equivalently, that $\ell \in \cM_{\geqs 1}$) we have that
\[ \scrM_{n} = \frac{q^n}{|g|}\chi(g)L(1,\chi), \]
where $\scrM_n$ is as in \eqref{def:Rvoronoi_main_term} and $L(s,\chi)$ is the Dirichlet $L$-function associated to $\chi$.
\end{lemma}
\begin{proof}
Collecting the terms $f_1, f_2$ according to their congruence classes $a$ and $b$ modulo $g$ and applying \eqref{eqn: Z simple}, we obtain
     \begin{align}
         \mathcal{R}_2\Big(u;\chi,\frac{c}{g}\Big)& = \sum_{f_1, f_2 \in \mathcal{M}} \chi(f_1) e \Big( 
 \frac{cf_1f_2}{g} \Big) u^{\deg(f_1)+\deg(f_2)} \nonumber \\
  &= \sum_{a, b \pmod g} e \Big( \frac{abc}{g} \Big) \mathcal{Z}_{\chi}(u;b,g) \mathcal{Z}(u;a,g)\nonumber\\
 &= \sum_{a, b \pmod g} e \Big( \frac{abc}{g} \Big) \mathcal{Z}_{\chi}(u;b,g) \Big[  \frac{u^{\deg(g)}}{1-qu}+  \1_{a \in \mathcal{M}}u^{\deg(a)}\Big]. \label{zchi}
     \end{align}
If we now focus on the term $\mathcal{Z}_{\chi}(u;b,g)$ and again collect terms according to their congruences modulo $g \ell$, we have, by applying \eqref{eqn: Z simple} again, 
     \begin{align}
         \mathcal{Z}_{\chi}(u;b,g) &= \sum_{\substack{r \pmod{g\ell} \\ r \equiv b \pmod g}} \chi(r) \sum_{f \equiv r \pmod{g \ell}} u^{\deg(f)}\nonumber \\
         &= \sum_{\substack{r \pmod{g\ell} \\ r \equiv b \pmod g}} \chi(r) \Big[  \frac{u^{\deg(g \ell)}}{1-qu} +\1_{r \in \mathcal{M}} u^{\deg(r)} \Big]\nonumber \\
         &= \frac{u^{\deg(g \ell)}}{1-qu} \sum_{a \pmod \ell} \chi(a) \sum_{\substack{r \pmod{g \ell}\\ r \equiv b \pmod g \\ r \equiv a \pmod{\ell}}} 1 + \sum_{\substack{r \pmod{g\ell} \\ r \equiv b \pmod g\\ r \in \mathcal{M}}} \chi(r) u^{\deg(r)} \label{eq:Zchi0}\\
         &= \sum_{\substack{r \pmod{g\ell} \\ r \equiv b \pmod g\\ r \in \mathcal{M}}} \chi(r) u^{\deg(r)}. \label{eq:Zchi}
     \end{align}
Note that in equation \eqref{eq:Zchi0} we used the fact      that there is a unique solution $r \pmod{g \ell}$ such that $r \equiv b \pmod g$ and $r \equiv a \pmod \ell$, since $(\ell,g)=1$, which implies that the sum over $r$ is equal to $1$, and then the sum over $a \pmod{\ell}$ vanishes by orthogonality of characters. Replacing \eqref{eq:Zchi} in \eqref{zchi}, dividing by $u^{n+1}$, and computing the residue at $u=1/q$ yields 
     \begin{align*}
\scrM_n = \frac{q^n}{|g|} \sum_{a,b \pmod g} e\Big(  \frac{abc}{g}\Big) \sum_{\substack{r \pmod{g\ell} \\ r \equiv b \pmod g\\ r \in \mathcal{M}}}  \frac{\chi(r)}{|r|}. 
     \end{align*}
It remains to evaluate the above sums. We have
  \begin{align*}
      \sum_{a,b \pmod g} & e\Big(  \frac{abc}{g}\Big) \sum_{\substack{r \pmod{g\ell} \\ r \equiv b \pmod g\\ r \in \mathcal{M}}}  \frac{\chi(r)}{|r|} = \sum_{\substack{r \pmod{g \ell} \\ r \in \mathcal{M}}} \frac{\chi(r)}{|r|} \sum_{a \pmod g}
 e \Big(  \frac{acr}{g} \Big)\\
 &= |g| \sum_{\substack{r \pmod{g \ell} \\ r \in \mathcal{M}\\ r \equiv 0 \pmod g}} \frac{\chi(r)}{|r|} =  \chi(g) L(1,\chi).
 \end{align*}  
 Combining the last two equations above, the conclusion follows.
 \end{proof}

\begin{lemma}
\label{residue_d}
Under the same assumptions as Proposition~\ref{prop:Rvoronoi_coprime} and the additional assumption that $\ell = 1$ (equivalently $\chi$ is principal), we have that
\[ \scrM_{n} = \frac{q^n}{|g|} (n+1- 2\deg(g)), \]
where $\scrM_n$ is as in \eqref{def:Rvoronoi_main_term}.
\end{lemma}
\begin{proof}
Recalling that $ \mathcal{R}_2(u;\chi,c/g)=\mathcal{D}_2(u;c/g)$ when $\ell=1$ and applying \eqref{eqn:D2_simple},  we have that
\begin{align*}
    \mathcal{D}_2 \Big( u; \frac{c}{g}\Big) &= \sum_{a \pmod g} \mathcal{Z}(u;a,g) \mathcal{D}_1 \Big(  u ; \frac{ac}{g}\Big) \\
    &= \sum_{a \pmod g} \Big[ \frac{u^{\deg(g)}}{1-qu}+ \1_{a\in \mathcal{M}} u^{\deg(a)} \Big] \mathcal{D}_1\Big(  u ; \frac{ac}{g}\Big) \\
    &= \frac{u^{\deg(g)}}{(1-qu)^2} + \frac{u^{\deg(g)}}{1-qu}\sum_{\substack{a \pmod g \\ a \neq 0}}  \mathcal{D}_1\Big(  u ; \frac{ac}{g}\Big) + \sum_{\substack{ a \pmod g \\ a \in \mathcal{M}}} u^{\deg(a)} \mathcal{D}_1\Big(  u ; \frac{ac}{g}\Big)\\
    &= \frac{u^{\deg(g)}}{(1-qu)^2} + \frac{u^{\deg(g)}}{1-qu}\sum_{\substack{a \pmod g \\ a \neq 0}}  \mathcal{D}_1\Big(  u ; \frac{a}{g}\Big) + \sum_{\substack{ a \pmod g \\ a \in \mathcal{M}}} u^{\deg(a)} \mathcal{D}_1\Big(  u ; \frac{ac}{g}\Big),
\end{align*}
where we used Lemma \ref{lem: DZ simple} and we made a change of variables in the last line. Using Lemma \ref{lem: DZ simple} again, it follows that $\mathcal{D}_1(u;ac/g)$ is a polynomial in $u$, so the third term in the expression above is a polynomial and there is no pole at $u=1/q$ in that case. It follows that
\begin{align}
    \scrM_n &=- \res_{u = 1/q} \bigg( \frac{u^{\deg(g)}}{(1-qu)^2} + \frac{u^{\deg(g)}}{1-qu}\sum_{\substack{a \pmod g \\ a \neq 0}}  \mathcal{D}_1\Big(  u ; \frac{a}{g}\Big)\bigg) \frac{1}{u^{n+1}} \nonumber\\
    &= ( n+1-\deg(g)) \frac{q^n}{|g|} + \frac{q^n}{|g|} \sum_{\substack{a \pmod g \\ a \neq 0}}  \mathcal{D}_1\Big(  \frac{1}{q}; \frac{a}{g}\Big).\label{eq:Mn}
\end{align}
Now using Lemma \ref{lem: DZ simple} again,
\begin{align*}
 \sum_{\substack{a \pmod g \\ a \neq 0}}  \mathcal{D}_1\Big(  \frac{1}{q}; \frac{a}{g}\Big)  & = \sum_{\substack{a \pmod g \\ a \neq 0}}  \Big(  e_q (\sgn(a)) + \deg(g) - \deg(a)-1\Big) \\
 &= -\sum_{k=0}^{\deg(g)-1} q^k +(\deg(g)-1) (|g|-1) - (q-1) \sum_{k=0}^{\deg(g)-1}  k q^k\\
 &= - \deg(g).
\end{align*}
Combining the above equation with \eqref{eq:Mn}, it follows that
\begin{align*}
\scrM_n  = \frac{q^n}{|g|} (n+1-2\deg(g)).
\end{align*}

\end{proof}

We also have the following Proposition.

\begin{proposition} \label{prop:Rvoronoi_notcoprime}
Let $g \in \cM_{\geqs 1}$, $\ell \in \cM$ such that $\ell\not = 1$ and  $\ell|g$, and $c \in \Fq[T]$ be such that $(c,g) = 1$ and let $r_\chi$ be as in \eqref{def:rchi} with $\chi$ a primitive Dirichlet character modulo $\ell$. Then,
\begin{multline*} %\label{eqn:Rvoronoi_divides} 
\sum_{f \in \cM_n} r_\chi(f)e\Big(\frac{cf}{g}\Big) = \frac{q^n}{|g|}G(\chi) \ov \chi(c) L(1, \ov \chi)\\ + \frac{q^{n}}{|g|} \ov \chi(-c)  \sum_{\lambda\in\Fq^\times}\sum_{\eta_1,\eta_2\in\{0,1\}} \sum_{0 \leqs k \leqs \nu} b_{\nu-k}(\chi) \sum_{f \in \mathcal{M}_{k}} r_{\chi}(f) e \Big(  \frac{ \lambda \ov c f}{g}\Big),
\end{multline*}
where $\nu=2\deg(g) -n-2$, and the coefficients $b_k(\chi)$ are defined in equation \eqref{bk}.
\end{proposition} 
\begin{proof}
    We proceed as in the proof of Proposition \ref{prop:Rvoronoi_coprime}, and we get that
    \begin{align}
     \sum_{f \in \cM_n} r_\chi(f)e\Big(\frac{cf}{g}\Big) = \scrM_n + \scrI(-1,n),   \label{vr} 
    \end{align}
    where, as before,
   \[\scrM_n = - \res_{u = 1/q} \bigg(\frac{\cR_2(u;\chi,c/g)}{u^{n+1}}\bigg),\] and  
   \[  \scrI(-1,n) = \frac{1}{2\pi i}\oint_{|u|=q} \cR_2\Big(u;\chi,\frac{c}{g}\Big) \frac{du}{u^{n+1}} .\]
   Using Proposition \ref{prop:R2_func_not_coprime} in the integrand above, we have that
    \begin{align*}
    \scrI(-1,n) = \frac{1}{2\pi i}\oint_{|u|=q}      \ov{\chi}(-c) (qu^2)^{\deg(g) - 1} \sum_{\lambda\in \Fq^\times} \cB_\lambda(u;\chi) \cR_2\Big(\frac{1}{qu};\chi,\frac{\lambda\ov{c}}{g}\Big) \frac{du}{u^{n+1}}.
    \end{align*}
    Making the change of variables $u=1/(qv)$, we get that 
    \begin{align*}
 \scrI(-1,n) = q^{n+1-\deg(g)} \ov \chi(-c)   \frac{1}{2 \pi i}  \sum_{\lambda\in \Fq^\times}\oint_{|v|=q^{-2}} \cB_\lambda \Big(\frac{1}{qv};\chi\Big) \cR_2\Big(v;\chi,\frac{\lambda\ov{c}}{g}\Big) \frac{dv}{v^{\nu+1}},       
    \end{align*}
    where $\nu=2\deg(g)-n-2$. 
Similarly as before, we have
\[\cB_\lambda\Big(\frac{1}{qv};\chi \Big)=\frac{1}{q} \sum_{\eta_1,\eta_2 \in \{0,1\}} \sum_{k=0}^\infty b_k(\chi) v^k,\] so
\begin{align*}
  \scrI(-1,n) = q^{n+1-\deg(g)} \ov \chi(-c)  \sum_{\lambda\in\Fq^\times}\sum_{\eta_1,\eta_2\in\{0,1\}}\sum_{k=0}^\infty \frac{b_k(\chi)}{2\pi i}\oint_{|v|=q^{-2}} \cR_2\Big(v;\chi,\frac{\lambda \ov{c}}{g}\Big) \frac{dv}{v^{\nu - k + 1}}.
\end{align*}
Note that the integral above vanishes if $k \geqs \nu+1$, and using Perron's formula \eqref{perron} it follows that 
\begin{align}
  \scrI(-1,n) = q^{n+1-\deg(g)} \ov \chi(-c)  \sum_{\lambda\in\Fq^\times}\sum_{\eta_1,\eta_2\in\{0,1\}} \sum_{0 \leqs k \leqs \nu} b_{\nu-k}(\chi) \sum_{f \in \mathcal{M}_{k}} r_{\chi}(f) e \Big(  \frac{ \lambda \ov c f}{g}\Big). \label{j-1}  
\end{align}
Now we focus on the evaluation of $\scrM_n$.
  We rewrite
  \begin{align*}
      \mathcal{R}_2 \Big( u ;\chi, \frac{c}{g} \Big) &= \sum_{a , b \pmod g} \sum_{\substack{f_1 \equiv a \pmod g \\ f_2 \equiv b \pmod g}} \chi(f_1) e \Big( \frac{abc}{g}\Big) u^{\deg(f_1)+\deg(f_2)}  \\
      &= \sum_{a,b \pmod g} \chi(a) e \Big(  \frac{abc}{g} \Big) \Big[  \frac{u^{\deg(g)}}{1-qu}+ \1_{a \in \mathcal{M}} u^{\deg(a)} \Big]\Big[  \frac{u^{\deg(g)}}{1-qu} + \1_{b \in \mathcal{M}} u^{\deg(b)} \Big],
  \end{align*}
where we used the fact that $\ell|g$ and Lemma \ref{lem: DZ simple}. Computing the residue at $u=1/q$, it follows that
\begin{align}
    \scrM_n &= -\frac{q^n(2\deg(g)-n-1) }{|g|^2} \sum_{a , b \pmod g} \chi(a) e\Big( \frac{abc}{g}\Big)+\frac{q^n}{|g|}\sum_{\substack{a,b \pmod g \\ a \in \mathcal{M}}} \frac{\chi(a)}{|a|} e \Big(\frac{abc}{g} \Big) \nonumber  \\
    &+\frac{q^n}{|g|}\sum_{\substack{a,b \pmod g\\ b \in\mathcal{M}}} \frac{\chi(a)}{|b|} e \Big(\frac{abc}{g} \Big) . \label{mn}
\end{align}

 Now note that
 \[ \sum_{a,b \pmod g} \chi(a) e \Big(  \frac{abc}{g}\Big) = \sum_{a \pmod g} \chi(a) |g| \1_{a \equiv 0 \pmod g} = 0,\] where we again used the fact that $\ell|g$.
 
 We similarly get that the second term in \eqref{mn} vanishes. For the third term above, writing $a \equiv \alpha \pmod{\ell}$ and then $b=b_1 (g/\ell)$, we have 
 \begin{align*}
    \sum_{\substack{a , b \pmod g \\ b \in \mathcal{M}}} & \frac{\chi(a)}{|b|}  e \Big(  \frac{abc}{g}\Big) = \sum_{\alpha \pmod \ell} \chi(\alpha) \sum_{\substack{b \pmod g \\ b \in \mathcal{M}}} \frac{1}{|b|}  e \Big(  \frac{bc \alpha}{g} \Big)\sum_{a_1 \pmod{g/\ell}} e \Big( \frac{a_1 b   c}{g/\ell} \Big)  \\
    &=\Big|  \frac{g}{\ell} \Big| \sum_{\alpha \pmod \ell} \chi(\alpha) \sum_{\substack{b \pmod g \\ b \in \mathcal{M} \\ b \equiv 0 \pmod{g/\ell}}} \frac{1}{|b|}  e \Big(  \frac{bc \alpha}{g}\Big)  \\
    &= \sum_{\alpha \pmod \ell} \chi(\alpha) \sum_{\substack{b_1 \pmod{\ell} \\ b_1 \in \mathcal{M}}} \frac{1}{|b_1|} e \Big(  \frac{b_1 c \alpha }{\ell}\Big) = \sum_{\substack{b_1 \pmod{\ell} \\ b_1 \in \mathcal{M}}} \frac{1}{|b_1|} \sum_{\alpha \pmod \ell} \chi(\alpha)  e \Big(  \frac{b_1 c \alpha }{\ell}\Big)\\
    &= G(\chi) \ov \chi(c) \sum_{\substack{b_1 \pmod{\ell} \\ b_1 \in \mathcal{M}}} \frac{\ov\chi(b_1)}{|b_1|} = G(\chi) \ov \chi(c) L(1, \ov \chi).
 \end{align*}
 Using the two equations above in \eqref{mn}, we get that
 \[\scrM_n = \frac{q^n}{|g|}G(\chi) \ov \chi(c) L(1, \ov \chi). \]
 Using the above and equations \eqref{vr} and \eqref{j-1}, the conclusion follows.
 \end{proof}

 \section{Setting up the proof when \texorpdfstring{$\ell_2 \neq 1$}{ell2 is not 1}}
 \label{sec:setup}

In this section, we assume that $\ell_2 \neq 1$ (so $\chi_2$ is a non-principal character modulo $\ell_2$). Let $d = \max \{m,n\}$, where $m=\deg(h)$. 

Observe that
\begin{equation*} %\label{alt div}
r_{\chi_2}(f+h) = \sum_{\substack{g \in \mathcal{M}_{\leqs [d/2]}\\
g \mid (f+h) }} \chi_2(g) + \sum_{\substack{g \in \mathcal{M}_{\leqs [(d-1)/2]}\\
g \mid (f+h)}} \chi_2\Big(\frac{f+h}{g}\Big). \end{equation*}
Using the above in the computation of the correlation, we find that
\[ \sum_{f \in \cM_n} r_{\chi_1}(f) r_{\chi_2}(f+h) = S_1 (\chi_1,\chi_2) + S_2 (\chi_1,\chi_2), \]
where 
\[S_1(\chi_1,\chi_2) = \sum_{f \in \mathcal{M}_n} r_{\chi_1}(f)\sum_{\substack{g \in \mathcal{M}_{\leqs [d/2]}\\g \mid (f+h)}} \chi_2(g), \] and
\[S_2(\chi_1,\chi_2) =\sum_{f \in \mathcal{M}_n} r_{\chi_1}(f) \sum_{\substack{g \in \mathcal{M}_{\leqs [(d-1)/2]}\\ g \mid (f+h)}} \chi_2\Big(\frac{f+h}{g}\Big). \]
From now on, for ease of notation, we will denote $S_i(\chi_1,\chi_2)$ by $S_i$ for $i=1,2$.

We rewrite 
\[ S_1 = \sum_{g \in \mathcal{M}_{\leqs  [d/2]}}  \chi_2(g) \sum_{\substack{f \in \cM_n\\ f \equiv -h \pmod g}} r_{\chi_1}(f).\] 
Detecting the congruence condition using additive characters modulo $g$, we get that 
\begin{equation*}
S_1 =\sum_{g \in \mathcal{M}_{\leqs  [d/2]}}  \frac{\chi_2(g)}{|g|} \sum_{b \pmod g} e \Big( \frac{bh}{g} \Big) \sum_{f \in \mathcal{M}_n} r_{\chi_1}(f) e \Big(  \frac{bf}{g}\Big).
\end{equation*}
We further rewrite $S_1$ as
\begin{equation}
    \label{s1initial}
    S_1 =  \sum_{g \in \mathcal{M}_{\leqs  [d/2]}}  \frac{\chi_2(g)}{|g|}  \sum_{g_1|g} \, \copsum_{b_1 \pmod{g_1}} e \Big( \frac{b_1h}{g_1} \Big) \sum_{f \in \mathcal{M}_n} r_{\chi_1}(f) e \Big(  \frac{b_1f}{g_1}\Big).
\end{equation}

For $S_2$, we interchange the $f$ and $g$ sums as before, and arrange the terms according to the congruence class modulo $\ell_2$ that $(f+h)/g$ takes. In particular, $(f+h)/g \equiv a \pmod {\ell_2}$ implies that $f \equiv -h + ag \pmod {g\ell_2}$.  This latter congruence entails the fact that $g \mid (f+h)$, so we omit the constraint. In doing so, we find that
\[ S_2 = \sum_{a \pmod{\ell_2}} \chi_2(a) \sum_{g \in \mathcal{M}_{\leqs [(d-1)/2]}} \sum_{\substack{f \in \cM_n \\ f \equiv -h + ag \pmod{g\ell_2}}} r_{\chi_1}(f). \]
Detecting the congruence using additive characters modulo $g\ell_2$, we find that
\[ S_2 = \frac{1}{|\ell_2|} \sum_{a \pmod{\ell_2}} \chi_2(a) \sum_{g \in \mathcal{M}_{\leqs [(d-1)/2]}} \frac{1}{|g|} \sum_{b \pmod{g\ell_2}} e\Big(\frac{bh -abg}{g\ell_2}\Big) \sum_{f \in \cM_n} r_{\chi_1}(f) e\Big(\frac{bf}{g\ell_2}\Big). \]
Focusing only on the sum over $a$, we observe that we have
\[\sum_{a\pmod{\ell_2}} \chi_2(a) e\Big(\frac{-ab}{\ell_2}
\Big) =  \ov{\chi_2}(-b) G(\chi_2). \]
  Note that this identity is valid even when $\ell_2 \mid b$.

Thus, we have obtained,
\begin{equation*} 
     S_2 = \frac{G(\chi_2)}{|\ell_2|}  \sum_{g \in \mathcal{M}_{\leqs  [(d-1)/2]}}  \frac{1}{|g|} \sum_{b \pmod{g\ell_2}} \ov{\chi_2}(-b) e\Big(\frac{bh}{g\ell_2}\Big) \sum_{f \in \cM_n} r_{\chi_1}(f) e\Big(\frac{bf}{g\ell_2}\Big). 
     \end{equation*}
    
To apply the functional equation of $\cR_2$ to the innermost sum, we need to reduce $\frac{b}{g\ell_2}$. To do this, first observe that $\chi_2(-b)$ ensures $\ell_2 \nmid b$. Thus, we must have
\[ \frac{b}{g\ell_2} = \frac{b_1}{g_1\ell_2}, \]
where $\ell_2 \nmid \frac{g}{g_1}$ (and the latter fraction is reduced). Further, for every reduced fraction of this type, there is exactly one choice of $b$ for which it occurs. Thus, we may rewrite the sum over $b\pmod{g\ell_2}$ as a sum over $g_1 \mid g$ satisfying $\ell_2 \nmid \frac{g}{g_1}$ and a sum over reduced residues $b_1 \pmod{g_1\ell_2}$. 

Upon making these substitutions, we find that
\[ S_2 =  \frac{G(\chi_2)}{|\ell_2|}  \sum_{g \in \mathcal{M}_{\leqs  [(d-1)/2]}}  \frac{1}{|g|} \sum_{\substack{g_1 \mid g\\\ell_2 \nmid \frac{g}{g_1}}} \copsum_{b_1 \pmod{g_1\ell_2}} \ov{\chi_2}\Big(-\frac{b_1g}{g_1}\Big) e\Big(\frac{b_1h}{g_1\ell_2}\Big) \sum_{f \in \cM_n} r_{\chi_1}(f) e\Big(\frac{b_1f}{g_1\ell_2}\Big). \]
By multiplicativity, we have $\chi_2(-b_1g/g_1) = \chi_2(-b_1) \chi_2(g/g_1)$. In particular, we see that we can drop the constraint $\ell_2 \nmid \frac{g}{g_1}$ from the sum over $g_1$. This gives
\begin{multline}
\label{s2}
S_2 =  \frac{G(\chi_2)}{|\ell_2|}  \sum_{g \in \mathcal{M}_{\leqs  [(d-1)/2]}}  \frac{1}{|g|} \sum_{g_1 \mid g} \ov{\chi_2}\Big(\frac{g}{g_1}\Big)\copsum_{b_1 \pmod{g_1\ell_2}} \ov{\chi_2}(-b_1) e\Big(\frac{b_1h}{g_1\ell_2}\Big) \\\times \sum_{f \in \cM_n} r_{\chi_1}(f) e\Big(\frac{b_1f}{g_1\ell_2}\Big). 
\end{multline}

\subsection{The case \texorpdfstring{$\ell_1 \neq 1$}{ell1 is not 1}}
Here, we assume that $\ell_1 \neq 1$. We first rewrite $S_1$ as follows. In equation \eqref{s1initial}, we apply the Voronoi summation formulas as in Propositions \ref{prop:Rvoronoi_coprime} and \ref{prop:Rvoronoi_notcoprime}. 
 In this case, we write
\[S_1 = S_{11}+S_{12},\] where $S_{11}$ corresponds to the sum over $\ell_1 \nmid g_1$ and $S_{12}$ corresponds to the sum over $\ell_1 | g_1$. Applying the Voronoi summation formula from Proposition \ref{prop:Rvoronoi_coprime} and Lemma \ref{residue}, we write
\begin{align*}
    S_{11} &= M_{11}+E_{11},
\end{align*}
where 
\begin{align}
\label{m11}
    M_{11} &= q^n L(1,\chi_1) \sum_{g \in \mathcal{M}_{\leqs  [d/2]}}  \frac{\chi_2(g)}{|g|}  \sum_{\substack{g_1|g}} \frac{\chi_1(g_1)}{|g_1|} \copsum_{b_1 \pmod{g_1}} e \Big( \frac{b_1h}{g_1} \Big) ,
\end{align}
and
\begin{align}
\label{e11}
    E_{11} &=\frac{q^nG(\chi_1)}{|\ell_1|}  \sum_{g \in \mathcal{M}_{\leqs [d/2]}}  \frac{\chi_2(g)}{|g|}  \sum_{\substack{g_1|g}} \frac{\chi_1(g_1)}{|g_1|}\, \copsum_{b_1 \pmod{g_1}} e \Big( \frac{b_1h}{g_1} \Big) \nonumber\\
    & \times \sum_{\lambda \in \Fq^\times} \sum_{\eta_1,\eta_2 \in \{0,1\}} \sum_{0\leqs k \leqs \mu}b_{\mu-k}(\ov \chi_1) \sum_{f \in \mathcal{M}_k} r_{\ov \chi_1}(f) e \Big(  \frac{\lambda f \ov b_1}{g_1}\Big),
\end{align}
where $\mu= 2\deg(g_1) + \deg(\ell_1)-n-2$.
\begin{remark}
\label{rmk1}
The error term $E_{11}$ vanishes in the following cases:
\begin{itemize}
\item when $m \leqs n$, $n$ is even, and $\deg(\ell_1)=1$;
\item when $m \leqs n$, $n$ is odd, and $\deg(\ell_1)<3$;
\item when  $m=n+1$, $n$ is even, and $\deg(\ell_1)=1$;
\end{itemize}
where we recall that $\deg(\ell_1)> 0$ since $\ell_1\not = 1$. 
\end{remark}
Applying the Voronoi summation formula from Proposition \ref{prop:Rvoronoi_notcoprime}, we have 
\begin{equation*}
S_{12}=M_{12}+E_{12},
\end{equation*}
where 
\begin{align}
\label{m12}
    M_{12} &=  q^n G(\chi_1) L(1,\ov \chi_1)\sum_{g \in \mathcal{M}_{\leqs [d/2]}}  \frac{\chi_2(g)}{|g|}  \sum_{\substack{g_1|g\\ \ell_1|g_1}} \frac{1}{|g_1|} \, \copsum_{b_1 \pmod{g_1}}  \ov{\chi_1}(b_1) e \Big( \frac{b_1h}{g_1} \Big) ,
\end{align}
and
\begin{align}
    \label{e12}
    E_{12} &= q^{n+2} \sum_{g \in \mathcal{M}_{\leqs [d/2]}}  \frac{\chi_2(g)}{|g|}  \sum_{\substack{g_1|g \\ \ell_1|g_1}} \frac{1}{|g_1|} \copsum_{b_1 \pmod{g_1}} e \Big( \frac{b_1h}{g_1} \Big) \ov \chi_1(-b_1) \nonumber\\
    &\times \sum_{\lambda\in\Fq^\times}\sum_{\eta_1,\eta_2\in\{0,1\}}\sum_{0\leqs k \leqs \nu} b_{\nu-k}(\chi_1) \sum_{f \in \mathcal{M}_{k}} r_{\chi_1}(f) e \Big(  \frac{ \lambda \ov b_1 f}{g_1}\Big),
\end{align}
where $\nu=2\deg(g_1)-n-2$.

\begin{remark}
\label{rmk2} The error term $E_{12}$ vanishes when $m<n+2$ or when $m=n+2$ and $n$ is odd. 
    \end{remark}

\begin{remark} 
\label{remarkell} 
If $\ell_1=\ell_2=\ell$, then  $\ell\mid g_1$ implies that $\ell\mid g$ and we get $M_{12}=E_{12}=0$.
\end{remark}

Now we focus on $S_2$ given by equation \eqref{s2}. We will use the Voronoi summation formula for the sum over $f$. 
We first assume that $\ell_1 \neq \ell_2$. We write 
\[S_2 = S_{21}+S_{22},\] where $S_{21}$ corresponds to the case $\ell_1 \nmid g_1$ and $S_{22}$ corresponds to the case $\ell_1|g_1$. Using Proposition \ref{prop:Rvoronoi_coprime} and Lemma \ref{residue}, we write
\[S_{21}=M_{21}+E_{21},\] where 
\begin{multline}
    M_{21}= \frac{q^n \chi_1(\ell_2) G(\chi_2) L(1,\chi_1) }{|\ell_2|^2}  \sum_{g \in \mathcal{M}_{\leqs  [(d-1)/2]}}  \frac{1}{|g|} \sum_{g_1 \mid g} \frac{\ov{\chi_2}({g}/{g_1}) \chi_1(g_1)}{|g_1|}\\\times \copsum_{b_1 \pmod{g_1\ell_2}} \ov{\chi_2}(-b_1) e\Big(\frac{b_1h}{g_1\ell_2}\Big) ,\label{m21}
\end{multline}
and 
\begin{multline}
E_{21} = \frac{q^n \chi_1(\ell_2) G(\chi_1) G(\chi_2) }{|\ell_1 \ell_2^2|}  \sum_{g \in \mathcal{M}_{\leqs  [(d-1)/2]}}  \frac{1}{|g|} \sum_{g_1 \mid g} \frac{\ov{\chi_2}({g}/{g_1}) \chi_1(g_1)}{|g_1|}  \\
 \times \copsum_{b_1 \pmod{g_1\ell_2}} \ov{\chi_2}(-b_1) e\Big(\frac{b_1h}{g_1\ell_2}\Big) \sum_{\lambda \in \Fq^\times} \sum_{\eta_1,\eta_2 \in \{0,1\}} \sum_{0\leqs k \leqs \mu} b_{\mu-k}(\ov{\chi_1}) \sum_{f \in \mathcal{M}_k} r_{\ov \chi_1}(f) e \Big(  \frac{\lambda \ov b_1 f}{g_1 \ell_2}\Big), \label{e21}
    \end{multline} 
where $\mu= 2 \deg(g_1 \ell_2)+\deg(\ell_1)-n-2$.
\begin{remark}
\label{rmk3}
The error term $E_{21}$ vanishes when $m\leqs n$, $n$ is even, and $\deg(\ell_1)=\deg(\ell_2)=1$. We recall that $\deg(\ell_1), \deg(\ell_2)>0$ since $\ell_1, \ell_2\not = 1$. 
\end{remark}

Now using Proposition \ref{prop:Rvoronoi_notcoprime}, we further rewrite
\begin{multline}
\label{m22}
    M_{22} = \frac{q^n G(\chi_1) G(\chi_2)L(1,\ov \chi_1)}{|\ell_2|^2}   \sum_{g \in \mathcal{M}_{\leqs  [(d-1)/2]}}  \frac{1}{|g|} \sum_{\substack{g_1 \mid g  \\ \ell_1|g_1}} \frac{\ov{\chi_2}({g}/{g_1})}{|g_1|} \\\times \copsum_{b_1 \pmod{g_1\ell_2}} \ov{\chi_2}(-b_1)\ov{\chi_1}(b_1)  e\Big(\frac{b_1h}{g_1\ell_2}\Big) ,
\end{multline}
and 
\begin{multline}
    E_{22} = \frac{q^{n+1}G(\chi_2)}{|\ell_2|^2}  \sum_{g \in \mathcal{M}_{\leqs  [(d-1)/2]}}  \frac{1}{|g|} \sum_{\substack{g_1 \mid g \\ \ell_1 | g_1}} \ov{\chi_2}\Big(\frac{g}{g_1}\Big)\copsum_{b_1 \pmod{g_1\ell_2}} \ov{\chi_1 \chi_2}(-b_1) e\Big(\frac{b_1h}{g_1\ell_2}\Big) \\
 \times \sum_{\lambda\in\Fq^\times}\sum_{\eta_1,\eta_2\in\{0,1\}} \sum_{0 \leqs k \leqs \nu} b_{\nu-k}(\chi_1) \sum_{f \in \mathcal{M}_{k}} r_{\chi_1}(f) e \Big(  \frac{ \lambda \ov b_1 f}{g_1 \ell_2}\Big), \label{e22}\end{multline}
where $\nu = 2 \deg(g_1 \ell_2)-n-2$.
\begin{remark}
\label{rmk4} The error term $E_{22}$ vanishes when $m\leqs n$ and $ \deg(\ell_2)=1$, or when $m=n+1$, $n$ is odd, and $\deg(\ell_2)=1$. (Recall that here $\deg(\ell_2)>0.$)
\end{remark}

Now we assume that $\ell_1 = \ell_2:=\ell$. Since $\ell \mid (g_1\ell)$ trivially, we use Proposition \ref{prop:Rvoronoi_notcoprime}. In this case, we write
\[S_2 = M_2+E_2,\] and $M_2$ is now similar to $M_{22}$ and $E_2$ is similar to $E_{22}$. Namely, we have
\begin{multline}
    M_2 =  \frac{q^n G(\chi_1) G(\chi_2)L(1,\ov \chi_1)}{|\ell|^2}   \sum_{g \in \mathcal{M}_{\leqs  [(d-1)/2]}}  \frac{1}{|g|} \sum_{\substack{g_1 \mid g }} \frac{\ov{\chi_2}({g}/{g_1})}{|g_1|} \\\times \copsum_{b_1 \pmod{g_1\ell}} \ov{\chi_2}(-b_1)\ov{\chi_1}(b_1)  e\Big(\frac{b_1h}{g_1\ell}\Big),
    \label{m2}
\end{multline} 
and 
\begin{multline}
    E_{2} = \frac{q^{n+1}G(\chi_2)}{|\ell|^2}  \sum_{g \in \mathcal{M}_{\leqs  [(d-1)/2]}}  \frac{1}{|g|} \sum_{\substack{g_1 \mid g }} \frac{\ov{\chi_2}({g}/{g_1})}{|g_1|}\copsum_{b_1 \pmod{g_1\ell}} \ov{\chi_1 \chi_2}(-b_1) e\Big(\frac{b_1h}{g_1\ell}\Big)\\\times \sum_{\lambda\in\Fq^\times}\sum_{\eta_1,\eta_2\in\{0,1\}} \sum_{0 \leqs k \leqs \nu} b_{\nu-k}(\chi_1) \sum_{f \in \mathcal{M}_{k}} r_{\chi_1}(f) e \Big(  \frac{ \lambda \ov b_1 f}{g_1 \ell}\Big), \label{e2} \end{multline}
where $\nu = 2 \deg(g_1 \ell)-n-2$.
\begin{remark}
\label{rmk6}
As before, the error term $E_2$ vanishes when  $m \leqs n$  and $\deg(\ell)<2$ or when  $m=n+1$, $n$ s odd, and $\deg(\ell)<2$. (Recall that here $\deg(\ell)>0$.)

\end{remark}
\subsection{The case \texorpdfstring{$\ell_1 =1$}{ell1 is 1}}
If $\ell_1=1$, using equation \eqref{s1initial}, we have 
\begin{align*}S_1&=\sum_{g \in \mathcal{M}_{\leqs  [d/2]}}  \frac{\chi_2(g)}{|g|}  \sum_{g_1|g} \, \copsum_{b_1 \pmod{g_1}} e \Big( \frac{b_1h}{g_1} \Big) \sum_{f \in \mathcal{M}_n} d(f) e \Big(  \frac{b_1f}{g_1}\Big).
\end{align*}

Using Proposition \ref{prop:Rvoronoi_coprime} and Lemma \ref{residue_d}, we write $S_1= M_1+E_1$, where 
\begin{align}
\label{m1d}
    M_1 &= q^n \sum_{g \in \mathcal{M}_{\leqs  [d/2]}}  \frac{\chi_2(g)}{|g|}  \sum_{g_1|g}  \frac{n+1-2 \deg(g_1)}{|g_1|} \copsum_{b_1 \pmod{g_1}} e \Big( \frac{b_1h}{g_1} \Big),
\end{align}
and
\begin{multline}
\label{e1}
    E_1 = q^n \sum_{g \in \mathcal{M}_{\leqs  [d/2]}}  \frac{\chi_2(g)}{|g|}  \sum_{g_1|g}  \frac{1}{|g_1|} \copsum_{b_1 \pmod{g_1}} e \Big( \frac{b_1h}{g_1} \Big)\\\times \sum_{\lambda \in \Fq^\times} \sum_{\eta_1,\eta_2 \in \{0,1\}} \sum_{0\leqs k \leqs \mu} b_{\mu-k} \sum_{f \in \mathcal{M}_k} d(f) e \Big(  \frac{\lambda \ov{b_1} f}{g_1}\Big),
\end{multline}
with $\mu =2 \deg(g_1)-n-2$. 
\begin{remark}
\label{rmk4.5}
As before, the error term $E_1$ vanishes when  $m < n+2$ or when  $m=n+2$ and  $n$ is odd. 
\end{remark}
In the case of $S_2$, given in equation \eqref{s2}, we use Proposition \ref{prop:Rvoronoi_coprime} and Lemma \ref{residue_d} again, and we write $S_2=M_2+E_2,$ where 
      \begin{multline}
      \label{m2d}
          M_2 = \frac{q^n G(\chi_2)}{|\ell_2|^2}  \sum_{g \in \mathcal{M}_{\leqs  [(d-1)/2]}}  \frac{1}{|g|} \sum_{g_1 \mid g} \frac{\ov{\chi_2}({g}/{g_1})( n+1-2\deg(g_1 \ell_2))}{|g_1|}\\ \times \copsum_{b_1 \pmod{g_1\ell_2}} \ov{\chi_2}(-b_1) e\Big(\frac{b_1h}{g_1\ell_2}\Big),
      \end{multline}
and
      \begin{multline}
      E_2 = \frac{q^n G(\chi_2)}{|\ell_2|^2}  \sum_{g \in \mathcal{M}_{\leqs  [(d-1)/2]}}  \frac{1}{|g|} \sum_{g_1 \mid g} \frac{\ov{\chi_2}({g}/{g_1})}{|g_1|} \copsum_{b_1 \pmod{g_1\ell_2}} \ov{\chi_2}(-b_1) e\Big(\frac{b_1h}{g_1\ell_2}\Big) \\
       \times \sum_{\lambda \in \Fq^\times} \sum_{\eta_1,\eta_2 \in \{0,1\}} \sum_{0\leqs k \leqs \mu} b_{\mu-k}(1) \sum_{f \in \cM_k} d(f)e\Big(\frac{\lambda\ov{b_1}f}{g_1 \ell_2}\Big), \label{e2d} \end{multline}
with $\mu = 2\deg(g_1 \ell_2)  - n - 2$. 
\begin{remark}
\label{rmk5}
As before, the error term $E_2$ vanishes when  $m \leqs n$  and $\deg(\ell_2)<2$ or when  $m=n+1$, $n$ s odd, and $\deg(\ell_2)<2$. (Recall that here $\deg(\ell_2)>0$.)
\end{remark}
\section{Evaluating the main terms when \texorpdfstring{$\ell_2 \neq 1$}{ell2 is not 1}}
\label{section_mt1}
\subsection{The main terms when \texorpdfstring{$\ell_1 \neq 1$}{ell1 is not 1}}
Here, we will evaluate the main terms $M_{11}$, $M_{12}$, $M_{21}$, and $M_{22}$ in the case where  $\ell_1 \neq 1$.

We begin by evaluating $M_{11}$, given in \eqref{m11}. Recall that expression \eqref{m11} holds regardless whether $\ell_1 = \ell_2$ or not. 
Note that we have
\begin{align}
\label{rel1}
 \copsum_{b_1 \pmod{g_1}} e \Big( \frac{b_1h}{g_1} \Big) = \sum_{g_2|g_1} \mu\Big(\frac{g_1}{g_2}\Big) \sum_{b_2 \pmod{g_2}} e \Big(  \frac{b_2 h }{g_2}\Big)   =  \sum_{g_2|(g_1,h)} \mu\Big(\frac{g_1}{g_2}\Big) |g_2|,
\end{align}
and hence, writing $g_1=Bg_2$ and $g=Ag_1$,
\begin{align} 
    M_{11} &= q^n L(1,\chi_1)  \sum_{g \in \mathcal{M}_{\leqs  [d/2]}}   \frac{\chi_2(g)}{|g|}  \sum_{\substack{g_1|g\\\ell_1\nmid g_1}} \frac{\chi_1(g_1)}{|g_1|} \sum_{g_2|(g_1,h)} \mu\Big(\frac{g_1}{g_2}\Big) |g_2| \nonumber \\
&= q^n L(1,\chi_1) \sum_{g_2|h} \frac{(\chi_1 \chi_2)(g_2)}{|g_2|} \sum_{\deg(AB) \leqs [d/2]-\deg(g_2)} \frac{\mu(B)\chi_2(A) (\chi_1 \chi_2)(B) }{|AB^2|}. \label{m11_interm}
\end{align}
Using Perron's formula \eqref{perron} twice for the sums over $A$ and $B$, we get that
\begin{multline*}
    M_{11} = q^n L(1,\chi_1) \frac{1}{(2 \pi i)^2} \oint_{|y|=q^{-\epsilon}} \oint_{|x|=q^{-\epsilon}} \frac{ \mathcal{L} \left( {xy}/{q},\chi_2\right)}{\mathcal{L} \left ({x}/{q^2},\chi_1 \chi_2\right) (1-x)(1-y)(xy)^{[d/2]}} \\
    \times 
    \sum_{g_2|h} \frac{(\chi_1 \chi_2)(g_2) (xy)^{\deg(g_2)}}{|g_2|}
     \, \frac{dx}{x} \frac{dy}{y}.
\end{multline*}
%where we are integrating along $|x|=|y|=q^{-\epsilon}$. 
In the integral over $x$, we shift the contour of integration to $|x|=q^{1/2}$ and encounter a pole at $x=1$. The integral over the new contour is bounded up to a constant by $q^{n-d/4+\epsilon d}$. We then get that
\begin{align*}
    M_{11} &= \frac{ q^n  L(1,\chi_1)}{L(2,\chi_1\chi_2)} \frac{1}{2 \pi i} \oint_{|y|=q^{-\epsilon}}  \frac{ \mathcal{L} \left({y}/{q}, \chi_2\right)}{ (1-y)y^{[d/2]}} 
     \sum_{g_2|h} \frac{(\chi_1 \chi_2)(g_2) y^{\deg(g_2)}}{|g_2|}
     \frac{dy}{y} + O\Big(q^{n-d/4+\epsilon d} \Big).
\end{align*}
Now in the integral over $y$ we shift the contour of integration to $|y|=q^{1/2}$ and encounter the pole at $y=1$. By the Lindel\"of hypothesis \eqref{lindelof}, the integrand on the circle integral over the new contour is bounded up to a constant by $q^{n-d/4+\epsilon d}$. It follows that
\begin{align}\label{m11_expression}
    M_{11} =\frac{ q^n  L(1,\chi_1)L(1,\chi_2)}{L(2,\chi_1\chi_2)}  \sum_{\substack{g_2|h \\ \deg(g_2) \leqs [d/2]}} \frac{ (\chi_1 \chi_2)(g_2)}{|g_2|}+O\Big(q^{n-d/4+\epsilon d} \Big).
\end{align}

We now focus on evaluating $M_{12}$, which is given by equation \eqref{m12}. As remarked before, if $\ell_1=\ell_2$, then $M_{12}=0$. We then assume that $\ell_1 \neq \ell_2$. We rewrite
\begin{equation*}
%\label{m12prime}
    M_{12} = \frac{q^nG(\chi_1) L(1,\ov{\chi_1})}{|\ell_1|} \sum_{\substack{g \in \mathcal{M}_{\leqs [d/2]} \\ \ell_1 | g}} \frac{\chi_2(g)}{|g|} \sum_{g_2| (g/\ell_1)} \frac{1}{|g_2|} \, \copsum_{b_1 \pmod{\ell_1 g_2}} \ov{\chi_1}(b_1) e\Big(  \frac{b_1h}{\ell_1 g_2}\Big).
\end{equation*}
We have
\begin{align}
\label{sumstar}
    \copsum_{b_1 \pmod{\ell_1 g_2}} \ov \chi_1(b_1) e\Big(  \frac{b_1h}{\ell_1 g_2}\Big) = \sum_{g_3 | (\ell_1 g_2)} \mu\Big( \frac{\ell_1 g_2}{g_3}\Big) \ov{\chi_1}\Big(\frac{\ell_1 g_2}{g_3}\Big) \sum_{b_1 \pmod{g_3}} \ov{\chi_1}(b_1) e \Big(  \frac{b_1h}{g_3}\Big).
\end{align}
In the sum above, note that we need $\ell_1 |g_3$, since otherwise $\ov{\chi_1}(\ell_1g_2/g_3)=0$; hence we write $g_3=\ell_1 g_4$. Writing $b_1=\alpha+b_2\ell_1$, we then get that
\begin{align*}
    \eqref{sumstar} &= \sum_{g_4|g_2} \mu\Big(\frac{g_2}{g_4}\Big) \ov{\chi_1}\Big(\frac{g_2}{g_4}\Big) \sum_{b_1 \pmod{\ell_1 g_4}} \ov{\chi_1}(b_1) e \Big(  \frac{b_1 h}{\ell_1 g_4}\Big) \\
    &= \sum_{g_4|g_2} \mu\Big(\frac{g_2}{g_4}\Big) \ov{\chi_1}\Big(\frac{g_2}{g_4}\Big) \sum_{\alpha \pmod{\ell_1}} \ov{\chi_1}(\alpha)  e \Big(  \frac{\alpha h}{\ell_1 g_4} \Big) \sum_{b_2 \pmod{g_4}} e \Big( \frac{b_2 h }{g_4} \Big) \\
    &= \sum_{g_4|(g_2,h)} |g_4|\mu\Big(\frac{g_2}{g_4}\Big) \ov{\chi_1}\Big(\frac{g_2}{g_4}\Big) \sum_{\alpha \pmod{\ell_1}} \ov{\chi_1}(\alpha)  e \Big(  \frac{\alpha h/g_4}{\ell_1 } \Big) \\
    &= G(\ov{\chi_1}) \sum_{g_4|(g_2,h)} |g_4|
    \mu\Big(\frac{g_2}{g_4}\Big) \ov{\chi_1}\Big(\frac{g_2}{g_4}\Big)  \chi_1\Big(\frac{h}{g_4}\Big).
\end{align*}
Now, using the above and the fact that $G(\chi_1) G(\ov{\chi_1})= |\ell_1|$, and writing $g_2=Bg_4$ and $\frac{g}{\ell_1}=Cg_4$ so that $B\mid C$, we get that 
\begin{align*}
  & M_{12} = q^n L(1,\ov{\chi_1}) \sum_{\substack{g \in \mathcal{M}_{\leqs [d/2]} \\ \ell_1 | g}} \frac{\chi_2(g)}{|g|} \sum_{g_2| (g/\ell_1)} \frac{1}{|g_2|}\sum_{g_4|(g_2,h)} |g_4|\mu\Big(\frac{g_2}{g_4}\Big) \ov{\chi_1}\Big(\frac{g_2}{g_4}\Big) \chi_1\Big(\frac{h}{g_4}\Big)\\
    &= \frac{q^n \chi_2(\ell_1)L(1,\ov{\chi_1}) }{|\ell_1|} \sum_{g_4|h} \frac{ \chi_2(g_4) \chi_1(h/g_4)}{|g_4|}  \sum_{\substack{C \in \mathcal{M}\\ \deg(C) \leqs [d/2]-\deg(\ell_1)-\deg(g_4)}} \frac{\chi_2(C)}{|C|} \sum_{B|C} \frac{\mu(B)\ov{\chi_1}(B) }{|B|} \\
    &= \frac{q^n\chi_2(\ell_1)L(1,\ov{\chi_1}) }{|\ell_1|} \sum_{g_4|h}  \frac{\chi_2(g_4) \chi_1(h/g_4)}{|g_4|}  \sum_{\substack{C \in \mathcal{M}\\ \deg(C) \leqs [d/2]-\deg(\ell_1)-\deg(g_4)}} \frac{\chi_2(C)}{|C|}  \prod_{P|C} \Big( 1 - \frac{\ov{\chi_1}(P) }{|P|}\Big).
\end{align*}

%\begin{align*}
%   & M_{12} = q^n L(1,\ov{\chi_1}) \sum_{\substack{g \in \mathcal{M}_{\leqs [d/2]} \\ \ell_1 | g}} \frac{\chi_2(g)}{|g|} \sum_{g_2| (g/\ell_1)} \frac{1}{|g_2|}\sum_{g_4|(g_2,h)} |g_4|\mu\Big(\frac{g_2}{g_4}\Big) \ov{\chi_1}\Big(\frac{g_2}{g_4}\Big) \chi_1\Big(\frac{h}{g_4}\Big)\\
%    &= \frac{q^n \chi_2(\ell_1)L(1,\ov{\chi_1}) }{|\ell_1|} \sum_{g_4|h} \frac{ \chi_2(g_4) \chi_1(h/g_4)}{|g_4|}  \sum_{C \in \mathcal{M}_{\leqs [d/2]-\deg(\ell_1)-\deg(g_4)}} \frac{\chi_2(C)}{|C|} \sum_{B|C} \frac{\mu(B)\ov{\chi_1}(B) }{|B|} \\
%    &= \frac{q^n\chi_2(\ell_1)L(1,\ov{\chi_1}) }{|\ell_1|} \sum_{g_4|h}  \frac{\chi_2(g_4) \chi_1(h/g_4)}{|g_4|}  \sum_{C \in \mathcal{M}_{\leqs [d/2]-\deg(\ell_1)-\deg(g_4)}} \frac{\chi_2(C)}{|C|}  \prod_{P|C} \Big( 1 - \frac{\ov{\chi_1}(P) }{|P|}\Big).
%\end{align*}

Looking at the generating series for the sum over $C$, we have
\[\sum_{C \in \mathcal{M}} \frac{\chi_2(C)}{|C|}  \prod_{P|C} \Big( 1 - \frac{\ov{\chi_1}(P) }{|P|}\Big) u^{\deg(C)} = \frac{\mathcal{L}(u/q,\chi_2)}{\mathcal{L}(u/q^2,\ov{\chi_1} \chi_2)}.\]
Using Perron's formula \eqref{perron}, it follows that
\begin{multline*}
 \sum_{C \in \mathcal{M}_{\leqs [d/2]-\deg(\ell_1)-\deg(g_4)}} \frac{\chi_2(C)}{|C|}  \prod_{P|C} \Big( 1 - \frac{\ov{\chi_1}(P) }{|P|}\Big)\\
 = \frac{1}{2 \pi i} \oint_{|u|=q^{-\epsilon}}  \frac{\mathcal{L}(u/q,\chi_2)}{\mathcal{L}(u/q^2,\ov{\chi_1} \chi_2)(1-u)u^{[d/2]-\deg(\ell_1) - \deg(g_4)}} \frac{du}{u} .
\end{multline*}
In the integral above, we shift the contour of integration to $|u|=q^{1/2}$ and encounter a pole at $u=1$. Using the Lindel\"of hypothesis \eqref{lindelof} to bound the integrand on the circle $|u|=q^{1/2}$ and evaluating the residue at $u=1$, it follows that
\begin{align*}
 \sum_{C \in \mathcal{M}_{\leqs [d/2]-\deg(\ell_1)-\deg(g_4)}} &\frac{\chi_2(C)}{|C|}  \prod_{P|C} \Big( 1 - \frac{\ov{\chi_1}(P) }{|P|}\Big) = \frac{\mathcal{L}(1/q,\chi_2)}{\mathcal{L}(1/q^2,\ov{\chi_1} \chi_2)} + O \Big( \frac{|\ell_2|^{\epsilon} \sqrt{|\ell_1 g_4|}}{q^{d/4}}\Big) .
\end{align*}
Using the above in the expression for $M_{12}$, we get that 
\begin{align}
\label{m12_expression}
   M_{12} = \frac{q^n \chi_2(\ell_1) L(1,\ov{\chi_1}) L(1,\chi_2)}{|\ell_1|L(2,\ov{\chi_1}  \chi_2)}  \sum_{\substack{g_4|h \\ \deg(g_4) \leqs [d/2]-\deg(\ell_1)}}  \frac{\chi_2(g_4) \chi_1(h/g_4)}{|g_4|} + O \Big( \frac{q^{n-d/4} |h \ell_2|^{\epsilon}}{|\ell_1|^{1/2}} \Big). 
\end{align}
\kommentar{Now for the sum over $g_4$, note that if $\ell_1^j || h$, then we must have $\ell_1^j || g_4$, otherwise $\chi_1(h/g_4)=0$. 
We write $g_4=\ell_1^j g_5$ and $(\ell_1,g_5)=1$. Then we have 
\begin{align*}
   \sum_{g_4|h}  \frac{\chi_2(g_4) \chi_1(h/g_4)}{|g_4|} &= \frac{\chi_2(\ell_1^j)}{|\ell_1^j|} \sum_{\substack{g_5| (h/\ell_1^j) \\ (g_5,\ell_1)=1}} \frac{\chi_2(g_5) \chi_1(h/(\ell_1^j g_5))}{|g_5|}  \\
   &= \frac{\chi_2(\ell_1^j) \chi_1(h/\ell_1^j)}{|\ell_1^j|} \prod_{P|(h/\ell_1^j)} \Big(  \sum_{i=0}^{v_P(h/\ell_1^j)} \frac{ \ov{\chi_1} \chi_2(P)^i}{|P|^i}\Big).
\end{align*}
Combining the two equations above, it follows that
\begin{align}
    M_{12} &= \frac{q^n L(1,\ov{\chi_1}) L(1,\chi_2)}{|\ell_1|^{v_{\ell_1}(h)+1}L(2,\ov{\chi}_1\chi_2)} \chi_2(\ell_1) \chi_2(\ell_1^{v_{\ell_1}(h)}) \chi_1(h/\ell_1^{v_{\ell_1}(h)})\prod_{P|(h/\ell_1^{v_{\ell_1}(h)})} \Big(  \sum_{i=0}^{v_P(h/\ell_1^{v_{\ell_1}(h)})} \frac{ \ov{\chi_1} \chi_2(P)^i}{|P|^i}\Big) \nonumber \\
    &+ O \Big( \frac{q^{n-d/4} |h \ell_2|^{\epsilon}}{|\ell_1|^{1/2}} \Big). \label{m12_expression}
\end{align}
}

Now we will evaluate the terms $M_{21}$ and $M_{22}$ in the case $\ell_1 \neq \ell_2$. Using the computation in \eqref{sumstar}, we have that 
\begin{align*}
    \sumstar_{b_1 \pmod{g_1 \ell_2}} \ov{\chi_2}(-b_1) e \Big(  \frac{b_1 h}{g_1 \ell_2}\Big) = \ov{\chi_2}(-1) G(\ov{\chi_2}) \sum_{g_3|(g_1,h)} |g_3| \mu\Big(\frac{g_1}{g_3}\Big) \ov{\chi_2}\Big(\frac{g_1}{g_3}\Big) \chi_2\Big(\frac{h}{g_3}\Big),
\end{align*}
so using equation \eqref{m21} and the fact that $G(\chi_2) G(\ov{\chi_2})=|\ell_2|$, we get that
\begin{multline*}
    M_{21} = \frac{q^n \chi_1(\ell_2) \ov{\chi_2}(-1) L(1,\chi_1) }{|\ell_2|}  \sum_{g \in \mathcal{M}_{\leqs  [(d-1)/2]}}  \frac{1}{|g|} \sum_{g_1 \mid g} \frac{\ov{\chi_2}({g}/{g_1}) \chi_1(g_1)}{|g_1|} \\\times \sum_{g_3|(g_1,h)} |g_3| \mu\Big(\frac{g_1}{g_3}\Big) \ov{\chi_2}\Big(\frac{g_1}{g_3}\Big) \chi_2\Big(\frac{h}{g_3}\Big).
\end{multline*}
Setting $g_1=g_3 A$ and $g=g_1 B$, we have
\begin{multline*}
    M_{21} = \frac{q^n \chi_1(\ell_2) \ov{\chi_2}(-1) L(1,\chi_1)}{|\ell_2|} \sum_{g_3|h} \frac{\chi_1(g_3) \chi_2(h/g_3)}{|g_3|}\\\times \sum_{\deg(AB) \leqs [(d-1)/2]-\deg(g_3)} \frac{\mu(A) \chi_1 \ov{\chi_2}(A) \ov{\chi_2}(B)}{|A^2B|}.
\end{multline*}
Applying Perron's formula twice for the sums over $A, B$, we get that
\begin{align}
 &\sum_{\deg(AB) \leqs [(d-1)/2]-\deg(g_3)}  \frac{\mu(A) \chi_1 \ov{\chi_2}(A) \ov{\chi_2}(B)}{|A^2B|} \nonumber \\
 & = \frac{1}{(2 \pi i)^2}\oint_{|y|=q^{-\epsilon}} \oint_{|x|=q^{-\epsilon}} \frac{ \mathcal{L}(x/q,\ov{\chi_2})}{\mathcal{L}(xy/q^2,\chi_1 \ov{\chi_2})(1-x)(1-y)(xy)^{[(d-1)/2]-\deg(g_3)}}   \frac{dx}{x} \frac{dy}{y}.\label{eq:superperron}
\end{align}
%where we are integrating over circles around the origin of radii $|x|=|y|=q^{-\epsilon}$. 
In the integral over $x$, we shift the contour of integration to $|x|=q^{1/2}$, and encounter the pole at $x=1$. By applying the  Lindel\"of hypothesis \eqref{lindelof}, the error term coming from the double integral over the new contour is bounded by $q^{n-d/4-\deg(\ell_2)+\epsilon d}$. 
We further shift the contour over $y$ to $|y|=q^{1/2}$, and the error term coming from the new single integral will be similarly bounded by $q^{n-d/4-\deg(\ell_2)+\epsilon d}$.
Overall, it follows that
\begin{multline}
\label{m21_expression}
   M_{21} = \frac{q^n \chi_1(\ell_2) \ov{\chi_2}(-1) L(1,\chi_1) L(1,\ov{\chi_2})}{|\ell_2|L(2,\chi_1 \ov{\chi_2})} \sum_{\substack{g_3|h \\ \deg(g_3) \leqs [(d-1)/2]}} \frac{\chi_1(g_3) \chi_2(h/g_3)}{|g_3|}\\ + O \Big(q^{n-d/4-\deg(\ell_2)+\epsilon d} \Big).
\end{multline}
\kommentar{Similarly as before, we have
\begin{align*}
   \sum_{g_3|h}  \frac{\chi_1(g_3) \chi_2(h/g_3)}{|g_3|} &= \frac{\chi_1(\ell_2^j) \chi_2(h/\ell_2^j)}{|\ell_2^j|} \prod_{P|(h/\ell_2^j)} \Big(  \sum_{i=0}^{v_P(h/\ell_2^j)} \frac{ \ov{\chi_2} \chi_1(P)^i}{|P|^i}\Big),
\end{align*}
where $j$ is such that $\ell_2^j||h$.
Hence 
\begin{align*}
    M_{21} &= \frac{q^n  \ov{\chi_2}(-1) L(1,\chi_1) L(1,\ov{\chi_2})}{|\ell_2|^{v_{\ell_2}(h)+1}L(2,\chi_1 \ov{\chi_2})} \chi_1(\ell_2)\chi_1(\ell_2^{v_{\ell_2}(h)}) \chi_2(h/\ell_2^{v_{\ell_2}(h)}) \prod_{P|(h/\ell_2^{v_{\ell_2}(h)})}  \Big(  \sum_{i=0}^{v_P(h/\ell_2^{v_{\ell_2}(h)})} \frac{ \ov{\chi_2} \chi_1(P)^i}{|P|^i}\Big) \\
    &+ O \Big(q^{n-d/4+\epsilon d-\deg(\ell_2)} |\ell_2 h|^{\epsilon} \Big). 
\end{align*}
}
Now we evaluate $M_{22}$ given in \eqref{m22}. \kommentar{Recall that we have
\begin{align*}
    M_{22} = \frac{q^n G(\chi_1) G(\chi_2)L(1,\ov \chi_1)}{|\ell_2|^2}   \sum_{g \in \mathcal{M}_{\leqs  [(d-1)/2]}}  \frac{1}{|g|} \sum_{\substack{g_1 \mid g  \\ \ell_1|g_1}} \frac{\ov{\chi_2}\Big(\frac{g}{g_1}\Big)}{|g_1|} \copsum_{b_1 \pmod{g_1\ell_2}} \ov{\chi_2}(-b_1)\ov{\chi_1}(b_1)  e\Big(\frac{b_1h}{g_1\ell_2}\Big) ,
\end{align*}}
Let $\chi= \ov{\chi_1 \chi_2}$. We proceed similarly as in equation \eqref{sumstar}, and we have 
\kommentar{\begin{align*}
 & \sumstar_{b_1 \pmod{g_1 \ell_2}} \chi(b_1) e \Big( \frac{b_1h}{g_1 \ell_2}\Big) = \sum_{g_4|g_1} \mu(g_1/g_4)\chi(g_1/g_4) \sum_{b_1 \pmod{\ell_2 g_4}} \chi(b_1) e \Big(  \frac{b_1 h}{\ell_2 g_4}\Big).
 \end{align*}
 Note that we must have $\ell_1|g_4$. We write $b_1 =\alpha+\ell_1 \ell_2 c$. Then we have}
 \begin{align*}
  \sumstar_{b_1 \pmod{g_1 \ell_2}} \chi(b_1) e \Big( \frac{b_1h}{g_1 \ell_2}\Big)  &=   \sum_{g_4|g_1} \mu\Big(\frac{g_1}{g_4}\Big)\chi\Big(\frac{g_1}{g_4}\Big)  \sum_{\alpha \pmod {\ell_1 \ell_2}} \chi(\alpha)e \Big( \frac{\alpha h}{\ell_2 g_4}\Big)\\ &\times  \sum_{c \pmod{g_4/\ell_1}}  e \Big(\frac{\ell_1 c h}{g_4} \Big) \\
  &=\frac{1}{|\ell_1|} \sum_{\substack{g_4|g_1 \\ (g_4/\ell_1)|h}} |g_4|\mu\Big(\frac{g_1}{g_4}\Big)\chi\Big(\frac{g_1}{g_4}\Big)  \sum_{\alpha \pmod {\ell_1 \ell_2}} \chi(\alpha)e \Big( \frac{\alpha h}{\ell_2 g_4}\Big) \\
 &=  \sum_{\substack{g_5| (g_1/\ell_1) \\ g_5|h}} |g_5| \mu\Big(\frac{g_1}{\ell_1 g_5}\Big) \chi\Big(\frac{g_1}{\ell_1 g_5}\Big)\sum_{\alpha \pmod{\ell_1 \ell_2}} \chi(\alpha) e \Big( \frac{\alpha (h/g_5)}{\ell_1 \ell_2}\Big) \\
 &=G(\chi) \sum_{\substack{g_5| (g_1/\ell_1) \\ g_5|h}} |g_5| \mu\Big(\frac{g_1/\ell_1}{g_5}\Big)\chi\Big(\frac{g_1/\ell_1}{g_5}\Big) \ov{\chi}\Big(\frac{h}{g_5}\Big).
 \end{align*}
Using the above in the expression \eqref{m22} for $M_{22}$ and setting $g_1=\ell_1 g_2$, we have 
\begin{align*}
    M_{22} &= \frac{q^n \ov{\chi_2}(-1)G(\chi_1) G(\chi_2) G(\ov{\chi_1 \chi_2})L(1,\ov \chi_1) }{|\ell_2|^2}   \sum_{g \in \mathcal{M}_{\leqs  [(d-1)/2]}}  \frac{1}{|g|} \sum_{\substack{g_1 \mid g  \\ \ell_1|g_1}} \frac{\ov{\chi_2}({g}/{g_1})}{|g_1|} \\
    & \times  \sum_{\substack{g_5| (g_1/\ell_1) \\ g_5|h}} |g_5| \mu\Big(\frac{g_1/\ell_1}{g_5}\Big)\ov{\chi_1 \chi_2}\Big(\frac{g_1/\ell_1}{g_5}\Big) \chi_1 \chi_2\Big(\frac{h}{g_5}\Big) \\&= \frac{q^n \ov{\chi_2}(-1) G(\chi_1) G(\chi_2) G(\ov{\chi_1 \chi_2})L(1,\ov \chi_1)}{|\ell_1 \ell_2^2|} \\
    & \times \sum_{\substack{g \in \mathcal{M}_{\leqs [(d-1)/2]}\\ \ell_1|g}} \frac{1}{|g|} \sum_{g_2|(g/\ell_1)} \frac{\ov{\chi_2}({g}/{ (\ell_1 g_2)})}{|g_2|}  \sum_{\substack{g_5| g_2 \\ g_5|h}} |g_5| \mu\Big(\frac{g_2}{g_5}\Big)\ov{\chi_1 \chi_2}\Big(\frac{g_2}{g_5}\Big) \chi_1 \chi_2\Big(\frac{h}{g_5}\Big) \\
    &= \frac{q^n \ov{\chi_2}(-1) G(\chi_1) G(\chi_2) G(\ov{\chi_1 \chi_2})L(1,\ov \chi_1)}{|\ell_1^2 \ell_2^2|}  \sum_{g_5|h} \frac{\chi_1 \chi_2(h/g_5)}{|g_5|} \\
    & \times \sum_{\deg(AB) \leqs  [(d-1)/2]-\deg(\ell_1 g_5)} \frac{\mu(A) \ov{\chi_1 \chi_2}(A) \ov{\chi_2}(B)}{|A^2B|},
\end{align*}
where in the last equality, we set $g_2=g_5A$ and $g=\ell_1g_2B$.

We proceed as before, applying Perron's formula as in \eqref{eq:superperron} and using the fact that
\[G(\ov{\chi_1 \chi_2}) = \ov{\chi_1}(\ell_2) \ov{\chi_2}(\ell_1) G(\ov{\chi_1}) G(\ov{\chi_2}),\]
which follows from the Chinese Remainder Theorem, to obtain that
\begin{multline}
\label{m22_expression}
  M_{22}= \frac{q^n  \ov{\chi_1}(\ell_2)  \ov{\chi_2}(-\ell_1) L(1,\ov \chi_1)L(1,\ov{\chi_2})}{|\ell_1 \ell_2 | L(2,\ov{\chi_1 \chi_2})}     \sum_{\substack{g_5|h \\ \deg(g_5) \leqs [(d-1)/2]-\deg(\ell_1)}} \frac{ \chi_1 \chi_2(h/g_5)}{|g_5|}\\
  +O \Big( q^{n-d/4-\deg(\ell_1)/2-\deg(\ell_2)+\epsilon d} \Big).
\end{multline}
Now we evaluate $M_2$, the main term coming from $S_2$ in the case that $\ell_1=\ell_2$. Proceeding with equation \eqref{m2} as in equation \eqref{sumstar} and applying Perron's formula as in \eqref{eq:superperron},
we have 
\begin{align}
    M_2 &=    \frac{q^n \ov{\chi_2}(-1) G(\chi_1) G(\chi_2)G(\ov{\chi_1 \chi_2})L(1,\ov \chi_1)}{|\ell|^2}   \sum_{g \in \mathcal{M}_{\leqs  [(d-1)/2]}}  \frac{1}{|g|} \sum_{\substack{g_1 \mid g }} \frac{\ov{\chi_2}({g}/{g_1})}{|g_1|} \nonumber\\ &\times \sum_{g_4|(g_1,h)} |g_4| \mu\Big(\frac{g_1}{g_4}\Big) \ov{\chi_1 \chi_2}\Big(\frac{g_1}{g_4}\Big )  \chi_1 \chi_2\Big(\frac{h}{g_4}\Big) \nonumber \\& =  \frac{q^n \ov{\chi_2}(-1) G(\chi_1) G(\chi_2)G(\ov{\chi_1 \chi_2})L(1,\ov \chi_1)}{|\ell|^2} \sum_{g_4|h} \frac{\chi_1 \chi_2(h/g_4)}{|g_4|} \nonumber \\
    & \times \sum_{\deg(AB) \leqs  [(d-1)/2]-\deg(g_4)} \frac{\mu(A) \ov{\chi_1 \chi_2}(A) \ov{\chi_2}(B)}{|A^2B|}  \nonumber \\
    &=  \frac{q^n \ov{\chi_2}(-1) G(\chi_1) G(\chi_2)G(\ov{\chi_1 \chi_2})L(1,\ov \chi_1)L(1,\ov{\chi_2})}{|\ell^2|L(2,\ov{\chi_1 \chi_2})} \sum_{\substack{g_4| h \\ \deg(g_4) \leqs [(d-1)/2]}} \frac{ \chi_1 \chi_2(h/g_4)}{|g_4|}\nonumber\\ &+ O \Big(q^{n-d/4-\deg(\ell)/2+\epsilon d} \Big),
    \label{m2l}
\end{align}
where we have set $g_1=g_2A$ and $g=g_1B$. 
\kommentar{Alternatively,
$$M_2 = \frac{q^n G(\chi_1) G(\chi_2)G(\ov{\chi_1 \chi_2})L(1,\ov \chi_1)L(1,\ov{\chi_2})}{|\ell^2|L(2,\ov{\chi_1 \chi_2})} \sum_{\substack{g_4| h \\ \deg(g_4) \leqs [(d-1)/2]}} \frac{ \chi_1 \chi_2(h/g_4)}{|g_4|}.$$
}

\subsection{The main terms when \texorpdfstring{$\ell_1=1$}{ell1 is 1}}

To evaluate the main term $M_1$ given by \eqref{m1d}, we proceed as before,  using \eqref{rel1}, and applying Perron's formula as in \eqref{eq:superperron}, we have
\begin{align}
     M_1 &=q^n \sum_{g \in \mathcal{M}_{\leqs  [d/2]}}  \frac{\chi_2(g)}{|g|}  \sum_{g_1|g}  \frac{n+1-2\deg(g_1)}{|g_1|} \sum_{g_2|(g_1,h)} \mu\Big(\frac{g_1}{g_2}\Big) |g_2| \nonumber \\      %\label{m1} \\ 
     &= q^n \sum_{g \in \mathcal{M}_{\leqs  [d/2]}}  \frac{\chi_2(g)}{|g|}  \frac{1}{2 \pi i} \oint_{|u|=q^{-\epsilon}} \sum_{g_1|g}\frac{du}{|g_1|(1-u)^2 u^{n+1-2\deg(g_1)}}  \sum_{g_2|(g_1,h)} \mu\Big(\frac{g_1}{g_2}\Big) |g_2|  \nonumber \\
     &=  \frac{q^n}{(2 \pi i)^3} \oint_{|y|=q^{-\epsilon}} \oint_{|x|=q^{-\epsilon}} \oint_{|u|=q^{-\epsilon}} \frac{ \mathcal{L}(x/q,\chi_2)}{\mathcal{L}( {xyu^2}/{q^2}, \chi_2 ) (1-u)^2 (1-x)(1-y)(xy)^{[d/2]}u^{n}} \nonumber\\& \times \sum_{g_2|h} \frac{\chi_2(g_2)(xyu^2)^{\deg(g_2)}}{|g_2| } \frac{du}{u} \frac{dx}{x} \frac{dy}{y}.\label{m11again1}
     \end{align}
     %where we are integrating along small circles around the origin, say $|u|=|x|=|y|=q^{-\epsilon}$. 
In the integral over $y$ we shift the contour of integration to $|y|=q$ and we encounter a pole at $y=1$. The integral over the new contour will be bounded by $q^{n-d/2+\epsilon n}$.
Then we have 
\begin{multline*}
    M_1 = \frac{q^n}{(2 \pi i)^2} \oint_{|x|=q^{-\epsilon}} \oint_{|u|=q^{-\epsilon}}  \frac{ \mathcal{L}(x/q,\chi_2)}{\mathcal{L}({xu^2}/{q^2}, \chi_2) (1-u)^2 (1-x)x^{[d/2]}u^{n}}\\ \times \sum_{g_2|h} \frac{\chi_2(g_2)(u^2 x)^{\deg(g_2)}}{|g_2| } \frac{du}{u} \frac{dx}{x} + O \Big(  q^{n-d/2+\epsilon n}\Big).
\end{multline*}
  In the integral over $x$ we shift the contour to $|x|=q$ and encounter the pole at $x=1$. The integral over the new contour will be bounded by $q^{n-d/2+\deg(\ell_2)/2+\epsilon n} $. Then we obtain
  \begin{align}
      M_1 &= q^n L(1,\chi_2) \frac{1}{2 \pi i} \oint_{|u|=q^{-\epsilon}} \frac{ 1}{\mathcal{L}( {u^2}/{q^2}, \chi_2 ) (1-u)^2 u^{n}} \sum_{g_2|h} \frac{\chi_2(g_2)u^{2\deg(g_2)}}{|g_2| } \frac{du}{u} \nonumber  \\
      &+O \Big(q^{n-d/2+\epsilon n} |\ell_2|^{1/2} \Big) \nonumber \\
      &=\frac{q^n L(1,\chi_2)}{L(2,\chi_2)} \sum_{\substack{g_2|h   \\ \deg(g_2) \leqs [n/2]}} \frac{\chi_2(g_2)}{|g_2|} \Big[ n+1-2\deg(g_2)- \frac{2L'(2,\chi_2)}{(\log q) L(2,\chi_2)}\Big]\nonumber\\&+ O \Big(q^{n-d/2+\deg(\ell_2)/2+\epsilon n}  +q^{n/2+\epsilon n} \Big), \label{m1_1chi}
      \end{align}
where in the final step we shifted the last integral to $|u|=q^{1/2}$ resulting in a contribution from the residue at $u=1$ and an error term bounded by $q^{n/2+\epsilon n}$.

      Now we will evaluate the term $M_2$ given by \eqref{m2d}. We proceed as before (see \eqref{sumstar}). We have
      \begin{align*} M_2  &=\frac{q^n\ov{\chi_2}(-1)}{|\ell_2|} \sum_{g \in \mathcal{M}_{\leqs [(d-1)/2]}} \frac{1}{|g|} \sum_{g_1 \mid g} \frac{\ov{\chi_2}({g}/{g_1})(n+1-2\deg(g_1 \ell_2))}{|g_1|} \\ & \quad \times \sum_{g_4|(g_1,h)} |g_4| \mu\Big(\frac{g_1}{g_4}\Big) \ov{\chi_2}\Big(\frac{g_1}{g_4}\Big) \chi_2\Big(\frac{h}{g_4}\Big) \\ &= \frac{q^n\ov{\chi_2}(-1)}{|\ell_2|} \sum_{g \in \mathcal{M}_{\leqs [(d-1)/2]}} \frac{1}{|g|}  \frac{1}{2 \pi i} \oint_{|u|=q^{-\epsilon}} \sum_{g_1 \mid g} \frac{\ov{\chi_2}({g}/{g_1}) du }{|g_1|(1-u)^2u^{n+1-2\deg(g_1 \ell_2)}} \\ & \quad \times \sum_{g_4|(g_1,h)} |g_4| \mu\Big(\frac{g_1}{g_4}\Big) \ov{\chi_2}\Big(\frac{g_1}{g_4}\Big) \chi_2\Big(\frac{h}{g_4}\Big) \end{align*}
Now, we write $g = g_1B$ and $g_1 = g_4 A$ to get
\begin{align*}
M_2          &= \frac{q^n\ov{\chi_2}(-1)}{|\ell_2|}  \sum_{g_4|h} \frac{\chi_2(h/g_4)}{|g_4|} \frac{1}{2 \pi i} \oint_{|u|=q^{-\epsilon}} \sum_{\deg(AB)<d/2-\deg(g_4)}  \frac{u^{2\deg(g_4 A)} \ov{\chi_2}(AB) \mu(A) ~du}{|A^2B| u^{n+1-2\deg(\ell_2)}(1-u)^2}  \\
          &= \frac{q^n\ov{\chi_2}(-1)}{|\ell_2|}  \frac{1}{(2 \pi i)^3} \oint_{|y|=q^{-\epsilon}} \oint_{|x|=q^{-\epsilon}} \oint_{|u|=q^{-\epsilon}} \sum_{g_4|h} \frac{\chi_2(h/g_4) u^{2\deg(g_4 \ell_2)}}{|g_4|} \\
          & \quad \times \frac{\mathcal{L} ( {x}/{q},\ov{\chi_2}) }  {\mathcal{L} ( {xyu^2}/{q^2},\ov{\chi_2}) (1-u)^2 (1-x)(1-y) (xy)^{[(d-1)/{2}]-\deg(g_4)}u^{n}}  \frac{du}{u} \frac{dx}{x} \frac{dy}{y},
      \end{align*}
      where we wrote $g=g_1B$ and $g_1=g_4B$.

    %  \begin{align*}
    %M_2&=\frac{q^n\ov{\chi_2}(-1)}{|\ell_2|} \sum_{g \in \mathcal{M}_{\leqs [(d-1)/2]}} \frac{1}{|g|} \sum_{g_1 \mid g} \frac{\ov{\chi_2}({g}/{g_1})(n+1-2\deg(g_1 \ell_2))}{|g_1|} \\
    %& \quad \times \sum_{g_4|(g_1,h)} |g_4| \mu\Big(\frac{g_1}{g_4}\Big) \ov{\chi_2}\Big(\frac{g_1}{g_4}\Big) \chi_2\Big(\frac{h}{g_4}\Big) \\
   %       &= \frac{q^n\ov{\chi_2}(-1)}{|\ell_2|} \sum_{g \in \mathcal{M}_{\leqs [(d-1)/2]}} \frac{1}{|g|}  \frac{1}{2 \pi i} \oint_{|u|=q^{-\epsilon}} \sum_{g_1 \mid g} \frac{\ov{\chi_2}({g}/{g_1}) du }{|g_1|(1-u)^2u^{n+1-2\deg(g_1 \ell_2)}} \\ & \quad \times \sum_{g_4|(g_1,h)} |g_4| \mu\Big(\frac{g_1}{g_4}\Big) \ov{\chi_2}\Big(\frac{g_1}{g_4}\Big) \chi_2\Big(\frac{h}{g_4}\Big) \\
  %        &= \frac{q^n\ov{\chi_2}(-1)}{|\ell_2|}  \sum_{g_4|h} \frac{\chi_2(h/g_4)}{|g_4|} \frac{1}{2 \pi i} \oint_{|u|=q^{-\epsilon}} \sum_{\deg(AB)<d/2-\deg(g_4)}  \frac{u^{2\deg(g_4 A)} \ov{\chi_2}(AB) \mu(A) ~du}{|A^2B| u^{n+1-2\deg(\ell_2)}(1-u)^2}  \\
 %         &= \frac{q^n\ov{\chi_2}(-1)}{|\ell_2|}  \frac{1}{(2 \pi i)^3} \oint_{|y|=q^{-\epsilon}} \oint_{|x|=q^{-\epsilon}} \oint_{|u|=q^{-\epsilon}} \sum_{g_4|h} \frac{\chi_2(h/g_4) u^{2\deg(g_4 \ell_2)}}{|g_4|} \\
%          & \quad \times \frac{\mathcal{L} ( {x}/{q},\ov{\chi_2}) }  {\mathcal{L} ( {xyu^2}/{q^2},\ov{\chi_2}) (1-u)^2 (1-x)(1-y) (xy)^{[(d-1)/{2}]-\deg(g_4)}u^{n}}  \frac{du}{u} \frac{dx}{x} \frac{dy}{y},
%      \end{align*}
%      where we wrote $g=g_1B$ and $g_1=g_4B$.

      In the integral over $x$ we shift the contour to $|x|=q$ and encounter a pole at $x=1$. The integral over the new contour will be bounded by $q^{n-d/2-\deg(\ell_2)+\epsilon d}$. 
      
      It follows that
      \begin{align*}
          M_2 &=  \frac{q^n \ov{\chi_2}(-1) L(1,\ov{\chi_2})}{|\ell_2|} \sum_{g_4|h} \frac{\chi_2(h/g_4) }{|g_4|} \\ & \times \frac{1}{(2 \pi i)^2} 
          \oint_{|y|=q^{-\epsilon}} \oint_{|x|=q^{-\epsilon}} \frac{u^{2\deg(g_4 \ell_2)}}{\mathcal{L} ( {yu^2}/{q^2},\ov{\chi_2}) (1-u)^2 (1-y) y^{[(d-1)/2]-\deg(g_4)}u^{n}} \frac{du}{u}\frac{dy}{y} \\
          &+ O \Big( q^{n-d/2-\deg(\ell_2)+\epsilon d}\Big).
      \end{align*}
      In the integral over $y$ we shift the contour of integration to $|y|=q$, and we encounter the pole at $y=1$. The integral over the new contour will be bounded by $q^{n-d/2-\deg(\ell_2)+ \epsilon d} $,
       and we get that
       \begin{multline*}
           M_2 = \frac{q^n \ov{\chi_2}(-1) L(1,\ov{\chi_2})}{|\ell_2|} \sum_{g_4|h} \frac{\chi_2(h/g_4) }{|g_4|} \\\times  \frac{1}{2 \pi i} \oint  \frac{u^{2\deg(g_4 \ell_2)} }{\mathcal{L} ({u^2}/{q^2},\ov{\chi_2} ) (1-u)^2 u^{n}} \frac{du}{u} + O \Big( q^{n-d/2-\deg(\ell_2)+\epsilon d} \Big).
       \end{multline*}
       We further shift the contour to $|u|=q^{1/2}$ and encounter the double pole at $u=1$. The integral over the new contour will be bounded by $q^{n/2+\epsilon d}$.  Computing the residue at $u=1$ we get that 
       \begin{multline}
           M_2 = \frac{q^n \ov{\chi_2}(-1) L(1,\ov{\chi_2})}{|\ell_2|L(2,\ov{\chi_2})} \sum_{\substack{g_4|h \\ \deg(g_4) \leqs n/2-\deg(\ell_2)}} \frac{\chi_2(h/g_4) }{|g_4|} \\\times  \Big(n+1-2\deg(g_4 \ell_2) - \frac{2L'(2,\ov{\chi_2})}{(\log q) L(2,\ov{\chi_2})} \Big) +  O \Big(  q^{n-d/2-\deg(\ell_2)+\epsilon d} +q^{n/2+\epsilon d} \Big).\label{m2_1chi}
       \end{multline}
\section{Setting up the proof when \texorpdfstring{$\ell_2=1$}{ell2 is 1} and the main terms}
Here, we will consider the correlations
\begin{equation*} 
%\label{corr_d} 
\sum_{f \in \mathcal{M}_n} r_{\chi_1}(f) d(f+h).
\end{equation*}
We write
\[d(f+h) = 2 \sum_{\substack{g|f+h \\ \deg(g) \leqs [d/2]}}1 - \sum_{\substack{g | f+h \\ \deg(g) =d/2}} 1,\]
so 
\begin{align*}
    \sum_{f \in \mathcal{M}_n}r_{\chi_1}(f) d(f+h)= 2 \sum_{g \in \mathcal{M}_{\leqs [d/2]}} \sum_{\substack{f \in \mathcal{M}_n \\ f \equiv -h \pmod g}} r_{\chi_1}(f) - \sum_{g \in \mathcal{M}_{d/2}} \sum_{\substack{f \in \mathcal{M}_n \\ f \equiv -h \pmod g}} r_{\chi_1}(f).
\end{align*}
We note that the second term above vanishes if $d$ is odd. 
We then write 
\[\sum_{f \in \mathcal{M}_n} r_{\chi_1}(f) d(f+h) = 2S_1 - S_2,\] and the term $S_1$ above is the same as that given in \eqref{s1initial} when $\chi_2=1$ (after detecting the congruence condition using additive characters). We will only present the details for the  evaluation of $S_1$, since $S_2$ is very similar. 
As before, we write $S_{11}= M_{11}+E_{12}$, and $S_{12}=M_{12}+E_{12}$, where $S_{11}$ and $S_{12}$ correspond to the sums with $\ell_1|g_1$ and $\ell_1 \nmid g_1$ (in equation \eqref{s1initial}). If $\ell_1 =1$, then we only have one term, and we write $S_1=M_1+E_1$. We rewrite the error terms here as (see equations \eqref{e11} and \eqref{e12}): 
\begin{align}
   E_{11} &=\frac{q^n G(\chi_1)}{|\ell_1|}  \sum_{g \in \mathcal{M}_{\leqs [d/2]}}  \frac{1}{|g|}  \sum_{\substack{g_1|g}} \frac{\chi_1(g_1)}{|g_1|}\, \copsum_{b_1 \pmod{g_1}} e \Big( \frac{b_1h}{g_1} \Big) \nonumber\\
    & \times \sum_{\lambda \in \Fq^\times} \sum_{\eta_1,\eta_2 \in \{0,1\}} \sum_{0\leqs \mu \leqs k} b_{\mu-k}(\ov \chi_1) \sum_{f \in \mathcal{M}_k} r_{\ov \chi_1}(f) e \Big(  \frac{\lambda f \ov b_1}{g_1}\Big),
    \label{e11again}
\end{align}
where $\mu=2\deg(g_1)+\deg(\ell_1)-n-2$,
and 
\begin{align}
     E_{12} &= q^{n+2} \sum_{g \in \mathcal{M}_{\leqs [d/2]}}  \frac{1}{|g|}  \sum_{\substack{g_1|g \\ \ell_1|g_1}} \frac{1}{|g_1|} \copsum_{b_1 \pmod{g_1}} e \Big( \frac{b_1h}{g_1} \Big) \ov \chi_1(-b_1) \nonumber\\
    &\times \sum_{\lambda\in\Fq^\times}\sum_{\eta_1,\eta_2\in\{0,1\}}\sum_{0\leqs k \leqs \nu } b_{\nu-k}(\chi_1) \sum_{f \in \mathcal{M}_{k}} r_{\chi_1}(f) e \Big(  \frac{ \lambda \ov b_1 f}{g_1}\Big),\label{e12again}
\end{align}
where $\nu=2\deg(g_1)-n-2$

If $\ell_1=1$, then the error term becomes (see equation \eqref{e1}):
\begin{align}
E_1 &= q^n \sum_{g \in \mathcal{M}_{\leqs  [d/2]}}  \frac{1}{|g|}  \sum_{g_1|g}  \frac{1}{|g_1|} \copsum_{b_1 \pmod{g_1}} e \Big( \frac{b_1h}{g_1} \Big)\nonumber\\
&\times \sum_{\lambda \in \Fq^\times} \sum_{\eta_1,\eta_2 \in \{0,1\}} \sum_{0\leqs k \leqs \mu } b_{\mu-k} \sum_{f \in \mathcal{M}_k} d(f) e \Big(  \frac{\lambda \ov{b_1} f}{g_1}\Big),\label{e1again}
\end{align}
where $\mu=2\deg(g_1)-n-2$.
\begin{remark}
\label{rmk7}
    As before, we remark that $E_{11}=0$ in the following cases: 
    \begin{itemize}
   \item when  $m \leqs n$, $n$ is even, and $\deg(\ell_1)=1$;
   \item when $m \leqs n$, $n$ is odd,  and $\deg(\ell_1) <3$;
   \item when  $m=n+1$, $n$ is even, and $\deg(\ell_1)=1$.
    \end{itemize}
    Note that $E_{12}=E_1=0$ when $m < n+2$ or when $m=n+2$ and $n$ is odd.
\end{remark}

\subsection{The case \texorpdfstring{$\ell_1 \neq 1$}{ell1 is not 1}}
Similarly as in equation \eqref{m11_interm}, we have
\begin{align*}
    M_{11} &= q^n L(1,\chi_1) \sum_{g_2|h} \frac{\chi_1 (g_2)}{|g_2|} \sum_{\deg(AB) \leqs [d/2]-\deg(g_2)} \frac{ \mu(B) \chi_1 (B) }{|AB^2|} \nonumber  \\
    &= q^n L(1,\chi_1) \frac{1}{(2 \pi i)^2} \oint_{|y|=q^{-\epsilon}} \oint_{|x|=q^{-\epsilon}}  \frac{ 1}{\mathcal{L}({x}/{q^2},\chi_1)(1-xy) (1-x)(1-y)(xy)^{[d/2]}} \nonumber  \\
    & \times 
    \sum_{g_2|h} \frac{ \chi_1(g_2) (xy)^{\deg(g_2)}}{|g_2|} \frac{dx}{x} \frac{dy}{y},%\label{m11again}
\end{align*} 
In the integral over $x$, we shift the contour of integration to $|x|=q$ and we encounter two poles, at $x=1$ and $x=1/y$. The integral over the new contour will be bounded by $q^{n-d/2+\epsilon d}$. We then get that
\begin{align*}
   M_{11} &= \frac{q^n L(1,\chi_1)}{L(2,\chi_1)} \frac{1}{2 \pi i} \oint_{|y|=q^{-\epsilon}} \frac{1}{(1-y)^2 y^{[d/2]}} \sum_{g_2|h} \frac{ \chi_1(g_2) y^{\deg(g_2)}}{|g_2|} 
     \,  \frac{dy}{y}\\
     &- q^n L(1,\chi_1)\sum_{g_2|h} \frac{ \chi_1(g_2) }{|g_2|} \frac{1}{2 \pi i} \oint \frac{1}{ \mathcal{L} ( 1/{(yq^2)},\chi_1)(1-y)^2} dy + O \Big( q^{n-d/2+\epsilon d}\Big).
\end{align*}
Note that the second integral above vanishes, because the integrand is analytic inside the circle of integration $|y|=q^{-\epsilon}$. In the first integral, we shift the contour of integration to $|y|=q$ and obtain an error term of size $q^{n-d/2+\epsilon d}$. The main term will be given by the residue at $y=1$, so 
\begin{align}
\label{m11_chi1}
  M_{11} &= \frac{q^n L(1,\chi_1)}{L(2,\chi_1)}  \sum_{\substack{g_2|h \\ \deg(g_2) \leqs [d/2]}} \frac{\chi_1(g_2)}{|g_2|} \Big(  [d/2]+1 - \deg(g_2)\Big)+ O \Big( q^{n-d/2+\epsilon d}\Big).
\end{align}

The evaluation of $M_{12}(\chi_1,1)$ is similar to that of $M_{12}(\chi_1,\chi_2)$ in the previous sections. We skip some of the details, and we get that 
\begin{align*}
    M_{12} &= \frac{q^n  L(1,\ov{\chi_1})}{|\ell_1|} \sum_{g_4|h}  \frac{\chi_1(h/g_4)}{|g_4|}\frac{1}{2 \pi i} \oint_{|u|=q^{-\epsilon}}  \frac{1}{\mathcal{L}(u/q^2,\ov{\chi_1} )(1-u)^2u^{[d/2]-\deg(\ell_1) - \deg(g_4)}} \frac{du}{u}.
\end{align*}
We shift the integral to $|u|=q$  and encounter a pole at $u=1$. The integral over the new contour will be bounded by $q^{n-d/2+\epsilon d}$. 
\kommentar{We have 
\begin{align*}
   \sum_{g_4|h} & \frac{ \chi_1(h/g_4)}{|g_4|} u^{\deg(g_4)} = \frac{\chi_1(h/\ell_1^j)u^{\deg(\ell_1^j)}}{|\ell_1^j|} \prod_{P|(h/\ell_1^j)} \Big(  \sum_{i=0}^{v_P(h/\ell_1^j)} \frac{ \ov{\chi_1} (P)^i u^{i \deg(P)}}{|P|^i}\Big) \\
   &:= \frac{ \chi_1(h/\ell_1^j)}{|\ell_1^j|} u^{\deg(\ell_1^j)} \mathcal{A}(u).
\end{align*}
Then we have
\begin{align}
    M_{12} &= \frac{q^n L(1,\ov{\chi_1}) \chi_1(h/\ell_1^j)}{|\ell_1|^{1+j}L(2,\ov{\chi_1})} \Big[  ([d/2]-(j+1) \deg(\ell_1)+1) \mathcal{A}(1) + \mathcal{A}'(1) - \frac{\mathcal{A}(1) L'(2,\ov{\chi_1})}{\log q L(2,\ov{\chi_1})} \Big] \nonumber \\
    &+ O \Big(  q^{n-d/2+\epsilon d} |\ell_1 h|^{\epsilon}\Big). \label{m12_chi1}
\end{align}
In the above, $j= v_{\ell_1}(h).$ We have
\[\mathcal{A}(1) =\prod_{P|(h/\ell_1^j)} \Big(  \sum_{i=0}^{v_P(h/\ell_1^j)} \frac{ \ov{\chi_1} (P)^i} {|P|^i}\Big), \]
and
\[\mathcal{A}'(1) =\mathcal{A}(1) \sum_{P|(h/\ell_1^j)} \frac{\Big(  \sum_{i=0}^{v_P(h/\ell_1^j)} \frac{ i \deg(P) \ov{\chi_1} (P)^i }{|P|^i}\Big) }{ \Big(  \sum_{i=0}^{v_P(h/\ell_1^j)} \frac{ \ov{\chi_1} (P)^i u^{i \deg(P)}}{|P|^i}\Big)}.\]}
We have
\begin{multline}
\label{m12_chi1}
    M_{12} = \frac{q^n L(1,\ov \chi_1)}{|\ell_1|L(2,\ov \chi_1)} \sum_{\substack{g_4 | h \\ \deg(g_4) \leqs [d/2]-\deg(\ell_1)}} \frac{ \chi_1(h/g_4)}{|g_4|} \Big[[d/2]+1-\deg(g_4\ell_1) - \frac{L'(2,\ov \chi_1)}{\log q L(2,\ov \chi_1)} \Big] \\
    +  O \Big(  q^{n-d/2+\epsilon d}\Big).
\end{multline}
For  the other term \[\sum_{g \in \mathcal{M}_{d/2}} \sum_{\substack{f \in \mathcal{M}_n \\ f \equiv -h \pmod g}} r_{\chi_1}(f),\] which exists when $2\mid d$, we have  
\begin{align*}
    M_{21} & =q^n L(1,\chi_1) \sum_{g_2|h} \frac{\chi_1 (g_2)}{|g_2|} \sum_{\deg(B) \leqs  d/2-\deg(g_2)} \frac{\mu(B) \chi_1 (B) }{|B^2|} \sum_{\deg(A)=d/2-\deg(g_2)-\deg(B)}\frac{1}{|A|}.
\end{align*}    
  \kommentar{\begin{align*} 
    &= q^n L(1,\chi_1) \frac{1}{(2 \pi i)^2} \oint \oint \frac{ 1}{\mathcal{L} \Big( \frac{x}{q^2},\chi_1 \Big)(1-x)(1-y)(xy)^{d/2}} \\
    & \times 
     \prod_{P|h} \Bigg(\frac{ 1 - \Big(\frac{\chi_1 (P)}{|P|} (xy)^{\deg(P)}\Big)^{v_p(h)+1}}{ 1 - \frac{\chi_1 (P)}{|P|} (xy)^{\deg(P)}}\Bigg)
     \, \frac{dx}{x} \frac{dy}{y}\\
     &= \frac{q^n L(1,\chi_1)}{L(2,\chi_1)} \frac{1}{2 \pi i} \oint \frac{1}{(1-y) y^{d/2}} \prod_{P|h} \Bigg(\frac{ 1 - \Big(\frac{\chi_1 (P)}{|P|} y^{\deg(P)}\Big)^{v_p(h)+1}}{ 1 - \frac{\chi_1 (P)}{|P|} y^{\deg(P)}}\Bigg)
     \,  \frac{dy}{y}\\
     &+ O \Big( q^{n-d/2+\epsilon d} |h\ell_1|^{\epsilon}\Big)\\
      &= \frac{q^n L(1,\chi_1)}{L(2,\chi_1)} 
      \prod_{P|h} \sum_{i=0}^{v_P(h)} \Big( \frac{\chi_1(P)}{|P|}\Big)^i+ O \Big( q^{n-d/2+\epsilon d} |h\ell_1|^{\epsilon}\Big)
\end{align*}
and }
\kommentar{Since the inner sum over $A$ equals 1, we can write 
\begin{align}
    M_{21}   % &=q^n L(1,\chi_1) \sum_{g_2|h} \frac{\chi_1 (g_2)}{|g_2|} \sum_{\deg(B) \leqs  d/2-\deg(g_2)} \frac{ \chi_1 (B) \mu(B)}{|B^2|} \\
     &= q^n L(1,\chi_1) \frac{1}{2 \pi i} \oint \frac{ 1}{\mathcal{L} \Big( \frac{x}{q^2},\chi_1 \Big)(1-x)x^{d/2}} 
     \prod_{P|h} \Bigg(\frac{ 1 - \Big(\frac{\chi_1 (P)}{|P|} x^{\deg(P)}\Big)^{v_p(h)+1}}{ 1 - \frac{\chi_1 (P)}{|P|} x^{\deg(P)}}\Bigg)
     \, \frac{dx}{x} \nonumber \\
      &= \frac{q^n L(1,\chi_1)}{L(2,\chi_1)} 
      \prod_{P|h} \sum_{i=0}^{v_P(h)} \Big( \frac{\chi_1(P)}{|P|}\Big)^i+ O \Big( q^{n-d/2+\epsilon d} |h\ell_1|^{\epsilon}\Big).
      \label{m21_chi1}
\end{align}   }
 We get that
\begin{align}
 M_{21} = \frac{q^n L(1,\chi_1)}{L(2,\chi_1)} \sum_{\substack{g_2|h \\ \deg(g_2) \leqs d/2}} \frac{\chi_1(g_2)}{|g_2|} + O \Big( q^{n-d/2+\epsilon d} \Big).
      \label{m21_chi1}
\end{align}
Finally, we have
\begin{align}
    M_{22} &= \frac{q^n L(1,\ov{\chi_1})}{|\ell_1|} \sum_{g_4|h}  \frac{ \chi_1(h/g_4)}{|g_4|}\frac{1}{2 \pi i} \oint_{|u|=q^{-\epsilon}} \frac{du}{\mathcal{L}(u/q^2,\ov{\chi_1} )(1-u)u^{1+d/2-\deg(\ell_1) - \deg(g_4)}} \nonumber \\
    &=  \frac{q^n L(1,\ov{\chi_1})}{|\ell_1|L(2,\ov{\chi_1})} \sum_{\substack{g_4|h \\ \deg(g_4) \leqs d/2-\deg(\ell_1)}} \frac{\chi_1(h/g_4)}{|g_4|} + O \Big(q^{n-d/2+\epsilon d}\Big).
    \label{m22_chi1}
\end{align}
\subsection{The case \texorpdfstring{$\ell_1=1$}{ell1 is 1}}
Finally, we consider the correlations of $d(f)$ and $d(f+h)$.
Similarly as before (see equation \eqref{m11again1}), we have 
     \begin{multline*}
         M_{1} = \frac{q^n}{(2 \pi i)^3} \oint_{|y|=q^{-\epsilon}} \oint_{|x|=q^{-\epsilon}} \oint_{|u|=q^{-\epsilon}} \frac{1-xyu^2/{q}}{ (1-u)^2 (1-x)^2(1-y)(xy)^{[d/2]}u^{n}}\\\times  \sum_{g_2|h} \frac{(xyu^2)^{\deg(g_2)}}{|g_2| } \frac{du}{u} \frac{dx}{x} \frac{dy}{y}.
     \end{multline*}
   We note that the integral over $y$ above is precisely given by the residue of the pole at $y=1$. We then get that
   \begin{align*}
       M_{1} &= \frac{q^n}{(2 \pi i)^2} \oint_{|x|=q^{-\epsilon}} \oint_{|u|=q^{-\epsilon}} \frac{1-xu^2/{q}}{ (1-u)^2 (1-x)^2x^{[d/2]}u^{n}} \sum_{g_2|h} \frac{(xu^2)^{\deg(g_2)}}{|g_2| } \frac{du}{u} \frac{dx}{x}.
   \end{align*}
    Note that in order for the integral over $u$ not to vanish we need $\deg(g_2) \leqs [n/2]$. We similarly have that
   \begin{align*}
       M_2 &=  \1_{2|d} \frac{q^n}{(2 \pi i)^2}\oint_{|x|=q^{-\epsilon}} \oint_{|u|=q^{-\epsilon}} \frac{ 1-{xu^2}/{q}}{(1-u)^2 (1-x)x^{[d/2]}u^{n}} \sum_{g_2|h} \frac{(xu^2)^{\deg(g_2)}}{|g_2| } \frac{du}{u} \frac{dx}{x}.\nonumber 
   \end{align*}
   The integrals above are given by the residue of the poles at $u=1$ and $x=1$ (with no error terms), and computing both of the terms and putting them together, we get that the final main term $M$ is given by
   \begin{multline}
M=2q^n\sum_{\substack{g_2|h \\ \deg(g_2) \leqs [n/2]}} \frac{1}{|g_2|} \Bigg[2\deg(g_2)^2\Big(1-\frac{1}{q}\Big)\\- \deg(g_2)
\bigg(\Big(3+\frac{1}{q}\Big)+\Big(1-\frac{1}{q}\Big)(n+2[d/2]) 
\bigg)\\+ \bigg( n+1+[d/2] \Big( n\Big(1-\frac{1}{q}\Big) +
\Big(1+\frac{1}{q}\Big) \Big) \bigg) \Bigg]   \\
 - \1_{2|d}q^n \sum_{\substack{g_2|h \\ \deg(g_2) \leqs [n/2]}} \frac{1}{|g_2|} \Bigg[ (n-2\deg(g_2)) \Big( 1-\frac{1}{q}\Big)+1+\frac{1}{q}\Bigg]. \label{maindd}
\end{multline}

\section{Bounding the error terms}

   Combining Remarks \ref{rmk1}, \ref{rmk2}, \ref{rmk3}, \ref{rmk4}, \ref{rmk4.5}, \ref{rmk5}, \ref{rmk6}, and  \ref{rmk7}, we see that when $\deg(h) \leqs n$ and $\deg(\ell_1)<2$ and $2 \deg(\ell_2)+\deg(\ell_1)<4$, the various error terms given by \eqref{e11}, \eqref{e12}, \eqref{e21}, \eqref{e22}, \eqref{e2}, \eqref{e1}, \eqref{e2d}, \eqref{e11again}, \eqref{e12again}, \eqref{e1again} vanish and the only error terms in the asymptotic formulas are those in the main term computations (if any).  In particular, if $(\deg(\ell_1), \deg(\ell_2)) \in \{ (0,0), (0,1), (1,0) ,(1,1)\}$ and $\deg(h) \leqs n$, then the error terms mentioned above vanish.

   In fact, we can say a little more. Namely, we obtain exact formulas in the following scenarios as well. Let $m=n+a$ for $a\geqs 0$, where we recall that $m=\deg(h)$. Then the error terms $E_{ij}$ and $E_i$ for $i=1,2$ and $j=1,2$ vanish in the following cases:
     $$ (a,\deg(\ell_1),\deg(\ell_2)) \in 
\begin{cases}
    (0,0,0)&  \\
    (0,0,1)  &  \\
    (0,1,0) &    \\
    (0,2,0) & \mbox{ and } n\equiv 1 \pmod 2 \\
   (0,1,1)   & \mbox{ and } n \equiv 0 \pmod 2 \\
    (1,0,0)   &  \\
    (1,1,0) & \mbox{ and } n \equiv 0 \pmod 2  \\
    (1,0,1) & \mbox{ and } n \equiv 1 \pmod 2  \\
    (2,0,0) & \mbox{ and } n \equiv 1 \pmod 2.
    \end{cases}$$
      
    Now we focus on bounding the various error terms when we are not in one of the scenarios described above. We will only bound a few of them, since they are all similar. We first bound the terms $E_{11}$ and $E_1$, given in \eqref{e11} and \eqref{e1again} (they are very similar, the only difference being that in the latter case, $\chi_1=\chi_2=1$.)
We interchange the sums over $g_1$ and $g$ writing $g=g_1g_2$ and we have 
\begin{align*}
    E_{11} &= \frac{q^nG(\chi_1)}{|\ell_1|}  \sum_{g_1\in \mathcal{M}_{\leqs [d/2]}} \frac{ \chi_1 \chi_2(g_1)}{|g_1|^2} \sum_{g_2 \in \mathcal{M}_{\leqs [d/2]-\deg(g_1)}} \frac{\chi_2(g_2)}{|g_2|}\\& \times \sum_{\lambda \in \Fq^\times} \sum_{\eta_1,\eta_2 \in \{0,1\}}  \sum_{0\leqs k \leqs \mu } b_{\mu-k}(\ov \chi_1) \sum_{f \in \mathcal{M}_k} r_{\ov \chi_1}(f) S(h,\lambda f; g_1),
    \end{align*}
    where $\mu= 2\deg(g_1) + \deg(\ell_1)-n-2$ and we 
    recall that $S(a,b;c)$ is the function field Kloosterman sum given by \eqref{def:S_over_FqT}.
    We further write 
    \begin{align*}
        E_{11} &= \frac{q^nG(\chi_1)}{|\ell_1|}  \sum_{j=0}^{[d/2]} q^{-2j}  \sum_{g_2 \in \mathcal{M}_{\leqs [d/2]-j}} \frac{\chi_2(g_2)}{|g_2|} \sum_{\lambda \in \Fq^\times}\sum_{\eta_1,\eta_2 \in \{0,1\}}  \\
    & \times \sum_{k=0}^{ 2j + \deg(\ell_1)-n-2} b_{2\deg(g_1)+\deg(\ell_1)-n-2-k}(\ov \chi_1) \sum_{f \in \mathcal{M}_k} r_{\ov \chi_1}(f) \sum_{g_1 \in \mathcal{M}_j} \chi_1 \chi_2(g_1) S(h,\lambda f; g_1).
    \end{align*}
    In the case when $\chi_1 \chi_2 = 1$ (namely, when $\ell_1=\ell_2=1$) 
    we use the Linnik--Selberg theorem \eqref{linnik-selberg} and we have that
    $$\sum_{g_1 \in \mathcal{M}_j} S(h,\lambda f ; g_1) \ll q^{j+\epsilon j} |fh|^{\epsilon}.$$
      %When $\chi_1 = \ov \chi_2$ and %$\chi_1, \chi_2$ are characters %modulo $\ell$, we have 
%\begin{align*}
   %  \sum_{g_1 \in \mathcal{M}_j} &\chi_1 \chi_2(g_1) S(h,\lambda f; g_1) =    \sum_{g_1 \in \mathcal{M}_j}  S(h,\lambda f; g_1)  - \sum_{g_1 \in \mathcal{M}_{j-\deg(\ell)}} S(h,\lambda f; \ell g_1) \\
   %  & \ll q^{j+\epsilon j} |fh|^{\epsilon},
   % \end{align*}
  %  where again we used the Linnik--Selberg theorem above. {\color{red} Note that in Sardari's paper he has a divisibility condition on $g_1$ as well and a twist; he mentions that the untwisted version is true. I'm not sure if that includes the divisibility condition or not. I'll assume it does, but we should ask. }
    Trivially bounding the other sums, we get that in this case,
    \begin{equation}
    \label{e11error}
        E_{11} \ll q^{d/2+\epsilon d}.
        \end{equation}
    If we are not in the case mentioned above, then we use the Weil bound \eqref{weil} on the sum over $g_1$ and we have that
    $$ |S(h,\lambda f;g_1)| \leqs q^{\deg(g_1)/2+\deg( (f,h,g_1))/2} 2^{\omega(g_1)}.$$
    Let $A=(f,h)$. Then 
    \begin{align*}
        \sum_{g_1 \in \mathcal{M}_j} 2^{\omega(g_1)} q^{\deg(g_1)/2+\deg((A,g_1))/2} \ll q^{j/2+\epsilon j} \sum_{R|A} \sum_{\substack{g_1 \in \mathcal{M}_j \\ (g_1,A)=R}} 1 \ll q^{3j/2+\epsilon j} |A|^{\epsilon} \ll q^{3j/2+\epsilon j+\epsilon d}.
    \end{align*}
    Trivially bounding all the other sums, we get that
    \begin{equation} \label{ellerror-2}
        E_{11} \ll q^{3d/4+\deg(\ell_1)/2+\epsilon d}.
        \end{equation}
    Now we focus on bounding $E_{21}$, given in \eqref{e21} (in the case $\ell_1 \neq \ell_2)$.) We rewrite it as 
    \begin{align}
E_{21} &= \frac{q^n  \chi_1(\ell_2) G(\chi_1) G(\chi_2)}{|\ell_1 \ell_2^2|} \sum_{g_1 \in \mathcal{M}_{\leqs [(d-1)/2]}} \frac{\chi_1(g_1)}{|g_1|^2} \sum_{g_2 \in \mathcal{M}_{\leqs [(d-1)/2]-\deg(g_1)}} \frac{\ov{\chi_2}(g_2)}{|g_2|} \nonumber \\
& \times \copsum_{b_1 \pmod{g_1\ell_2}} \ov{\chi_2}(-b_1) e\Big(\frac{b_1h}{g_1\ell_2}\Big) \sum_{\lambda \in \Fq^\times} \sum_{\eta_1,\eta_2 \in \{0,1\}} \sum_{0\leqs k \leqs \mu} b_{\mu-k}(\ov{\chi_1}) \sum_{f \in \mathcal{M}_k} r_{\ov \chi_1}(f) e \Big(  \frac{\lambda \ov{b_1} f}{g_1 \ell_2}\Big), \nonumber
    \end{align} 
    where $\mu= 2 \deg(g_1 \ell_2)+\deg(\ell_1)-n-2$.
    Now we rewrite
    $$ G(\chi_2) \ov{\chi_2}(-b_1) = \sum_{a \pmod{\ell_2}} \chi_2(a) e \Big( - \frac{ab_1}{\ell_2}\Big),$$ so we have that
    \begin{align*}
        E_{21} &= \frac{q^n \chi_1(\ell_2)G(\chi_1) }{|\ell_1 \ell_2^2|} 
        \sum_{j=0}^{[(d-1)/2]} q^{-2j}\sum_{g_2 \in \mathcal{M}_{\leqs [(d-1)/2]-j}} \frac{\ov{\chi_2}(g_2)}{|g_2|} \sum_{\lambda \in \Fq^\times} \sum_{\eta_1,\eta_2 \in \{0,1\}}\\
& \times  \sum_{0\leqs k \leqs \mu} b_{\mu-k}(\ov{\chi_1}) \sum_{f \in \mathcal{M}_k} r_{\ov \chi_1}(f) \sum_{a \pmod{\ell_2}}\chi_2(a)\sum_{g_1 \in \mathcal{M}_j} \chi_1(g_1)  S(h-ag_1,\lambda f; g_1 \ell_2). \nonumber
    \end{align*}
Now we use the Weil bound \eqref{weil} on the Kloosterman sum and we have that 
\begin{align*}
|S(h-ag_1,\lambda f;g_1 \ell_2)|&\leqs q^{\deg(g_1 \ell_2)/2+\deg((h-ag_1,f,g_1\ell_2))/2} 2^{\omega(g_1 \ell_2)} \\ &\ll q^{\deg(g_1 \ell_2)/2+\deg((f,g_1\ell_2))/2} |g_1 \ell_2|^{\epsilon}. \end{align*}
It follows that introducing the sum over $f$ and trivially bounding everything else,
\begin{multline*}
  \sum_{f \in \mathcal{M}_k}  r_{\ov \chi_1}(f)\sum_{g_1 \in \mathcal{M}_j} \chi_1(g_1)  S(h-ag_1,\lambda f; g_1 \ell_2) \\  \ll |\ell_2|^{\epsilon} q^{\deg(\ell_2)/2+j/2+\epsilon j} \sum_{\substack{f \in \mathcal{M}_k \\ \ell_2|f}} \sqrt{|\ell_2|} \sum_{g_1 \in \mathcal{M}_j} q^{\deg((f/\ell_2,g_1))/2} \\
  +|\ell_2|^{\epsilon} q^{\deg(\ell_2)/2+j/2+\epsilon j} \sum_{\substack{f \in \mathcal{M}_k \\ \ell_2 \nmid f}} \sum_{g_1 \in \mathcal{M}_j} q^{\deg((f,g_1))/2} \ll q^{3j/2+\epsilon j+\deg(\ell_2)/2} |\ell_2 |^{\epsilon} q^{k(1+\epsilon)} .
\end{multline*}

We trivially bound all the remaining sums, and we get that
\begin{equation}
      \label{error}
      E_{21} \ll q^{3d/4+\epsilon d+\deg(\ell_1)/2+3\deg(\ell_2)/2}.
      \end{equation}

\section{Proof of the main theorems}
Finally, we are ready to prove Theorems \ref{th1}, \ref{dd}, \ref{th3}, \ref{th4}, \ref{th5}.

Using Remark \ref{remarkell} and equations \eqref{m11_expression}, \eqref{m2l}, \eqref{ellerror-2} finishes the proof of Theorem \ref{th1}. 

Now using equations \eqref{maindd}, \eqref{e11error} finishes the proof of Theorem \ref{dd}.

Using equations \eqref{m11_expression}, \eqref{m12_expression}, \eqref{m21_expression}, \eqref{m22_expression}, and \eqref{error} finishes the proof of Theorem \ref{th3}.

Using \eqref{m1_1chi},\eqref{m2_1chi}, and \eqref{error} finishes the proof of Theorem \ref{th4}. 

Finally, combining \eqref{m11_chi1}, \eqref{m12_chi1}, \eqref{m21_chi1}, \eqref{m22_chi1}, \eqref{error} finishes the proof of Theorem \ref{th5}.

To deduce Corollary~\ref{thm:sumsof2squares}, we simply set $\chi_1 = \chi_2 = \chi$ in Theorem~\ref{th1} and note that the quadratic character modulo $T$ is an odd character. In that case, we have $\chi(-1)G(\chi)^2G(\chi^2) = -q$. 
\small
\bibliography{divcorr}
\bibliographystyle{plain}

\end{document}